\documentclass[10pt]{article}

\usepackage{amsmath,amsfonts,amsthm,amssymb,mathrsfs}
\usepackage[T1]{fontenc}
\usepackage[utf8]{inputenc}
\usepackage{natbib} 
\usepackage[colorlinks]{hyperref}
\usepackage{xcolor}
\usepackage{fullpage}
\usepackage{mathtools}
\usepackage{authblk}
\usepackage{algorithm}
\usepackage{algorithmic}
\usepackage{thmtools}
\usepackage{thm-restate}
\usepackage{cleveref}

\DeclareMathOperator*{\argmin}{arg\,min}

\newcommand{\eps}{\ensuremath{\varepsilon}}
   \newcommand{\R}{\mathbb{R}}
\newcommand{\RR}{\mathbb{R}}

\renewcommand{\t}{^{\top}}
\newcommand{\diag}{\mathrm{Diag}}
\newcommand{\inv}{^{-1}}
\newcommand{\E}{\mathbb{E}}

\newcommand{\p}{\mathbb{P}}

\newcommand{\per}{U_{\perp}U_{\perp}\t}
\renewcommand{\P}{\mathcal{P}}
\renewcommand{\hat}{\widehat}
\newcommand{\snr}{\mathrm{SNR}}

\newtheorem{theorem}{Theorem}
\newtheorem{corollary}{Corollary}
\newtheorem{lemma}{Lemma}

\theoremstyle{definition}

\newtheorem{assumption}{Assumption}
\setcitestyle{round}
\hypersetup{%
  colorlinks=true,% hyperlinks will be black
  linkcolor=blue,
  citecolor = blue,
  urlcolor = blue%,% hyperlink borders will be red
}
\usepackage{mathtools}
\usepackage{algorithmic}
\usepackage{algorithm}

\newcommand{\tr}{\text{Tr}}

\newcommand{\utilde}{\tilde U}
\newcommand{\uhat}{\hat U}
\providecommand{\keywords}[1]
{
  \small
  \textbf{Keywords:} #1
}

\usepackage{authblk}

\usepackage{natbib} 
\bibpunct{(}{)}{;}{a}{,}{,}

\title{Entrywise Estimation of Singular Vectors of Low-Rank Matrices with Heteroskedasticity and Dependence}
\author[1]{Joshua Agterberg}
\author[1]{Zachary Lubberts}
\author[1]{Carey Priebe}
\date{\today}

\affil[1]{Department of Applied Mathematics and Statistics, Johns Hopkins University}

\begin{document}

\maketitle

\begin{abstract}
We propose an estimator for the singular vectors of high-dimensional low-rank matrices corrupted by additive subgaussian noise, where the noise matrix is allowed to have dependence within rows and heteroskedasticity between them. We prove finite-sample $\ell_{2,\infty}$ bounds and a Berry-Esseen theorem for the individual entries of the estimator, and we apply these results to high-dimensional mixture models.  Our Berry-Esseen theorem 
clearly shows the geometric relationship between the signal matrix, the covariance structure of the noise, and the distribution of the errors in the singular vector estimation task. These results are illustrated in numerical simulations. Unlike previous results of this type, which rely on assumptions of Gaussianity or independence between the entries of the additive noise, handling the dependence between entries in the proofs of these results requires careful leave-one-out analysis and conditioning arguments. Our results depend only on the signal-to-noise ratio, the sample size, and the spectral properties of the signal matrix.
\end{abstract}


\tableofcontents

\section{Introduction}
Consider the signal-plus-noise model
 \begin{align*}
     \hat M = M + E,
 \end{align*}
 where $M \in \R^{n\times d}$ is a deterministic rank $r$ signal matrix and $E$ is a mean-zero noise matrix.  We assume $M$ has the singular value decomposition 
 \begin{align*}
     M = U \Lambda V\t,
 \end{align*}
 where $U$ is an $n \times r$ matrix with orthonormal columns, $\Lambda$ is an $r\times r$ diagonal matrix whose entries are the $r$ nonzero singular values $\lambda_1 \geq \lambda_2 \geq \cdots \geq \lambda_r > 0$ of $M$, and $V$ is an $d \times r$ matrix with orthonormal columns. Letting  $E_i\t$ denote the $i$'th row of the noise matrix $E$, we suppose each noise vector $E_i$ is of the form 
 \begin{align*}
     E_i = \Sigma_{i}^{1/2} Y_i,
 \end{align*}
where $\E Y_i = 0$, $\Sigma_{i} \in \R^{d\times d}$ is a positive semidefinite matrix allowed to depend on the row $i$, and the coordinates $Y_{i\alpha}$ of the vector $Y_i$ are independent subgaussian random variables satisfying $\E Y_{i\alpha}^2 = 1$.  Define the \emph{signal to noise ratio}
 \begin{align*}
     \snr := \frac{\lambda_r}{\sigma \sqrt{rd}},
 \end{align*}
 where $\sigma^2 := \max_i \| \Sigma_i \|$ is the maximum spectral norm of the covariances of each row. 
 
 In this work we study the entrywise estimation of the matrix $U$ (or column space of $M$) in the ``quasi-asymptotic regime'' wherein $n$ and $d$ are assumed to be large but finite.  Because we allow the rows of $E$ to have different covariances, the ``vanilla SVD'' can be biased, so we propose a bias-corrected estimator for the matrix $U$ and analyze its limiting distribution when the signal to noise ratio is large enough relative to the dimension of the problem. Furthermore, we do not assume distinct singular values for $M$.  
 
 Our main contributions are the following:
 \begin{itemize}
     \item We provide a nonasymptotic Berry-Esseen Theorem (Theorem \ref{thm:b-e}) for the entries of our proposed estimator for $U$ under general assumptions on the signal matrix $M$ and the noise matrix $E$;
     \item We study the $\ell_{2,\infty}$ approximation (Theorem \ref{thm:2_infty}) of our proposed estimator to $U$ which matches previous $\ell_{2,\infty}$ bounds in the special case of independent noise.
 %    \item The estimator we consider reduces the bias associated to other related estimators for this problem, leading to an upper bound on the $\ell_{2,\infty}$ error that scales with the noise.
 \item We apply our results to the particular setting of a $K$-component subgaussian mixture model (Corollary \ref{cor:mixture}) and show that one can accurately estimate the resulting limiting covariance of the rows of our estimator within each community (Corollary \ref{cor:consistent_estimate}), allowing us to derive data-driven, asymptotically valid confidence regions.
 \end{itemize}
Since we allow for dependence within each row of the noise matrix $E$, our asymptotic results highlight the geometric relationship between the covariance structure and the singular subspaces of $M$. Furthermore, our estimator is based off the HeteroPCA algorithm proposed in \citet{zhang_heteroskedastic_2021}, so as a byproduct of our results, we also provide a more refined analysis of this algorithm, %.  Furthermore, the estimator we consider reduces the bias associated to other related estimators for this problem, 
leading to an upper bound on the $\ell_{2,\infty}$ error that scales with the noise.

The organization of the paper is as follows. In the next subsection we discuss related work, and we end this section with notation and terminology we will use throughout. In Section~\ref{sec:bckgd}, we recall the HeteroPCA algorithm of \cite{zhang_heteroskedastic_2021}, and use this to define our estimator $\hat{U}$ for $U$. We also discuss alternative approaches to the problem and some of the shortcomings of those approaches which have motivated the present work. In Section~\ref{sec:results}, we state our main theorems, namely our $\ell_{2,\infty}$ concentration and Berry-Esseen theorems. We also discuss the various assumptions of our model, and compare these to previous work. In Section~\ref{sec:applications}, we discuss the statistical implications of our results for mixture distributions. We further illustrate our results in simulations in Section~\ref{sec:simulations}. Discussion of these results and potential future work is in Section~\ref{sec:discussion}. Finally, Section~\ref{sec:proofs} contains the proof of Theorem~\ref{thm:2_infty} and a proof sketch of Theorem~\ref{thm:b-e}.
The full proof of Theorem~\ref{thm:b-e} and additional supplementary lemmas are contained in the Appendices.

\subsection{Related Work}
Spectral methods, which refer to a collection of tools and techniques informed by matrix analysis and eigendecompositions, underpin a number of methods used in high-dimensional multivariate statistics and data science, including but not limited to network analysis \citep{abbe_ell_p_2020,abbe_entrywise_2020,agterberg_nonparametric_2020,athreya_statistical_2017,cai_subspace_2020,fan_simple_2019,jin_estimating_2019,lei_unified_2019,lei_consistency_2015,mao_estimating_2020,rubin-delanchy_statistical_2020,zhang_detecting_2020}, principal components analysis,  \citep{cai_subspace_2020,cai_rate-optimal_2018,cai_optimal_2021, koltchinskii_efficient_2020,koltchinskii_normal_2017,koltchinskii_perturbation_2016,koltchinskii_asymptotics_2016,wang_asymptotics_2017,johnstone_consistency_2009,lounici_sparse_2013,lounici_high-dimensional_2014,xie_bayesian_2019,zhu_high-dimensional_2019}, and spectral clustering \citep{abbe_ell_p_2020,abbe_entrywise_2020,amini_concentration_2021,cai_subspace_2020,lei_unified_2019,loffler_optimality_2020,schiebinger_geometry_2015,srivastava_robust_2021}.  In addition, eigenvectors or related quantities can be used as a ``warm start'' for  optimization methods (\citet{chen_inference_2019,chen_spectral_2021,chi_nonconvex_2019,lu_spectral_2017,ma_implicit_2020,xie_efficient_2020,xie_euclidean_2021}), yielding provable convergence to quantities of interest provided the initialization is sufficiently close to the optimum.  The model we consider includes as a submodel the high-dimensional $K$-component mixture model.  High-Dimensional mixture models play an important role in the analysis of spectral clustering \citep{abbe_ell_p_2020,amini_concentration_2021,loffler_optimality_2020,schiebinger_geometry_2015,vempala_spectral_2004,von_luxburg_tutorial_2007} and classical multidimensional scaling \citep{ding_modified_2019,li_central_2020,little_exact_2020}.

To analyze these methods, researchers have used existing results on matrix perturbation theory such as the celebrated Davis-Kahan $\sin\Theta$ Theorem \citep{bhatia_matrix_1997,yu_useful_2014,chen_spectral_2021}, which provides a deterministic bound on the difference between the eigenvectors of a perturbed matrix and the eigenvectors of the unperturbed matrix, provided the perturbation is sufficiently small.  Unfortunately, the Davis-Kahan Theorem and classical approaches to matrix perturbation analysis may fail to provide entrywise guarantees for estimated eigenvectors, though there has been work on studying the subspace distances in the presence of random noise \citep{li_two-sample_2018,xia_normal_2019,bao_singular_2021,orourke_random_2018,ding_high_2020}.

Entrywise eigenvector analysis plays an important role in furthering the understanding of spectral methods in a number of statistical problems
\citep{abbe_entrywise_2020,abbe_ell_p_2020,cai_subspace_2020,cape_signal-plus-noise_2019,cape_two--infinity_2019,chen_spectral_2021,damle_uniform_2019,eldridge_unperturbed_2018,fan_ell_infty_2018,lei_unified_2019,luo_schatten-q_2020,mao_estimating_2020,rohe_vintage_2020,xia_statistical_2020,xie_bayesian_2019,zhong_near-optimal_2018,zhu_high-dimensional_2019}.
A number of works have studied entrywise eigenvector or singular vector analysis when the noise matrix $E$ consists of independent entries  \citep{abbe_entrywise_2020,abbe_ell_p_2020,chen_spectral_2021,cape_signal-plus-noise_2019,cai_subspace_2020,lei_unified_2019},  and some authors have also studied the estimation of linear forms of eigenvectors
\citep{chen_asymmetry_2021,cheng_tackling_2020,fan_asymptotic_2020,koltchinskii_perturbation_2016,koltchinskii_normal_2017,koltchinskii_perturbation_2016,koltchinskii_asymptotics_1998,koltchinskii_efficient_2020,li_minimax_2021}.  In the present work, we extend existing results on entrywise analysis by allowing for dependence in the noise matrix. We address this more general setting using a combination of matrix series expansions  \citep{cape_signal-plus-noise_2019,chen_asymmetry_2021,eldridge_unperturbed_2018,xia_confidence_2019,xia_normal_2019,xia_statistical_2020,xie_bayesian_2019}, leave-one-out analysis \citep{abbe_entrywise_2020,abbe_ell_p_2020,chen_spectral_2021,lei_unified_2019}, and careful conditioning arguments. % (see Lemma \ref{lem:2infty_general}). 

While entrywise statistical guarantees refine classical deterministic perturbation techniques, they do not necessarily allow for studying the distributional properties of the eigenvectors.  Several works have studied the asymptotic distributions of individual eigenvectors \citep{cape_signal-plus-noise_2019,fan_asymptotic_2020} or $\sin\Theta$ distances \citep{bao_singular_2021,ding_high_2020,li_two-sample_2018,xia_normal_2019}, but there are very few finite-sample results on the distribution of the individual entries of singular vectors rather than eigenvectors, and existing results often depend on independence of the noise. We explicitly characterize the distribution of the individual entries of the estimated singular vectors, which showcases the effect that the geometric relationship of the covariance structure of the rows of $E$ and the spectral structure of the signal matrix $M$ has on this distribution.  

Finally, the bias of the singular value decomposition in the presence of heteroskedastic noise has been addressed in a number of works \citep{cai_subspace_2020,abbe_ell_p_2020,florescu_spectral_2016,koltchinskii_random_2000,leeb_optimal_2021,lei_bias-adjusted_2021,lounici_high-dimensional_2014}.  A common method for addressing this is the diagonal deletion algorithm, for which an entrywise analysis is carried out in \citet{cai_subspace_2020} for several different statistical problems.  However, as identified in \citet{zhang_heteroskedastic_2021}, in some situations diagonal deletion incurs unnecessary additional error, leading them to propose the HeteroPCA algorithm which we further study here.  We include detailed comparisons to existing work in Section \ref{sec:results}.

\subsection{Notation}
We use capital letters for both matrices and vectors, where the distinction will be clear from context, except for the letter $C$, which we use to denote constants.  For a matrix $M$, we write $M_{ij}$ as its $i,j$ entry, $M_{\cdot j}$ for its $j$'th column and $M_{j \cdot}$ for its $j$'th row. The symbol $e_i$ represents the standard basis vector in the appropriate dimension.  We use $\|\cdot \|$ to denote the spectral norm for matrices and the Euclidean norm for vectors, and $\|\cdot\|_F$ as the Frobenius norm for matrices.  We write the $\ell_{2,\infty}$ norm of a matrix  as $\|U\|_{2,\infty} = \max_i \|U_{i\cdot}\|$ which is the maximum Euclidean row norm.  We consider the set of orthogonal $r \times r $ matrices as $\mathbb{O}(r)$, and we exclusively use the letter $\mathcal{O}$ to mean an element of $\mathbb{O}(r)$.  We write $\langle \cdot, \cdot \rangle$ as the standard Euclidean inner product.

For a random variable $X$ taking values in $\R$, we let $\|X\|_{\psi_2}$ denote its $\psi_2$ (subgaussian) Orlicz norm; that is, 
\begin{align*}
    \| X \|_{\psi_2} := \sup_{z \in \R} \{ \E \exp( X^2/z) \leq 2 \},
\end{align*}
 and for a random variable $Y$ taking values in $\R^d$, we let $\|Y\|_{\psi_2}$ denote $\sup \| \langle Y, u \rangle \|_{\psi_2}$, where the supremum is over all unit vectors $u$.  Similarly we denote $\|\cdot \|_{\psi_1}$ as the $\psi_1$ (subexponential) Orlicz norm.  For more details on relationships between these norms, see \citet{vershynin_high-dimensional_2018}.
 
 For two equal-sized matrices $U_1$ and $U_2$ with orthonormal columns, we denote the $\sin\Theta$ distance between them as 
 \begin{align*}
     \| \sin\Theta(U_1, U_2) \| := \| U_1 U_1\t - U_2 U_2\t \|.
 \end{align*}
 For details on the $\sin\Theta$ distance between subspaces, see \citet{bhatia_matrix_1997}, \citet{chen_spectral_2021}, \citet{cape_two--infinity_2019}, or \citet{cape_orthogonal_2020}.   For a square matrix $M$, we write $\Gamma(M)$ to be the hollowed version of $M$; that is, $\Gamma(M)_{ij} = M_{ij}$ for $i\neq j$ and $\Gamma(M)_{ii} = 0$.  We write $G(M) := M - \Gamma(M)$.  

Occasionally we will write $f(n,d) \lesssim g(n,d)$ if there exists a constant $C$ sufficiently large such that $f(n,d) \leq C g(n,d)$ for sufficiently large $n$ and $d$.  We also write $f(n,d) \ll g(n,d)$ if $f(n,d)/g(n,d)$ tends to zero as $n$ and $d$ tend to infinity. We write $f(n,d) = O(g(n,d))$ if $f(n,d) \lesssim g(n,d)$.  For two numbers $a$ and $b$, we write $a \vee b$ to denote the maximum of $a$ and $b$.  Finally we write $f(n,d) \asymp g(n,d)$ if $f(n,d) = O(g(n,d))$ and $g(n,d) = O(f(n,d))$.

\section{Background and Methodology} \label{sec:bckgd}

We consider a rectangular matrix $M\in\RR^{n\times d}$ in the ``high-dimensional regime'' wherein $n$ and $d$ are large and comparable, though we are interested in the setting that $d$ is larger than $n$.  We observe $\hat{M}=M+E$ where the additive error matrix $E$ has rows $E_i\t$ which are independent with covariance matrices $\Sigma_i\in\RR^{d\times d}$ that are allowed to vary between rows. We write the singular value decomposition of $M$ as $U\Lambda V\t$, where $U\in \RR^{n\times r},V\in \RR^{d\times r}$ are matrices with orthonormal columns and $\Lambda\in \RR^{r\times r}$ is diagonal with nonincreasing positive diagonal entries $\lambda_1 \geq \lambda_2 \geq \cdots \geq \lambda_r > 0$ that are not necessarily distinct. We note that this decomposition of $M$ is not unique, so we fix some choice of $U$, which necessarily fixes a choice of $V$ also. Our results account for the nonuniqueness of $U$ by aligning the estimator $\hat{U}$ to $U$ through right-multiplication by an $r \times r$ orthogonal matrix. For more details on this alignment procedure, see, for example, \citet{chen_spectral_2021} or \citet{cape_orthogonal_2020}.

 \begin{algorithm}
\begin{algorithmic}[1]
\caption{HeteroPCA (Algorithm 1 of \citet{zhang_heteroskedastic_2021})}
\label{alg1}
\REQUIRE Input matrix $A + Z$, rank $r$, and maximum number of iterations $T$
\STATE Let $N_0 := \Gamma(A + Z)$; $T = 0$
\REPEAT
\STATE Take SVD of $N_T := \sum_{i} \lambda_i^{(T)} U_{\cdot i}^{(T)} (V_{\cdot i}^{(T)})\t$
\STATE Set $\tilde N_T := \sum_{i \leq r} \lambda_i^{(T)} U_{\cdot i}^{(T)} (V_{\cdot i}^{(T)})\t$ the best rank $r$ approximation of $N_T$
\STATE Set $N_{T+1} := G(\tilde N_T) + \Gamma(N_T)$
\STATE $T = T +1$
\UNTIL convergence or maximum number of iterations reached
\RETURN $\hat U := U^{(T)}$
\end{algorithmic}
\end{algorithm}

Since $U$ may equivalently be understood to be the matrix of eigenvectors of $A=MM\t$ corresponding to its $r$ nonzero eigenvalues, it is natural to consider $A$ and its noisy counterpart $A+Z,$ where $Z=ME\t+EM\t+EE\t$, sometimes referred to as the ``sample Gram matrix'' in the literature (e.g. \citep{cai_subspace_2020}). The matrix $\E[Z]$ is diagonal with diagonal entries $\tr(\Sigma_i)$. When $d$ is large and the rows of $E$ are heteroskedastic, this means that the eigenvectors of $A+\E[Z]$ may not well-approximate those of $A$.

Authors have suggested hollowing the matrix $A+Z$ as a method to correct this bias, which amounts to using the eigenvectors of $\Gamma(A+Z)$ as the estimator for those of $A$. An analysis of this approach is given in \citet{cai_subspace_2020} and \citet{abbe_ell_p_2020}, though this is not the primary focus of the latter. Unfortunately, while the eigenvectors of $\Gamma(A+Z)$ may be closer to the eigenvectors of $A$  than those of $A+Z$, they still incur a nontrivial, deterministic bias owing to the loss of information along the diagonal of $A$. In \citet{zhang_heteroskedastic_2021}, the authors provide an example where the eigenvectors of $\Gamma(A+Z)$ do not yield a consistent estimator for $U$ in the regime that $n$ and $d$ tend to infinity with $d\asymp n$. This motivates their alternative approach to correcting the bias, the HeteroPCA algorithm, which we review in Algorithm~\ref{alg1}. In essence, the algorithm proceeds by iteratively re-scaling the diagonals, attempting to fit the off-diagonal entries of $A+Z$ while maintaining the low rank $r$. We refer to the output of this algorithm after sufficiently many iterations as $\hat{A}$ %\textcolor{red}{
(see Theorem \ref{thm:b-e} for a quantitative condition on the number of iterations),
%}, 
and prove in Lemma~\ref{lem:deterministic_spectral} that it well-approximates the ``idealized'' perturbation of $A$ given by $\tilde{A}=A+\Gamma(Z)$ in spectral norm. This latter matrix has mean $A$, meaning that the bias along the diagonal has been accounted for. Removing this bias is also important for our distributional results, since otherwise one must center around the eigenvectors of $\Gamma(A)$ rather than those of $A$, but our interest is in the latter quantity.

The leading eigenvectors $\hat{U}$ of this matrix $\hat{A}$ serve as our estimator for $U$, and our concentration and distributional results demonstrate the quality of this estimator. This builds on the work in \cite{zhang_heteroskedastic_2021}, where they prove a bound on the $\sin\Theta$ distance between $\hat{U}$ and $U$, showing that $\hat{U}$ is a consistent estimator for $U$. While a bound on the discrepancy between $\hat{U}$ and $U$ in this metric is a first step towards the entrywise concentration results we obtain here, our results require much tighter control of the error between them. In addition, measuring the error in this way does not lend itself to the distributional results we consider in Theorem~\ref{thm:b-e}.

In the case that $M$ is well-conditioned and has highly incoherent singular vectors $U$ and $V$, the bias associated with diagonal deletion is not too severe, and the difference in performance between the eigenvectors of $\hat{A}$ and $\Gamma(A+Z)$ may not be significant, especially when $n$ and $d$ are large. On the other hand, for moderate $n$ and $d$, or for moderate levels of incoherence in $U$ and $V$, the bias incurred by diagonal deletion is highly significant. We consider an example of this in Figure~\ref{fig:illustration}, in the setting of estimating memberships in a Gaussian mixture model, under heteroskedasticity and dependence. The noise level in this problem is relatively small, which suggests that consistent estimation should be possible for the right estimator, however the moderate incoherence causes a severe breakdown in the performance of the estimate obtained by diagonal deletion, while the HeteroPCA algorithm continues to perform well.

\begin{figure}
    \centering
    \includegraphics[width=0.75\textwidth]{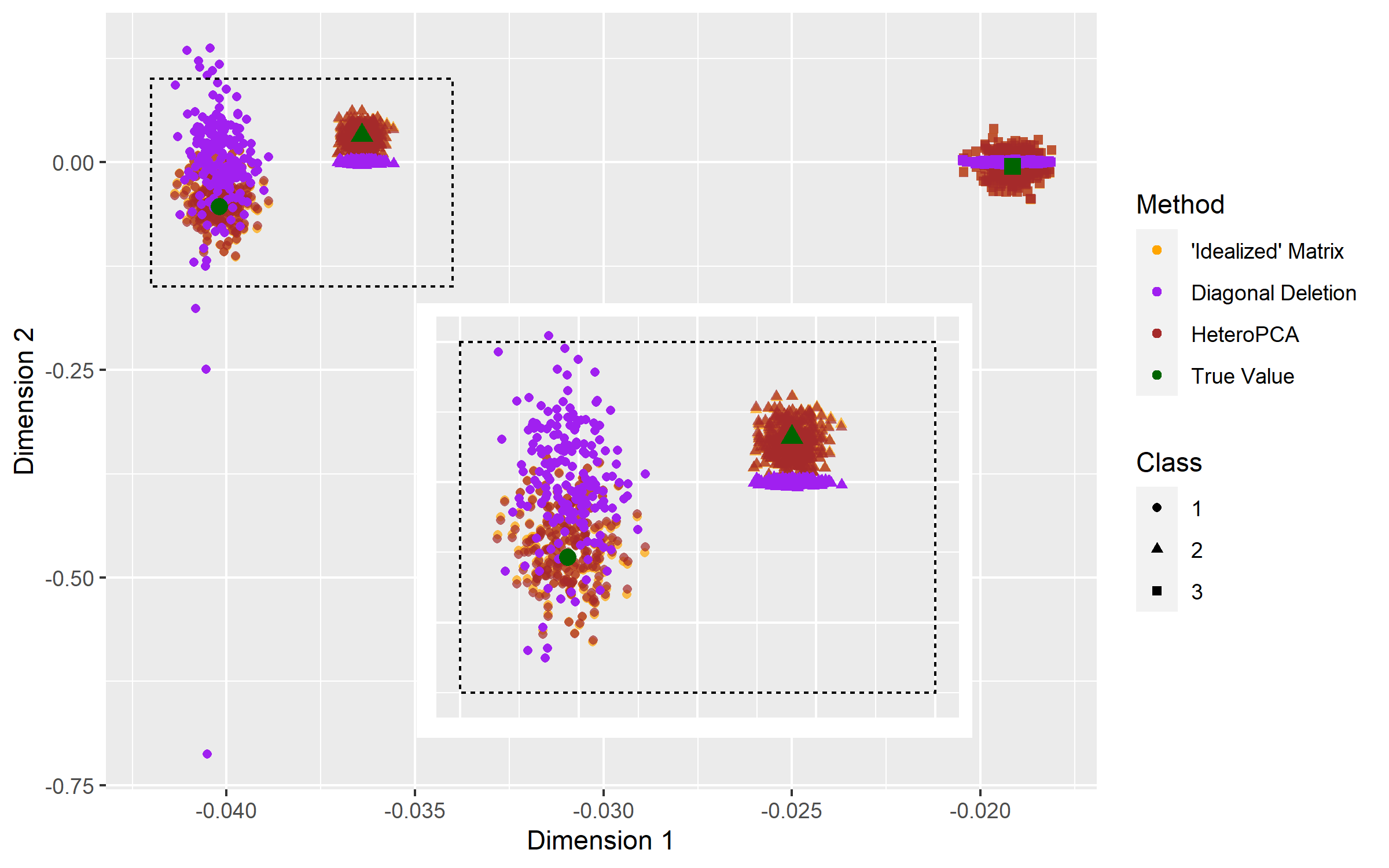}
    \caption{Comparison of the estimators for $U$ given by the eigenvectors of $\Gamma(A+Z)$ (diagonal deletion) and those of $\hat{A}$ (the output of the HeteroPCA algorithm). For convenience, we also plot the eigenvectors of the ``idealized'' matrix $A + \Gamma(Z)$ that the HeteroPCA algorithm is approximating as well as a zoomed-in reference for classes one and two.   For details on the experimental setup, see Section \ref{sec:simulations}.}
    \label{fig:illustration}
\end{figure}

\section{Main Results} \label{sec:results}

Before presenting our main results, we discuss the various assumptions required, and their role in the analysis. Our first assumption concerns the noise matrix $E$.

\begin{assumption}[Noise] \label{assumption:noise}
The noise matrix $E$ has rows $E_i\t$ that can be written in the form $E_i = \Sigma_i^{1/2} Y_i$, where each $Y_i$ is a vector of independent mean-zero subgaussian random variables with unit variance and $\psi_2$ norm uniformly bounded by 1, and $\Sigma_i$ is positive semidefinite. 
\end{assumption}

The following assumption ensures that there is sufficient signal to consistently identify the eigenvectors $U$. We let $\kappa := \lambda_1/\lambda_r$ denote the condition number of $M$.

\begin{assumption}[Enough Signal]\label{assumption:eigengap}
The signal-to-noise ratio satisfies $\snr \geq C_{SNR} \kappa \sqrt{\log(n\vee d)}$, for a sufficiently large constant $C_{SNR}$.
\end{assumption}

We remark that in the case of independent noise, $\snr \to \infty$ is required in order for consistency (e.g. \citet{cai_rate-optimal_2018,zhang_heteroskedastic_2021,xia_normal_2019}).  %In the random matrix theory literature, the setting $(nd)^{1/4} \lesssim \lambda_r \lesssim \sqrt{n\vee d}$ is often studied  (e.g. \citet{bao_singular_2021}), but we are interested in the regime where consistency is possible (occasionally referred to as the ``consistency regime'' \citep{xia_normal_2019}). % It is not clear whether $\snr \gtrsim \kappa \sqrt{\log(n\vee d)}$ is required for asymptotic normality or if it is just a technical artifact of our proof technique.  

The following assumption ensures that we have sufficiently many samples for our concentration results to hold.

\begin{assumption}[High Dimensional Regime, Low Rank]\label{assumption:dimension}
There exists a constant $c_1$ and a sufficiently small constant $c_2$ such that such that $d\geq c_1 n$, and $r \leq c_2 n$.  In addition, $\log(d) \leq n$.  
\end{assumption}

The next assumption concerns the incoherence of the matrix $M$, which measures the ``spikiness'' of the matrix.  We say $M$ is $\mu_0$-incoherent if
\begin{align*}
    \max\bigg\{ \| U\|_{2,\infty} \sqrt{\frac{n}{r}}, \|V\|_{2,\infty} \sqrt{ \frac{d}{r} }\bigg\} \leq \mu_0.
\end{align*}
When $\mu_0 = O(1)$, then the entries of $U$ and $V$ are spread, which corresponds to $M$ being fully incoherent. For more details, see \citep{chen_spectral_2021,chi_nonconvex_2019}.

\begin{assumption}[Incoherence]\label{assumption:incoherence}
%\textcolor{red}{The matrices $U$ and $V$ of singular vectors of $M$ satisfy $\|U\|_{2,\infty} \leq \mu_0 \sqrt{\frac{r}{n}}$ and $\|V\|_{2,\infty} \leq \mu_0 \sqrt{\frac{r}{d}}$, where $\mu_0$ satisfies $\kappa^2 \mu_0 \leq \sqrt{n}$.}
The matrix  $U$ of left singular vectors of $M$ satisfies $\|U\|_{2,\infty} \leq \mu_0 \sqrt{\frac{r}{n}}$, where $\mu_0$ satisfies $\kappa^2 \mu_0 \leq \sqrt{n}$.  In addition, there exists a constant $C_I$ sufficiently large such that $\|V\|_{2,\infty} \leq C_I \|U\|_{2,\infty}$. 
\end{assumption}

%\textcolor{red}{
We note that in much of the literature one assumes that both $V$ and $U$ are both $\mu_0$-incoherent, whereas Assumption \ref{assumption:incoherence} is slightly stronger as we assume that $\|V\|_{2,\infty} \leq C_I \|U\|_{2,\infty}$.   As we assume that $d \gtrsim n$, this assumption is not stringent, but rather we make this assumption for convenience as this results in a simple statement of Theorem \ref{thm:2_infty} in terms of $\|U\|_{2,\infty}$.    In order to apply our results to mixture distributions (see Section \ref{sec:applications}), we need that $\|\uhat \mathcal{O}_* - U\|_{2,\infty} \ll \|U\|_{2,\infty}$ in order to guarantee sufficient cluster separation. Consequently, assuming both $U$ and $V$ are $\mu_0$-incoherent is not quite sufficient for these purposes; instead we must additionally assume that $\|V\|_{2,\infty} \leq C_I \|U\|_{2,\infty}$. % }

Our final assumption concerns the relationship between the covariance of each row and the singular subspace of $M$. It ensures that the covariance matrices $\Sigma_i$ are not too ill-conditioned on the signal subspace of $M$.

\begin{assumption}[Covariance Condition Number]\label{assumption:be_assumptions} There exists a constant $\kappa_{\sigma}$ such that
\begin{align*}
   0<  \frac{1}{\kappa_{\sigma}} \leq \min_{i,j} \frac{\sigma_i}{\|\Sigma_i^{1/2} V_{\cdot j}\|} \leq \max_{i,j} \frac{\sigma}{\|\Sigma_i^{1/2}V_{\cdot j}\|} \leq \kappa_{\sigma} < \infty.
\end{align*}
\end{assumption}
 Informally, Assumption \ref{assumption:be_assumptions} requires that $\Sigma_i$ does not act ``adversarially'' along $V$ in the sense that the action of $\Sigma_i$ along the subspace $V$ is well-behaved.  Note that if $\Sigma_i =\sigma^2 I_d$ for all $i$, then this condition is automatically satisfied with $\kappa_{\sigma} = 1$.  %\textcolor{red}{
 Our main result (see the forthcoming Theorem \ref{thm:b-e}) is stated in terms of a fixed $i$ and $j$.   While we make the simplifying assumption that $\kappa_{\sigma}$ is uniformly bounded in both $i$ and $j$, we note that if instead $\kappa_{\sigma}$ is allowed to depend on $i$ and $j$ and satisfies 
 \begin{align*}
 0 < \frac{1}{\kappa_{\sigma}} \leq \min_{k} \frac{\sigma_k}{\|\Sigma_i^{1/2} V_{\cdot j}\|} \leq \frac{\sigma}{\|\Sigma_i^{1/2} V_{\cdot j}\|} \leq \kappa_{\sigma} < \infty,
 \end{align*}
 then our asymptotic normality results continue to hold with the proviso that $\kappa_{\sigma}$ depends on both $i$ and $j$ (with appropriate modifications in the setting of Corollary \ref{cor:mixture} and \ref{cor:consistent_estimate}).%}

We are now ready to present our main result concerning the distribution of the entrywise difference between $\hat U$ and $U$.  %\textcolor{red}{Our results hold for any fixed $i$ and $j$ chosen in advance, but nor unifo}
\begin{theorem}\label{thm:b-e}
Define
\begin{align*}
    \sigma^2_{ij} := \|\Sigma_i^{1/2} V_{\cdot j}\|^2 \lambda_j^{-2}.
\end{align*}
Suppose Assumptions \ref{assumption:noise}, \ref{assumption:eigengap}, \ref{assumption:dimension}, \ref{assumption:incoherence} and \ref{assumption:be_assumptions} hold.  %\textcolor{red}{
Let $\hat U$ be the output of the HeteroPCA algorithm after $T = \Theta\left( \frac{\lambda_r^2}{\| U\|_{2,\infty} \| \Gamma(Z)\|} \right)$ iterations.%}
Then there exist absolute constants $C_1$, $C_2$ and $C_3$ and an orthogonal matrix $\mathcal{O}_*$ such that
 \begin{align*}
       \sup_{x \in \R} \bigg| \p\bigg( \frac{1}{\sigma_{ij}} e_i\t \left( \hat U\mathcal{O}_* - U \right) e_j \leq x \bigg) - \Phi(x) \bigg| &\leq C_1 \frac{ \|\Sigma_i^{1/2} V_{\cdot j}\|_3^3}{\|\Sigma_i^{1/2}V_{\cdot j}\|^3} +  C_2 \kappa^3 \kappa_{\sigma} \mu_0 \frac{r \log(n\vee d)}{\snr} \\
       &\qquad + C_3 \kappa^2 \kappa_{\sigma} \mu_0 \sqrt{\frac{r}{n}} \bigg(  \sqrt{\log(n\vee d)} + \mu_0 \kappa^2 \sqrt{r} \bigg).
 \end{align*}
 \end{theorem}

 %\textcolor{red}{
 One should interpret Theorem \ref{thm:b-e} as stating that the entries of $\hat U$ are approximately Gaussian about their corresponding population counterparts modulo the nonidentifiability in the singular subspace stemming from the repeated singular values.  One can also generalize our analysis for the \emph{rows} of $\hat U \mathcal{O}_*$ to obtain the joint distribution; see Corollary \ref{cor:mixture} for an application of this for fixed $r$. %} %If instead $r$ is allowed to grow with $n$ and $d$, then similar asymptotic theory can be derived for the rows if one instead uses a more general multidimensional Berry-Esseen Theorem.  However, we elected   }
 
Suppose $\kappa,\kappa_{\sigma},\mu_0 = O(1)$.  Then we see that asymptotic normality holds as long as 
\begin{align*}
    \max\bigg\{ \frac{ \|\Sigma_i^{1/2} V_{\cdot j}\|_3^3}{\|\Sigma_i^{1/2}V_{\cdot j}\|^3} ,  \frac{r \log(n\vee d)}{\snr} ,\frac{r \log(n\vee d)}{\sqrt{n}}\bigg\} \to 0.
\end{align*}
 When $\Sigma_i$ is the identity and $V_{\cdot j} = \frac{1}{\sqrt{d}} \mathbf{1}$, then the first term is exactly equal to $\frac{1}{\sqrt{d}}$.  Moreover, if $\snr  \geq \sqrt{n}$, then we obtain the classical parametric rate up to logarithmic factors.  If instead $\mu_0 \log(n\vee d) \ll \min( \sqrt{n}, \snr )$, then asymptotic normality still holds.  Therefore, in the ``high-signal'' regime with $\snr \gg \sqrt{n}$, asymptotic normality holds as long as $\mu_0 \log(n\vee d) = o( \sqrt{n} )$.

We remark that the additional logarithmic factors stem from our $\ell_{2,\infty}$ result in Theorem \ref{thm:2_infty} below.  It may be possible that these logarithmic factors can be eliminated with more refined analysis, but we leave this for future work, since our primary focus is on studying asymptotic normality in the presence of dependence.    Furthermore, our results allow the covariances to be (strictly) positive semidefinite as long as the vector $V_{\cdot j}$ is not too close to the null space of the matrix $\Sigma_i$.  A simple example is if $\Sigma_i \propto V V\t$, then asymptotic normality still holds.

 %\textcolor{red}{
 Note that Theorem \ref{thm:b-e} (and Assumption \ref{assumption:be_assumptions}) depends on the fact that $\sigma_{ij} \neq 0$ (and, moreover, does not shrink to zero relative to the overall noise $\sigma$). If instead $\sigma_{ij} = 0$, then one must consider the higher-order asymptotics, in which case the dominant term contributing to the asymptotic normality becomes a higher order noise term, resulting in a different scaling for the asymptotic normality.  While it may be possible to obtain asymptotic normality in this setting, Theorem \ref{thm:2_infty} still holds regardless, thereby yielding strong concentration for $\uhat$.  
 %}

The following corollary specializes Theorem~\ref{thm:b-e} to the case of scalar matrices $\Sigma_i$. We note that in this case, the variances of the entries of the $j$'th singular vector estimate are proportional to the inverse of the $j$'th singular value of $MM\t$.  Furthermore, the leading term in Theorem \ref{thm:b-e} simplifies to be $\| V\|_{2,\infty} \leq %\textcolor{red}{
C_I \|U\|_{2,\infty} \lesssim \mu_0 \sqrt{\frac{r}{n}}, $%} \mu_0 \sqrt{\frac{r}{d}}$, 
which is smaller than the rightmost term by  %\textcolor{red}{
Assumption \ref{assumption:incoherence}.% }% \ref{assumption:dimension}. 
In particular, this result shows that the variance of the $i,j$ entry of $\hat U$ is asymptotically equal to to $\sigma_i^2/\lambda_j^2$, which reflects the fact that the variance increases deeper into the spectrum of $M$.

\begin{corollary}
\label{cor:b-e}
Assume the setting of Theorem~\ref{thm:b-e}, with $\Sigma_i=\sigma_i^2 I_d$ for all $i$. Then 
 \begin{align*}
       \sup_{x \in \R} \bigg| \p\bigg( \frac{\lambda_j}{\sigma_{i}} e_i\t \left( \hat U\mathcal{O}_* - U \right) e_j \leq x \bigg) - \Phi(x) \bigg| &\leq C_1' \kappa^3 \kappa_{\sigma} \mu_0 \frac{r \log(n\vee d)}{\snr}\\& + C_2' \kappa^2 \kappa_{\sigma} \mu_0 \sqrt{\frac{r}{n}} \bigg(  \sqrt{\log(n\vee d)} + \mu_0 \kappa^2 \sqrt{r} \bigg).
 \end{align*}
\end{corollary}

Our second theorem is an $\ell_{2,\infty}$ concentration result for the matrix $\uhat$ as an estimator for the matrix $U$. 
%\textcolor{red}{ 
Theorem \ref{thm:2_infty} is a consequence of a deterministic bound concerning the HeteroPCA algorithm (see Theorem \ref{thm:2_infty_deterministic}) and a bound concerning the idealized (random) perturbation $A + \Gamma(Z)$ (see Theorem \ref{thm:2_infty_random}).
As part of our proof of Theorem \ref{thm:2_infty}, we prove additional tight concentration for several residual terms that we rely on in the proof of Theorem \ref{thm:b-e}.
%}

\begin{theorem}\label{thm:2_infty}  Suppose Assumptions \ref{assumption:noise},  \ref{assumption:eigengap}, \ref{assumption:dimension} and \ref{assumption:incoherence}   hold.   %\textcolor{red}{
Let $\hat U$ be the output of the HeteroPCA algorithm after $T = \Theta\left( \frac{\lambda_r^2}{\| U\|_{2,\infty} \| \Gamma(Z)\|} \right)$ iterations. %}
Then there exists a universal constant $C > 0$ such that with probability at least $1 - 2(n\vee d)^{-4}$
\begin{align*}
    \inf_{\mathcal{O} \in \mathbb{O}(r)} \| \hat U - U\mathcal{O}\|_{2,\infty} \leq C \bigg(  \frac{\sqrt{r  n d} \log (n\vee d) \sigma^{2}}{\lambda_r^2} + \frac{\sqrt{rn \log (n \vee d)} \kappa \sigma }{\lambda_r} \bigg) \|U\|_{2,\infty}.
\end{align*}
\end{theorem}

Suppose $\kappa,\kappa_{\sigma}, \mu_0, r = O(1)$.  The upper bound in Theorem \ref{thm:2_infty} holds as long as $\snr \gtrsim \sqrt{\log(n\vee d)}$, %\textcolor{red}{(in fact, it holds under slightly weaker conditions as well, but this is not the focus of this work)}, 
whereas the asymptotic normality in  Theorem \ref{thm:b-e} holds when $\snr \gg \log(n\vee d)$. It is possible that with additional work the asymptotic normality in Theorem \ref{thm:b-e} holds in the regime $\sqrt{\log(n\vee d)} \lesssim \snr \lesssim \log(n\vee d)$, but this is  not the focus of the present paper.

\subsection{Comparison to Prior Work}

Our Theorem~\ref{thm:b-e} is the first distributional result for the entries of the singular vector estimator in the setting of dependent, heteroskedastic noise. In \citet{xia_normal_2019}, the author derives Berry-Esseen Theorems for the $\sin\Theta$ distance between the singular vectors under the assumption that the noise consists of independent Gaussian random variables with unit variance.  See also \citet{bao_singular_2021} for a similar result in slightly different regime.   In \citet{xia_statistical_2020}, the authors use the techniques in \citet{xia_normal_2019} to develop Berry-Esseen Theorems for linear forms of the matrix $M$. They also develop $\ell_{2,\infty}$ bounds (see their Theorem 4) en route to their main results that are similar to the bounds we obtain.  While our approach and that of \citet{xia_normal_2019} and \citet{xia_statistical_2020} share a common core, our main results require additional technical considerations due to the heteroskedasticity and dependence. We also include a deterministic analysis of the HeteroPCA algorithm, which is not needed in \citet{xia_statistical_2020} as the entries all have the same variance. 

%\textcolor{red}{
The works  \citet{koltchinskii_asymptotics_2016,koltchinskii_efficient_2020} study estimating general linear forms of the eigenvectors of a sample covariance matrix.  More specifically, Theorem 7 of \citet{koltchinskii_asymptotics_2016} derive the asymptotic normality of the linear form
\begin{align*}
    \sqrt{n} \left\langle \hat U_{\cdot i} - \sqrt{1 + b(n)} U_{\cdot i}, a \right\rangle,
\end{align*}
where $b(n)$ is a bias term, $a$ is a unit vector, and $\hat U_{\cdot i}$ is the $i$'th estimated eigenvector of the sample covariance matrix.  If $a = e_i$, then our results are similar in spirit to those of \citet{koltchinskii_asymptotics_2016}.  However, there are a few key differences: 
\begin{itemize}
    \item \citet{koltchinskii_asymptotics_2016} study covariance estimation, whereas we study the signal-plus-noise model, where the signal is assumed to be deterministic.  Viewing the matrix $X$ as the matrix whose rows are the observations, in effect \citet{koltchinskii_asymptotics_2016} study the right singular subspace of $X$ assuming that the vectors are mean-zero.  However, our analysis allows for dependence within \emph{rows}, whereas \citet{koltchinskii_asymptotics_2016} study iid observations of sample vectors, which corresponds to dependence within \emph{columns}.
    \item The results of \citet{koltchinskii_asymptotics_2016} rely heavily on the fact that the corresponding eigenvalue is simple, whereas our results allow for repeated singular values.  Consequently, our results only hold up an orthogonal transformation  $\mathcal{O}_*$ that accounts for the nonidentifiability of the associated singular spaces, whereas \citet{koltchinskii_asymptotics_2016} are able to directly analyze the corresponding empirical eigenvector up to a global sign flip.
    \item \citet{koltchinskii_asymptotics_2016} obtain asymptotic normality for a general linear form of the eigenvectors, whereas we obtain asymptotic normality for only the individual entries.  However,  \citet{koltchinskii_asymptotics_2016} consider iid \emph{Gaussian} random variables, whereas we allow general subgaussian tail conditions.  Consequently, \citet{koltchinskii_asymptotics_2016} are able to make use of powerful techniques tailored specifically to Gaussian random variables, whereas our analysis uses a combination of leave-one-out arguments and conditioning.
\end{itemize}
Therefore, the results of \citet{koltchinskii_asymptotics_2016}, while closely related, are not immediately comparable to our results.  
%}

In the context that $M$ is a symmetric, square matrix and $E$ is a matrix of independent noise along the upper triangle, \citet{cape_signal-plus-noise_2019} developed asymptotic normality results for the rows of the leading eigenvectors, which is similar in spirit to our main result in Theorem \ref{thm:b-e}.  Similarly, \citet{fan_asymptotic_2020} studied general bilinear forms of eigenvectors in this setting.  Our results do not apply in these settings even if $n = d$ since we require that $E_{ij}$ and $E_{ji}$ are independent if $j\neq i$. On the other hand, their results cannot be applied in our setting either because of the dependence structure.

Regarding our $\ell_{2,\infty}$ concentration results, perhaps the most similar results appear in \citet{cai_subspace_2020}, in which the authors study the performance of the diagonal deletion algorithm in the presence of independent heteroskedastic noise.  Our work differs from theirs in several ways:
\begin{itemize}
    \item We allow for dependence among the rows of the noise matrix, whereas \citet{cai_subspace_2020} requires the entries of the noise matrix to be independent random variables. 
    \item \citet{cai_subspace_2020} allow for missingness in the matrix $M$, whereas we assume that the matrix $M$ is fully observed.
    \item Assumption \ref{assumption:eigengap} requires that $\sqrt{d \log(n\vee d)} \lesssim \lambda_r/\sigma$, whereas the theory in \citet{cai_subspace_2020} covers the setting $(nd)^{1/4} \lesssim \lambda_r/\sigma$, which is a broader range than ours for the $\snr$ regime.   %\textcolor{red}{
    However, we note that our $\ell_{2,\infty}$ concentration holds under the weaker assumption $\lambda_r/\sigma \gtrsim \kappa (nd)^{1/4} \sqrt{r\log(n\vee d)}$ (which can be seen from the main proof in Section \ref{sec:proofs}).%}
    \item Our main results include both asymptotic normality and $\ell_{2,\infty}$ concentration, whereas \citet{cai_subspace_2020} only obtain $\ell_{2,\infty}$ concentration.  
    \item Our $\ell_{2,\infty}$ concentration result does not incur the ``diagonal deletion effect'' in the upper bound of Theorem 1 in \citet{cai_subspace_2020}, since this has been accounted for using the HeteroPCA algorithm (see Theorem~\ref{thm:2_infty_deterministic}). This allows our upper bound to scale with the noise.
\end{itemize}

The most important of these differences is perhaps the last point, since eliminating the diagonal deletion effect is crucial for our asymptotic normality analysis. Besides us removing this term, our upper bound agrees with theirs up to a $\sqrt{r}$ factor. Also, since we have rid ourselves of the error coming from diagonal deletion, we are able to achieve the minimax lower bound for this problem given in Theorem 2 of \citet{cai_subspace_2020} up to log factors when $r, \mu_0, \kappa \asymp 1$. %Our proof techniques also differ from \citet{cai_subspace_2020}, who use a leave-one-out analysis to obtain their bounds, whereas we use a series expansion together with a leave-one-out analysis to obtain our bounds.  

%\textcolor{red}{
Shortly after posting our manuscript to ArXiv and submitting for publication, a very closely related manuscript \citep{yan_inference_2021} was also posted, studying a very similar setting to ours.  In \citet{yan_inference_2021}, the authors study statistical inference for Heteroskedastic PCA under the spiked covariance model, where the spike component is assumed to be low rank; moreover, they also use the HeteroPCA algorithm of \citet{zhang_heteroskedastic_2021}. They also obtain $\ell_{2,\infty}$ concentration and asymptotic normality, though their asymptotic normality results are not directly comparable, as they focus on statistical inference for the spike component (as opposed to the singular subspace directly). The key difference is that our asymptotic normality results allow for dependence within rows, which is a setting not covered by \citet{yan_inference_2021}.
However, our $\ell_{2,\infty}$ concentration result is markedly similar to theirs (see their Theorem 10) and agrees up to factors of $r$ and $\kappa$.  In \citet{yan_inference_2021}, the authors study the regime  $(nd)^{1/4} \lesssim \lambda_r/\sigma$ (ignoring logarithmic terms, factors of $r$, and factors of $\kappa$).  While our $\ell_{2,\infty}$ concentration continues to hold in this regime, our asymptotic normality result may not hold.  %, since there is an additional residual term that need not tend to zero after dividing through by $\sigma_{ij}$ (see Lemma \ref{lem:final_res} in Appendix \ref{sec:be_proof}) unless one has sufficiently strong signal.  
Since our main focus in this paper is the entrywise estimation of singular subspaces under both dependence and heteroskedasticity, we leave deriving limit theory and asymptotic normality in this regime to future work. 

Finally, in \citet{abbe_ell_p_2020}, the authors study exact recovery in the case of the two-component mixture model.  Their results are similar to ours in that they allow for dependence and heteroskedasticity, but they do not study the limiting distribution of their diagonal deletion estimator. Moreover, their results are not directly comparable, as they use a different definition of incoherence and do not study the explicit dependence of their bound on the noise parameters and the spectral structure of the matrix $M$, but instead find conditions on their signal-to-noise ratio such that their  upper bound tends to zero.  On the other hand, their theory covers the weak-signal regime, and they extend their results to Hilbert spaces. In principle, since our results depend only on the properties of the Gram matrix, they could also be extended to a general Hilbert space, but we do not pursue such an extension here. % \textcolor{red}{We remark that our $\ell_{2,\infty}$ results could be applied in the mixtures setting to obtain perfect recovery guarantees for general subgaussian mixture models in terms of the mean separation, but we elect to leave the precise analysis to future work.}

%Assumption \ref{assumption:incoherence} is an incoherence assumption often required in the literature on low-rank estimation.  In \citet{bao_singular_2021} it was shown that the limit of the $\sin\Theta$ distance need not be Gaussian if $U$ and $V$ are not sufficiently incoherent.  
%
%%Realtionship to \citet{abbe_ell_p_2020}
%
%CLT comparison \citet{bao_singular_2021}, \citet{xia_confidence_2019},\citet{xia_normal_2019},\citet{xia_statistical_2020}, \citet{cheng_tackling_2020},\citet{cheng_tackling_2020}, \citet{cai_subspace_2020},\citet{abbe_ell_p_2020}
%
%
%
%two to infinity comparison 
%
%
%Bound scales with the noise.  Optimality when $d\gg n$.  

\subsection{Application to Mixture Distributions} \label{sec:applications}
%In this section we consider two concrete statistical applications.  

%\subsubsection{Mixture Distributions}
Consider the following submodel.  Suppose we observe $n$ observations of the form $X_i = M_i + E_i \in \R^d$, where there are $K$ unique vectors $\mu_1, \dots, \mu_K$.  Let $\hat M$ be the matrix whose $i$'th row is $X_i\t$.  If $M$ is rank $K$, by Lemma 2.1 of \citet{lei_consistency_2015}, there are $K$ unique rows of the matrix $U$, where each row $i$ corresponds to the membership of the vector $X_i$.  We then have the following Corollary to Theorem \ref{thm:b-e} in this setting.

\begin{corollary} \label{cor:mixture}
Let $X_i = M_i + E_i$, and define the matrix
\begin{align}
    S_i := \Lambda\inv V\t \Sigma_i V \Lambda\inv. \label{eq:comm_cov}
\end{align}
Suppose $\kappa, \kappa_{\sigma}, \mu_0$ are bounded  and $r$ stays fixed as $n$ and $d$ tend to infinity, and suppose \begin{align*}
    \frac{\log(n\vee d)}{\snr  } \to 0.
\end{align*}
Then
\begin{align*}
    (S_i)^{-1/2} \bigg( \hat U \mathcal{O}_* - U \bigg)_{i\cdot} \to N(0, I_r)
\end{align*}
as $n$ and $d$ tend to infinity with $d\geq n\geq \log(d)$.  %$n\asymp d$.  
\end{corollary}

We remark that the result above allows $E_i$ to have an $i$-dependent covariance matrix.  In the setting that the covariance matrix of $E_i$ depends only on the vector $\mu_k$, where $k$ is such that $M_i = \mu_k$,
the following result shows that we can leverage this structure to consistently estimate the matrix $S_i$, which is the same within each community.  We assume that one can accurately estimate the cluster memberships with probability tending to one, which holds by the signal-to-noise ratio condition, the $\ell_{2,\infty}$ bound in Theorem \ref{thm:2_infty}, and the setting for Corollary \ref{cor:mixture} since the eigenvector difference $\|\hat U - U \mathcal{O}_*\|_{2,\infty} \ll \|U\|_{2,\infty}$, which implies that the rows of $\hat U$ are asymptotically separated (see e.g. \citet{lei_consistency_2015}).

\begin{corollary} \label{cor:consistent_estimate}
Suppose the setting for Corollary \ref{cor:mixture}, and assume that there are $K$ different communities with each community having mean $\mu_k$ and covariance matrix $\Sigma^{(k)}$ (that is, $\Sigma_i = \Sigma^{(k)}$ for all $i$ in community $k$).  Let $n_k$ denote the number of observations in community $k$, and suppose that $n_k \asymp n$.  Suppose $\bar{U}^{(k)}$ is the estimate for the centroid of the $k$-th mean, and let $C_k$ denote the set of indices such that $M_i = \mu_k$.  Define the estimate
\begin{align*}
    \hat S^{(k)} := \frac{1}{n_k} \sum_{i \in C_k} \bigg(\hat U_{i\cdot}  - \bar{U}^{(k)}  \bigg) \bigg(\hat U_{i\cdot}  - \bar{U}^{(k)} \bigg)\t. 
\end{align*}
Then for the orthogonal matrix $\mathcal{O}_*$ appearing in Corollary \ref{cor:b-e}, 
\begin{align*}
    \|(S^{(k)})\inv\mathcal{O}_*\t\hat S^{(k)}\mathcal{O}_*  - I_r\| \to 0
\end{align*}
in probability, where $S^{(k)}$ is the community-wise covariance defined in \eqref{eq:comm_cov}.
\end{corollary}
We note that the appearance of the orthogonal matrix $\mathcal{O}_*$ is of no inferential consequence, since Gaussianity is preserved by orthogonal transformation. This result implies that one can consistently estimate the covariance matrix for the corresponding row, which immediately implies that one can derive a pivot for the $i$'th row in the mixture setting described above, by setting
\begin{align*}
    \hat T_{i} := (\hat S^{(k)})^{-1/2}(\hat U_i - \bar U^{(k)}).
\end{align*}
Corollaries \ref{cor:mixture} and \ref{cor:consistent_estimate}, the Continuous Mapping Theorem, and Slutsky's Theorem imply that $\hat T_i \to N(0, I_r)$ as $n$ and $d$ tend to infinity, which provides an asymptotically valid confidence region.  We remark that in the asymptotic regime in Corollaries \ref{cor:mixture} and \ref{cor:consistent_estimate}, when $r$ is fixed, any fixed finite collection of rows can be shown to be asymptotically independent, and hence the confidence region is simultaneously valid for any fixed set of rows; for example, for one row each of each community.

\section{Numerical Results} \label{sec:simulations}

We consider the following mixture model setup.  We let $M$ be the matrix whose first $n_1$, $n_2$, and $n_3$ rows are $\mu_1, \mu_2$, and $\mu_3$ respectively where
\begin{align*}
    \mu_1 :&= (10 , 10, \dots, 10, 12, \dots, 12)\t \\
    \mu_2 :&= (10 , 10 ,\dots ,10 ,10,\dots ,10)\t \\
    \mu_3 :&= (5, 5, \dots , 5, 5.5, \dots, 5.5)\t
\end{align*}
%with $n_1 = 200, n_2 = n_3= 400$.  
The  matrix $M$ is readily seen to be rank $2$, since $\mu_3 = .25 \mu_1 + .25 \mu_2$.  %We consider the class-wise covariances
%\begin{align*}
%    \Sigma_i := \sigma_i^2 F_i F_i\t + \sigma^2 I_d,
%\end{align*}
%where $\sigma_1 = 15$, $\sigma_2 =10$, $\sigma_3=7.5$, $\sigma^2 = 1$, and $F_i$ is drawn uniformly on the Stiefel manifold of dimensions $100,50,$ and $200$ respectively.   We keep $F_i$ fixed for each $n_1$, $n_2$, and $n_3$ rows respectively.
% \begin{figure}[H]
%     \centering
%     \includegraphics[width=0.85\textwidth]{b_e1.png}
%     \caption{Histogram of 200 Monte Carlo iterations of the difference of the $i,j$ entry of $\hat U \mathcal{O} - U$ with $i = j = 1$ in the model setup in Section \ref{sec:simulations}. The dotted line represents the theoretical variance from Theorem \ref{thm:b-e} and the solid line is the estimated variance.}
%     \label{fig:be}
% \end{figure}

%I%n Figure \ref{fig:be} we plot 200 Monte Carlo iterations of $(\hat U \mathcal{O} - U)_{11}$, where $\mathcal{O}$ is estimated using a Procrustes alignment between $\hat U$ and $U$.  The solid line represents the estimated Gaussian density, and the dotted line represents the Gaussian density implied by Theorem \ref{thm:b-e}.  The theoretical and estimated densities are readily seen to be close.
%Similarly, 

We compare $\hat{U}$ to the diagonal deletion estimator in this setting in  Figure~\ref{fig:illustration} of Section~\ref{sec:bckgd}, where we can clearly see the bias that comes from deleting the diagonal entries in the final approximation of the singular vectors. We consider the class-wise covariances
\begin{align*}
    \Sigma^{(k)} := \sigma_k^2 F_k F_k\t + I_d,
\end{align*}
where $\sigma_1 = 15$, $\sigma_2 =10$, $\sigma_3=7.5$, and $F_k$ is drawn uniformly on the Stiefel manifold of dimensions $100,50,$ and $200$ respectively, and $n_1 = 200$, and $n_2= n_3 = 400$, with $d = 1000$.  %What is more, 
For the smallest class with mean $\mu_1$ we clearly see that the reduced incoherence severely impacts the estimation with the diagonal deletion estimator, while the effect on $\hat{U}$ is relatively small. On the latter two classes, we observe that the rows of $\hat{U}$ are very close to those of the idealized matrix $A+\Gamma(Z)$, and the covariances are comparable between these. On the other hand, the rows of the diagonal deletion estimator do not preserve the covariance structure of this idealized matrix, since the diagonal deletion estimator is approximating the eigenvectors of $\Gamma(A+Z)$. Since $A+\Gamma(Z)$ is an unbiased perturbation of $A$ while $\Gamma(A+Z)$ is not, we consider this to be the more natural object of comparison, and our theory supports this view.

To study the effect of larger $n$ and $d$, in Figure \ref{fig:be2} we plot 1000 Monte Carlo iterations of $(\hat U \mathcal{O}_* - U)_{1\cdot}$ where $\mathcal{O}_*$ is estimated using a Procrustes alignment between $\hat U$ and $U$ on $n = d = 1500$ points with $n_1 = n_2 = n_3 =500$, where we use the balanced case to ensure incoherence.  The solid line represents the estimated 95\% confidence ellipse, and the dotted line represents the  95\% confidence ellipse implied by Theorem \ref{thm:b-e}.  We consider the spherical noise setting, with class-wise covariances $.1 I_d$, $.2 I_d$, and $.3 I_d$ for each component respectively.  The empirical and theoretical ellipses are readily seen to be close.

\begin{figure}
    \centering
    \includegraphics[width=0.8\textwidth,height=90mm]{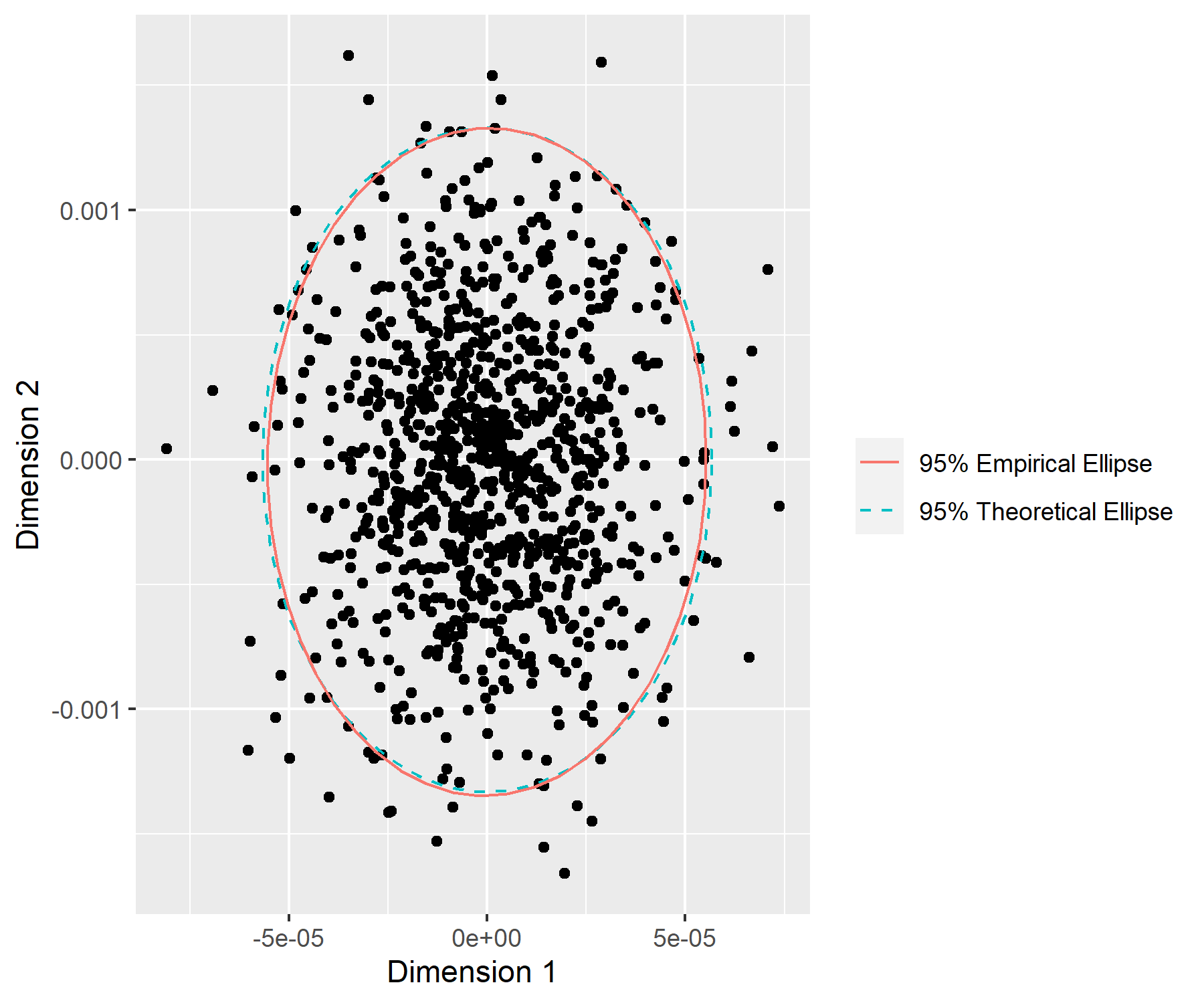}
    \caption{Plot of 1000 Monte Carlo iterations of the first row of $\hat U \mathcal{O}_* - U$ with the same $M$ matrix as above and $n = d = 1800$ with $n_1 = n_2 = n_3 = 600$.  The covariances here are spherical within each mixture component, though they differ between components. The dotted line represents the theoretical 95 percent confidence ellipse from Theorem \ref{thm:b-e} and Corollary \ref{cor:mixture}.  The solid line is the estimated ellipse. The different scalings on the two axes arise because the variances in these two dimensions are proportional to the first two squared singular values of $M$, as seen in Corollary~\ref{cor:b-e}.  Further details are in Section \ref{sec:simulations}.}
    \label{fig:be2}
\end{figure}

\subsection{Elliptical Versus Spherical Covariances} \label{sec:spherical_elliptical}

In Figure \ref{fig:spherical_elliptical} we examine the effect of elliptical covariances on the limiting distribution predicted by Corollary \ref{cor:mixture}.  We consider the same $M$ matrix and mixture sizes as in Figure \ref{fig:illustration}, only we fix the covariances as 
\begin{align*}
    \Sigma_E^{(1)} :&= F_1 F_1\t + .1 I_d \\
    \Sigma^{(2)} :&= F_2 F_2\t + I_d \\
    \Sigma^{(3)} :&= 2I_d,
\end{align*}
where $F_1$ and $F_2$ are square matrices with entries drawn independently from uniform distributions on $[0,.003]$ and $[0,.001]$ respectively.  We also consider the spherical case $\Sigma_S^{(1)} = VV\t + I_d$.  We run 1000 iterations of this simulation and examine the first row of the matrix $(\hat U \mathcal{O}_* - U)\Lambda$, where again $\mathcal{O}_*$ is estimated using the procrustes difference between $\hat U$ and $U$.  The only thing we change between each dataset is the covariance; i.e., we draw $E_1 \in \R^{n_1 \times d}$ as a random Gaussian matrix with independent entries and multiply by $(\Sigma_E^{(1)})^{1/2}$ or $(\Sigma_S^{(1)})^{1/2}$ to obtain the first $n_1$ rows of the matrix $E$.  We keep the other $n - n_1$ rows fixed within each Monte Carlo iteration, so the only randomness for each iteration is in drawing the $nd$ Gaussian random variables.  
\begin{figure}
    \centering
    \includegraphics[width=.8\textwidth,height=85mm]{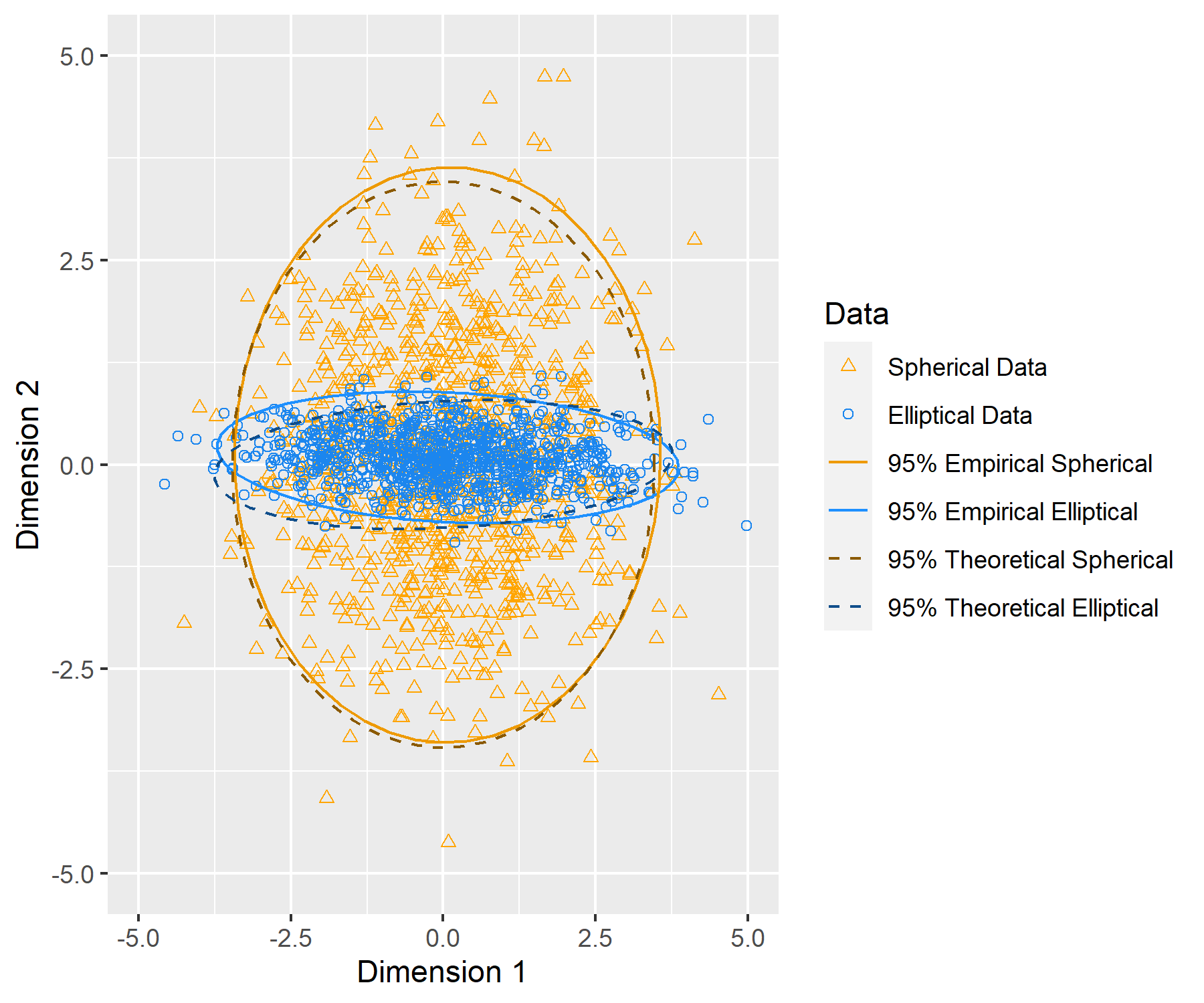}
    \caption{Comparison of $\Lambda(\hat U \mathcal{O}_* - U )_{1\cdot}$ for the same mixture distribution as above, only we modify the covariance for the first mixture component between each data set. Details are in Section \ref{sec:spherical_elliptical}.}
    \label{fig:spherical_elliptical}
\end{figure}
We scale $\hat U\mathcal{O}_* - U$ by $\Lambda$ in order to explicitly showcase the covariance structure.  The way we create $\Sigma_E^{(1)}$ yields that the limiting covariance $S_E^{(1)}$ is approximately
\begin{align*}
    S_E^{(1)} :=            \Lambda\inv \begin{pmatrix} 2.35 & 0.083  \\ %& 78611
 0.083 & 0.104 \end{pmatrix} \Lambda\inv,
\end{align*}
so that $\Lambda(\hat U\mathcal{O}_* - U)_{1\cdot}$ is approximately Gaussian with covariance $\begin{pmatrix} 2.35 & 0.083  \\ %& 78611
 0.083 & 0.104 \end{pmatrix}$.  On the other hand, when we consider the spherical case, we see that $S_S^{(1)}$ is of the form
    \begin{align*}
        S_S^{(1)} &= 2 \Lambda^{-2}
    \end{align*}
    so that $\Lambda(\hat U\mathcal{O}_* - U)_{1\cdot} $ is approximately Gaussian with covariance $2 I_r$.      Figure \ref{fig:spherical_elliptical} shows these differences, where we plot the empirical 95\% confidence ellipses with respect to both the estimated covariance (solid line) and theoretical covariance (dashed line).  
    
    To study the relationship between $\Sigma_i$ and $V$ in more detail, we also consider the following setting.  We consider the class-wise covariances again
    \begin{align*}
        \Sigma^{(1)} :&= 15 F_1 F_1\t + .1 I_d \\
    \Sigma_\theta^{(2)} :&= 5 V_{\cdot 1} V_{\cdot 1}\t + 5 V_2^\theta (V_2^\theta)\t + .1 I_d\\
    \Sigma^{(3)} :&= 10 F_3 F_3\t +.1 I_d,
    \end{align*}
    where again $F_1$ and $F_3$ are drawn uniformly from the Stiefel manifold of dimension 100 and 200 respectively, with $n_1 = 200$ and $n_2= n_3 = 400$. We also change $\mu_3=(5,\ldots,5,6,\ldots6)$ to better separate the clusters. The vector $V_2^\theta$ is orthogonal to $V_{\cdot 1}$ and satisfies $\langle V_2^\theta , V_{\cdot 2} \rangle = \theta$ for $\theta \in \{.9,.5,.1\}$.  As $\theta$ decreases, the limiting covariance matrix in Corollary \ref{cor:mixture} will change along the second dimension only.   Figure \ref{fig:angle} reflects this theory, where we plot 1000 Monte Carlo runs of $\Lambda(\hat U\mathcal{O}_* - U)_{(n_1+1)\cdot}$.  The variance stays fixed along the first dimension, but it shrinks along the second dimension, showcasing the geometric relationship between $V$ and $\Sigma^{(2)}$ as suggested by Corollary \ref{cor:mixture}.

    \begin{figure}
        \centering
        \includegraphics[width=.8\textwidth,height=85mm]{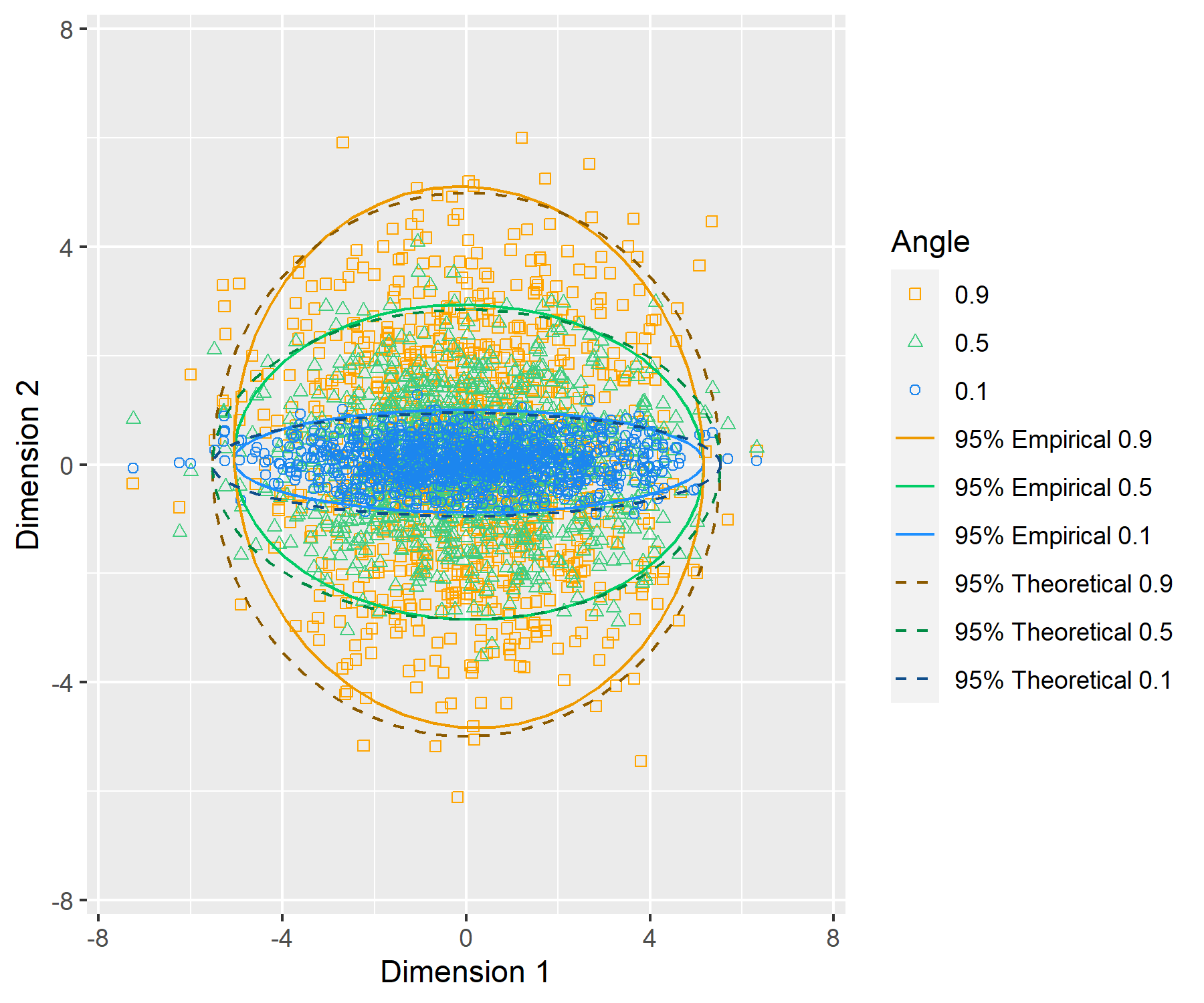}
        \caption{Comparison of $\Lambda(\hat U \mathcal{O}_* - U )_{(n_1+1)\cdot}$ for the same mixture distribution as above, only we modify the covariance for the second mixture component between each data set by changing the angle between the leading covariance and the second mixture component. Details are in Section \ref{sec:spherical_elliptical}.}
        \label{fig:angle}
    \end{figure}

%Simple model and some figs??????

%(Maybe a figure on how diagonal deletion is bad, but that might not be important).  

\section{Discussion}
\label{sec:discussion}

We have shown that under general model assumptions for the noise matrix, allowing for heteroskedasticity between rows and dependence within them, that the entries of the left singular vectors of the output from HeteroPCA are consistent estimators for those of the original signal matrix $M$, and the errors are asymptotically normally-distributed in a natural high-dimensional regime. Furthermore, our Berry-Esseen theorem makes clear the rate at which this asymptotic approximation becomes valid, revealing the  effect of the relationship between the noise covariances and the spectral structure of the signal matrix on the distributional convergence.  In the particular case that the individual covariances are scalar multiples of the identity, our results also show that the variances of the entries of the  $j$'th estimated singular vector are proportional to the inverse $j$'th singular value of the signal matrix. In particular, this means that estimating additional singular vectors in this model becomes more challenging and requires more data, since the variance in this estimation grows with $j$.
%Talk about our statistical applications section here also.

In this paper, we assume the rank $r$ is known \emph{a priori}, but in practice one may need to estimate $r$ using, for example, the methods proposed in \citet{zhu_automatic_2006}, \citet{han_universal_2020}, or \citet{yang_simultaneous_2020}.  In addition, while our results highlight the interplay of the dependence structure of the noise with the signal matrix, additional work is required to make the upper bound computable from observed data. For example, our results do not imply consistent estimation for the row-wise covariance matrices $\Sigma_i$, though we do show if one has covariances that are only distinct between clusters, then one can estimate the limiting covariance matrix $S_i$ for each row of $\hat{U}$.  %Provided one can estimate the individual covariances, 
Our techniques could therefore be appropriately modified to develop two-sample asymptotically valid confidence regions or test statistics, such as in deriving a Hotelling $T^2$ analogue for the singular vectors as in \citet{fan_simple_2019,du_hypothesis_2021}. Furthermore, one possibility for further inference would be to consider drawing several matrices $M+E$ independently from the same distribution, assuming the rows are matched together between samples, in which case one could leverage existing statistical methodology to conduct two-sample tests of hypothesis. 

Another possible extension is estimating linear forms of singular vectors under the dependence structure we consider, which has been studied in other settings under independent noise.  Our results naturally extend to sufficiently sparse linear forms (i.e. linear forms $T$ such that $\| T U\|\leq b \|U\|_{2,\infty}$),  but studying linear forms for which $ \| T U \|\asymp 1$ would require additional methods, as the entrywise analysis methods we use would not be applicable. Finally, our results hold for a natural subgaussian mixture model, but many high-dimensional datasets contain outlier vectors or heavier tails, in which case additional techniques are required.

%with the given observation, it is not clear how one would estimate the row-wise covariance matrices $\Sigma_i$. One possibility would be to consider drawing several matrices $M+E$ independently from the same distribution, assuming the rows are matched together between samples. This could then be used to derive confidence regions for the rows of $U$ or a two-sample test statistic between populations. 

%future work:
%- estimating the covariance
%- testing/inference
%- linear forms of singular vectors
%- outlier models
%- other mixture models with heavier tails
%- other forms of dependence

%future work: using this as some kind of two-sample test?  idk

\section{Proof Architecture for Theorems \ref{thm:b-e} and \ref{thm:2_infty}}
\label{sec:proofs}
In this section we state several intermediate lemmas, prove Theorem \ref{thm:2_infty},  and sketch the proof of Theorem \ref{thm:b-e}.  Full proofs are in the appendices.  First we collect some initial spectral norm bounds that are useful in the sequel.  The first is a bound on the noise error $\|\Gamma(Z)\|$, the proof of which is adapted from Theorem 2 in  \citet{amini_concentration_2021}.  %The proof is in Appendix \ref{sec:lemmas_random}. 
Throughout this section and all the proofs, we allow constants $C$ to change from line to line.  %Occasionally we write $f(n,d) \lesssim g(n,d)$ to mean that there exists a sufficiently large constant $C$ such that $f(n,d) \leq C g(n,d)$ for $n$ and $d$ sufficiently large.  

\begin{restatable}[Spectral Norm Concentration]{lemma}{spectralnormconcentration} \label{lem:spectral_norm_concentration}
Under assumption \ref{assumption:noise},  there exists a universal constant $C_{\text{spectral}}$ such that with probability at least $1 - 4(n\vee d)^{-6}$
\begin{align*}
    \| \Gamma(EM\t + ME\t + EE\t) \| \leq C_{\text{spectral}}\bigg(  \sigma^2 ( n + \sqrt{nd}) + \sigma \sqrt{n} \kappa \lambda_r \bigg).
\end{align*}

\end{restatable}

The next bound shows that the approximation of $\tilde A=A+\Gamma(Z)$ to $\hat A$, the output of the HeteroPCA algorithm, is much smaller than the approximation of $\tilde A$ to $A$. Recall that $N_T$ denotes the approximation of $A$ from Algorithm~\ref{alg1} after $T$ iterations. We note that the existence of $T_0$ in the statement of this lemma follows from \citet{zhang_heteroskedastic_2021} and Assumptions~\ref{assumption:dimension} and \ref{assumption:incoherence}.%  \textcolor{red}{
In particular, if we take $T_0 \geq C \log\left( \frac{\lambda_r^2}{\|\Gamma(Z)\|} \right)$, then by the proof of Theorem 7 in \citet{zhang_heteroskedastic_2021}, it holds that $\| N_T - A \| \leq 3 \| \Gamma(Z) \|$.%}

\begin{lemma}
\label{lem:deterministic_spectral}
Define $T_0$ as the first iteration such that $\|N_T - A\|\leq 3\|\Gamma(Z)\|$. Let $\rho=10\|\Gamma(Z)\|/\lambda_r^2$, and suppose Assumption~\ref{assumption:eigengap}.% \textcolor{red}{
Define $\tilde K_T := \|N_T - \tilde A\|$.
%}
Then for all $T \geq T_0$ and $n$ large enough, on the event in Lemma \ref{lem:spectral_norm_concentration}, we have that %$\rho<1$ %\textcolor{red}{
$\rho < \frac{1}{2}$, %},
and
\begin{align*}
     \tilde K_{T} &\leq 4 \rho^{T-T_0} \|\Gamma(Z)\| + \frac{20}{1-\rho}\|U\|_{2,\infty}\|\Gamma(Z)\|.
\end{align*}
%\textcolor{red}{
Consequently, when $T = \Theta\left( \log\left( \frac{\lambda_r^2}{\|U\|_{2,\infty} \| \Gamma(Z)\|} \right) \right)$, it holds that 
%In particular, when $T - T_0 = \Omega( \log( \frac{1}{ \| U\|_{2,\infty}}) )$, then 
$$\| \hat A - \tilde A \| \leq 41 \| U \|_{2,\infty} \| \Gamma(Z)\|.$$%} %Consequently, we must have that $T = \Omega\bigg( \log( \frac{\lambda_r^2}{\| \Gamma(Z) \|}) + \log( \frac{1}{\| U\|_{2,\infty}}) \bigg)$.}
%
%In particular, when $T$ is very large, $$\|\hat{A}-\tilde{A}\|\leq 41\|U\|_{2,\infty}\|\Gamma(Z)\|.$$
\end{lemma}

Assumption \ref{assumption:eigengap} implies that
 \begin{align*}
     \lambda_r^2 &\geq C \bigg(\sigma^2 rd \log(n\vee d) + \sigma \sqrt{n r\log(n\vee d)} \kappa \lambda_r\bigg),
 \end{align*}
  which ensures that that there is an eigengap on the event in Lemma \ref{lem:spectral_norm_concentration}.   Therefore, a standard application of the Davis-Kahan Theorem (e.g. \citet{chen_spectral_2021}) and Lemma \ref{lem:spectral_norm_concentration} immediately implies that $\hat U \hat U\t - UU\t$ tends to zero in spectral norm. 
%  \begin{align*}
%      \| \tilde U\tilde U\t - UU\t \| &\lesssim \frac{\|\Gamma(EM\t + ME\t + EE\t) \|}{\lambda_r^2} \\
%      &\lesssim \frac{ \sigma^2 ( n + \sqrt{nd})}{\lambda_r^2} + \frac{\sigma \sqrt{n} \kappa }{\lambda_r} \\
%      &\lesssim \frac{1}{\log(n\vee d)} + \frac{1}{\sqrt{\log(n\vee d)}} \\
%      &\lesssim  \frac{1}{\sqrt{\log(n\vee d)}},
%  \end{align*}
%  which ensures consistency for the leading eigenvectors of $\tilde A$ and $A$.  Here we write $a \lesssim b$ if $a \leq C b$ for some constant $C$.  
%Another application of the Davis-Kahan Theorem  ensures that the upper bound on $\|\hat U \hat U\t - \tilde U\tilde U\t\|$ is of much smaller order than the upper bound on $\| \tilde U \tilde U\t - UU\t\|$. 

To prove our main $\ell_{2,\infty}$ result, we analyze the statistical error and algorithmic error separately.  Define $H:= U\t \tilde U$, and $\tilde H := \tilde U\t \hat U$. %To prove our main $\ell_{2,\infty}$ result, we %first show that we can 
First, 
write $\hat U - U\mathcal{O}$ as %the sum of a ``deterministic part" and a random part via
$$\hat{U}-U\mathcal{O}= (\hat{U}-\tilde{U}\tilde H)+(\tilde{U}-UH)\tilde H+U(H\tilde H-\mathcal{O}).$$ The first term captures the algorithmic error between the eigenvectors of the output of Algorithm~\ref{alg1}, $\hat{A}$, and those of the matrix it approximates, $\tilde{A}$. The next term is the statistical error between the matrix $\tilde{A}$ approximated by the algorithm and the true matrix of interest $A$. Finally, we have a correction term which accounts for the fact that $\tilde{H}$ and $H$ are contractions rather than orthogonal matrices.  %The following lemma shows that the correction 

% \begin{align*}
%     \hat U - UW &= U (U\t \tilde U \tilde U\t \hat U - W) + \hat U - U U\t \tilde U \tilde U\t \hat U \\
%     &= U (U\t \tilde U \tilde U\t \hat U - W) + \hat U - U U\t \tilde U \tilde U\t \hat U + \tilde U\tilde U\t \hat U - \tilde U\tilde U\t \hat U \\
%     &=  U (U\t \tilde U \tilde U\t \hat U - W) + \bigg[ \tilde U\tilde U\t - U U\t \tilde U \tilde U\t \bigg]\hat U + \bigg[\hat U -  \tilde U\tilde U\t \hat U \bigg].
% \end{align*}

Since the bound on the algorithmic error depends on the properties of $\tilde{U}$, we first prove the bound on the middle term, or the statistical error between $\tilde{U}$ and $U$. Then we bound the algorithmic error and finally the correction term. The following result bounds the statistical error between $\tilde{U}$ and $UH$ with high probability.

\begin{theorem}\label{thm:2_infty_random}
Under Assumptions \ref{assumption:noise}, \ref{assumption:eigengap},  \ref{assumption:dimension} and \ref{assumption:incoherence}, we have that there exists a universal constant $C_R$ such that
\begin{align*}
    \| \tilde U - UH\|_{2,\infty} &\leq C_R \bigg(  \frac{\sqrt{r  n d} \log (\max (n, d)) \sigma^{2}}{\lambda_r^2} + \frac{\sqrt{rn \log (n\vee d)} \kappa \sigma }{\lambda_r} \bigg) \|U\|_{2,\infty} 
\end{align*}
with probability at least $1 - (n\vee d)^{-4}$.
\end{theorem}

In order to prove this result, we use the matrix series expansion developed in \citet{xia_normal_2019} to write the difference of projection matrices in terms of the noise matrix, each term of which requires careful  considerations due to the dependence between columns of the noise matrix $E$.%The proof of this theorem is in Appendix \ref{sec:randompart}.

Consequently, by Assumption \ref{assumption:eigengap}, the result in Theorem \ref{thm:2_infty_random} implies that
\begin{align*}
    \| \tilde U\|_{2,\infty} &\leq \|\tilde U - UH\|_{2,\infty} + \|U\|_{2,\infty} \\
    &\lesssim \|U\|_{2,\infty},
\end{align*}
which shows that the matrix $\tilde U$ is just as incoherent as $U$ up to constant factors.  We also have the following result for the deterministic analysis.

\begin{theorem}\label{thm:2_infty_deterministic}
Define $\tilde H := \tilde U\t \hat U$. Suppose the event in Theorem \ref{thm:2_infty_random} holds.  Then, in the setting of Lemma~\ref{lem:deterministic_spectral}, we have that there exists a universal constant $C_D$ such that
%$$\|\hat{U}-\tilde{U}\tilde{H}\|_{2,\infty}\leq \frac{\tilde{\lambda}_1}{\hat{\lambda}_r}\|\tilde{U}\|_{2,\infty}\left[\sqrt{2}\|\sin\Theta(\hat{U},\tilde{U})\|(1+2\sqrt{2}(\hat{\lambda}_1/\tilde{\lambda}_1))+\frac{K_T}{\tilde{\lambda}_1}\right]\left[1-\frac{4\hat{\lambda}_1}{\hat{\lambda}_r}\|\sin\Theta(\hat{U},\tilde{U})\|\right]^{-1}.$$
$$\|\hat{U}-\tilde{U}\tilde{H}\|_{2,\infty}\leq C_D  \kappa^2\frac{\|U\|_{2,\infty}^2\|\Gamma(Z)\|}{\lambda_r^2}.$$
\end{theorem}

Finally, we consider the correction term.

\begin{restatable}{lemma}{orthogonalmatrixlemma}\label{orthogonal_matrix_lemma}
There exists an orthogonal matrix $\mathcal{O}_*$ and a universal constant $C$ such that under Assumptions \ref{assumption:eigengap} and \ref{assumption:incoherence}, the event in Lemma \ref{lem:spectral_norm_concentration}, %\textcolor{red}{
and $T = \Theta\left( \frac{\lambda_r^2}{\|U\|_{2,\infty} \|\gamma(Z)\|} \right)$,
%}
\begin{align*}
   \| U H \tilde  H  - U\mathcal{O}_*\t \|_{2,\infty} &\leq C \|U\|_{2,\infty} \frac{\|\Gamma(Z)\|^2}{\lambda_r^4}.
 \end{align*}
\end{restatable}

%Note that both $H$ and $\tilde H$ have spectral norm at most one.  We note that 
%\begin{align*}
 %   \left\| \bigg[ \tilde U\tilde U\t - U U\t \tilde U \tilde U\t \bigg]\hat U \right\|_{2,\infty} &= \| \tilde U \tilde H - U H \tilde H \|_{2,\infty} \\
 %   &\leq \| \tilde U - U H  \|_{2,\infty} \| \tilde H\| \\
%    &\leq \| \tilde U - U H  \|_{2,\infty}.
%\end{align*}
%Therefore, we need only analyze $\|\tilde U - UH\|_{2,\infty}$ and $\| \hat U - \tilde U\tilde H\|_{2,\infty}$, which we refer to as the ``Random Part'' and the ``Deterministic Part'' respectively.  We  have the following Theorem.

We now have all the pieces to prove Theorem \ref{thm:2_infty}.  
\begin{proof}[Proof of Theorem \ref{thm:2_infty}]
We note that all the events that determine Lemmas~\ref{lem:deterministic_spectral},  \ref{orthogonal_matrix_lemma}, and Theorem~\ref{thm:2_infty_deterministic} are the event in Lemma \ref{lem:spectral_norm_concentration}, and the event in Theorem \ref{thm:2_infty_random}. Taking a union bound, these events occur simultaneously with probability at least $1 - (n\vee d)^{-4} - 4(n\vee d)^{-6} \geq 1 - 2(n\vee d)^{-4};$ henceforth we operate on the intersection of these events. Since Theorem~\ref{thm:2_infty_random} gives the stated upper bound for $\|\tilde{U}-UH\|_{2,\infty}$ on this set, by increasing the constants if necessary, we need only show that this bound holds for the algorithmic error $\|\hat{U}-\tilde{U}\tilde{H}\|_{2,\infty}$, and the correction term $\|UH\tilde{H}-U\mathcal{O}_{*}\|_{2,\infty}$.

Under assumption \ref{assumption:incoherence}, we have that 
\begin{align*}
    \kappa^2  \|U\|_{2,\infty} &\leq \kappa^2  \mu_0 \sqrt{r/n} \\
    &\leq \sqrt{r\log(n \vee d)}.
\end{align*}
Hence, the bound in Theorem \ref{thm:2_infty_deterministic} becomes
\begin{align}
    \|\hat U - \tilde U\hat{H}\|_{2,\infty} \leq C_D \|U\|_{2,\infty} \sqrt{r\log(n \vee d)} \frac{\|\Gamma(Z)\|}{\lambda_r^2}. \label{det_bd}
\end{align}
In addition, on the event of Lemma \ref{lem:spectral_norm_concentration},
\begin{align*}
    \|\Gamma(Z)\| \leq C_{\text{spectral}}\bigg(  \sigma^2 ( n + \sqrt{nd}) + \sigma \sqrt{n} \kappa \lambda_r \bigg).
\end{align*}
This gives the desired bound for the algorithmic error.
%Therefore, the bound in \eqref{det_bd} is upper bounded by the bound in Theorem \ref{thm:2_infty_random}.  
%By Lemma \ref{lem:deterministic_spectral} we have that
%\begin{align*}
%    \| \hat A - \tilde A \|\leq C \|U\|_{2,\infty}\| \Gamma(Z) \|.
%\end{align*}
For the correction term, we note that the upper bound in Lemma \ref{orthogonal_matrix_lemma} does not exceed that of Equation~(\ref{det_bd}), which has already been bounded.

Combining these bounds, there is a constant $C>0$ such that with  probability at least $1 - 2(n\vee d)^{-4}$,
\begin{align*}
    \inf_{\mathcal{O} \in \mathbb{O}(r)} \| \hat U - U\mathcal{O}\|_{2,\infty} \leq C  \bigg(  \frac{\sqrt{r  n d} \log (n\vee d) \sigma^{2}}{\lambda_r^2} + \frac{\sqrt{rn \log (n\vee d)} \kappa \sigma }{\lambda_r} \bigg) \|U\|_{2,\infty},
\end{align*}
as advertised.
\end{proof}

In order to prove Theorem \ref{thm:b-e}, %in Appendix \ref{sec:be_proof} we will show that %(perhaps at the cost of overloading notation) 
we show that
\begin{align}
    e_i\t \left( \hat U \mathcal{O}_* - U \right) e_j &= \langle E_i, V_{\cdot j}\rangle \lambda_j\inv + R, \label{mainresidual}
\end{align}
where $V_{\cdot j}$ is the $j$'th column of $V$ and $R$ is a residual term that we bound using similar ideas to the proof of Theorem \ref{thm:2_infty}. A straightforward calculation reveals that $\E \bigg(\langle E_i, V_{\cdot j} \rangle\lambda_j\inv \bigg)^2 = V_{\cdot j}\t \Sigma_i V_{\cdot j} \lambda_j^{-2}$, and hence if we define $\sigma_{ij} := \| \Sigma_i^{1/2} V_{\cdot j}\|\lambda_j\inv$, we have that
\begin{align*}
    \frac{\langle E_i, V_{\cdot j} \rangle \lambda_j\inv}{ \sigma_{ij}} &= \frac{ \langle E_i, V_{\cdot j} \rangle}{\| \Sigma_i^{1/2} V_{\cdot j}\|} \\
    &= \frac{ \langle Y_i, \Sigma_i^{1/2} V_{\cdot j} \rangle}{\| \Sigma_i^{1/2} V_{\cdot j}\|},
\end{align*}
which by Assumption \ref{assumption:noise} is a sum of $d$ independent mean-zero random variables to which the classical Berry-Esseen Theorem \citep{berry_accuracy_1941} can be applied.  %\textcolor{red}{
The residual term $R$ consists of higher-order terms stemming from a matrix series expansion (see Lemma \ref{lem:xia} in  Appendix \ref{sec:randompart}).  However, we have already bounded many of these residual terms as part of the proof of Theorem \ref{thm:2_infty_random}, so we need only show that dividing the residual terms by $\sigma_{ij}$ yields convergence to zero.
%}
We then use the Lipschitz property of the Gaussian cumulative distribution function to complete the proof of the theorem.

%
%\onehalfspacing

%\bibliography{CMDS.bib}

\section*{Acknowledgements}
Joshua Agterberg's work is supported by a fellowship from the Johns Hopkins Mathematical Institute of Data Science (MINDS).  Zachary Lubberts and Carey Priebe are supported through DARPA D3M under grant FA8750-17-2-011, the Naval Engineering Education Consortium (NEEC), Office of Naval Research (ONR) Award Number N00174-19-1-0011, and funding through Microsoft Research.  

%\end{document}
\appendix

\section{Proof of Theorem \ref{thm:2_infty_random}} \label{sec:randompart}First, note that it holds that
\begin{align*}
    \| \tilde U \hat H - U \tilde H \hat H\|_{2,\infty} &\leq  \| \tilde U  - U \tilde H \|_{2,\infty} \\
    &\leq \| \tilde U\tilde U\t - U U\t \tilde U \tilde U\t \|_{2,\infty} \\
    &\leq \| \tilde U\tilde U\t - UU\t \|_{2.\infty} \|\tilde U\tilde U\t\| \\
    &\leq \| \tilde U \tilde U\t - UU\t \|_{2,\infty},
\end{align*}
which is a difference of projection matrices.  We now wish to expand $\tilde U\tilde U\t$ in terms of the noise matrix $\Gamma(EM\t + ME\t + EE\t)$. 

In what follows, define, for some sufficiently large constant $C_0$
\begin{align*}
    \delta_1 :&= \sqrt{r  n d} \log (\max (n, d)) \sigma^{2}; \\
    \delta_2:&=  \sqrt{rn \log (n\vee d)} \lambda_{1} \sigma; \\
    \delta :&= C_0 (\delta_1 + \delta_2).
\end{align*}
The two terms $\delta_1$ and $\delta_2$ appear frequently in our bounds, so this notation simplifies the statements of several of our results.  With this new notation, to prove Theorem \ref{thm:2_infty_random} it is sufficient to show that 
\begin{align*}
    \| \tilde U \tilde U\t - UU\t \|_{2,\infty} &\leq C_R \frac{\delta}{\lambda_r^2} \|U\|_{2,\infty}.
\end{align*}

It is slightly more mathematically convenient to study the perturbation $EM\t + ME\t + \Gamma(EE\t)$, so we introduce the matrix $\tilde U_D \tilde U_D\t$ which is the projection onto the leading eigenspace of the matrix $MM\t + ME\t  + EM\t + \Gamma(EE\t)$.  We have the following spectral norm guarantee that shows that $\tilde U_D \tilde U_D\t$ and $\tilde U \tilde U\t$ are exceedingly close.

\begin{restatable}{lemma}{diagonals}\label{lem:diag}
 With probability at least $1-2(n\vee d)^{-5}$, 
$$\|\tilde{U}_D\tilde{U}_D\t-\tilde{U}\tilde{U}\t\|\leq \frac{C\delta_1}{\lambda_r^2}\|U\|_{2,\infty}$$
\end{restatable}

By Lemma \ref{lem:spectral_norm_concentration} and Lemma \ref{lem:diag}, we have that with probability at least $1 - c(n\vee d)^{-5}$ that
\begin{align*}
   \|EM\t + ME\t + \Gamma(EE\t)  \| &\leq \frac{\delta}{\sqrt{r\log(n\vee d)}}
\end{align*}
provided $C_0$ is sufficiently large.

In order to analyze the approximation of $\tilde U_D \tilde U_D\t$ to $UU\t$, we will use the projection matrix expansion in \citet{xia_confidence_2019}, restated slightly for our purposes here.

\begin{lemma}[Theorem 1 from \citet{xia_normal_2019}] \label{lem:xia}
Let $U$ be the eigenvectors of $A$ corresponding to its nonzero eigenvalues and let $\tilde U_D $ be the eigenvectors of $A + W$, where $\|W\| \leq \frac{\lambda_r^2}{2}$.  Then $\tilde U_D $ admits the series expansion
\begin{align*}
    \tilde U_D \tilde U_D\t &= UU\t +\sum_{p\geq 1} S_{A,k}(W),
\end{align*}
where $S_{A,p}$ is defined according to \citet{xia_normal_2019} via
\begin{align*}
    S_{A,p}(W) &= \sum_{\mathbf{s}: s_1 + \cdots + s_{p+1} = p} (-1)^{1 + \tau(\mathbf{s})} \mathcal{P}^{-s_1} W\mathcal{P}^{-s_2} W \cdots W \mathcal{P}^{-s_{p+1}},
\end{align*}
where $\mathcal{P}^{-p} = U \Lambda^{-2p} U\t$, and $\mathcal{P}^0 = U_{\perp}U_{\perp}\t$.  
\end{lemma}

By Assumption \ref{assumption:eigengap}, we have that $\lambda_r^2 \geq 2\| EM\t + ME\t + \Gamma(EE\t)\|$, and hence the assumptions to apply the expansion in Lemma \ref{lem:xia} hold. Let $W := ME\t + EM\t + \Gamma(EE\t)$.  We note that we therefore have
\begin{align*}
    \tilde U_D \tilde U_D\t - UU\t &= \sum_{p\geq 1} S_{MM\t,p}( W),
\end{align*}
Note that if $s_1 \geq 1$, $\| U \Lambda^{-2s_1} U\t W \cdots U \Lambda^{-2s_{p+1}} U\t \|_{2,\infty} \leq \| U\|_{2,\infty} \lambda_r^{-2p} \|W \|^{p}$ which is bounded above by $\delta^p \lambda_r^{-2p} \|U\|_{2,\infty}$ by Lemma \ref{lem:spectral_norm_concentration}.  Hence, it suffices to bound terms of the form
\begin{align*}
     \|  (\per W\per)^p W U \|_{2,\infty}.
\end{align*}
We have the following lemma characterizing terms of this form, whose proof is in Appendix \ref{sec:lemmas_random}.  This proof requires additional considerations about dependence which requires specially crafted ``leave-one-out" terms that have hitherto not been considered in the literature on entrywise eigenvector analysis.

\begin{restatable}{lemma}{leminftygeneral} 
\label{lem:2infty_general}
Let $W = EM\t + ME\t + \Gamma(EE\t)$.  There exists universal constants $C_1$ and $C_2$ such that for any $p \geq 1$, we have that with probability at least $1 - (p+1)(n\vee d)^{-5}$ for all $1 \leq p_0 \leq p$ that
\begin{align*}
    \|  (\per W\per)^{p_0-1} W U \|_{2,\infty} \leq C_1 (C_2\delta)^{p_0} \|U\|_{2,\infty},
\end{align*}
\end{restatable}

  Let $c$ be some number to be chosen later.  There are at most $4^p$ terms such that $s_1 + \cdots + s_{p+1} = p$, and hence
\begin{align*}
   \| \sum_{p\geq 1} S_{A,k}(W)\|_{2,\infty} &\leq \sum_{p=1}^{c\log(n\vee d)}\|  S_{A,p}(W)\|_{2,\infty} +  \sum_{p=c\log(n\vee d)}^{\infty}\|  S_{A,p}( W)\|_{2,\infty}  \\
   &\leq \| U\|_{2,\infty} \sum_{p=1}^{c\log(n\vee d)} C_1\bigg(\frac{4 C_2 \delta}{\lambda_r^2} \bigg)^p + \sum_{p=c\log(n\vee d)}^{\infty} \bigg(\frac{4\|W\|}{\lambda_r^2} \bigg)^p \\
   &\leq C \|U\|_{2,\infty} \frac{\delta}{\lambda_r^2} + \frac{1}{2^{c\log(n\vee d)}} \\
   &\leq C_R \|U\|_{2,\infty}\frac{\delta}{\lambda_r^2} ,
\end{align*}
where in the final line we have used the fact that the second term can be bounded by $2^{-c\log(n\vee d)} \leq \| U\|_{2,\infty}$ for $c$ taken to be sufficiently large.  Finally, this bound holds with probability at least $1 - c\log(n \vee d) (n\vee d)^{-5} - 4(n\vee d)^{-6} \geq 1 - (n\vee d)^{-4}$, which completes the proof of Theorem \ref{thm:2_infty_random}.

\section{Proof of Theorem \ref{thm:2_infty_deterministic} } \label{sec:deterministicpart}

First, we have the following result on the eigengap.

\begin{lemma}\label{lem:eigengaps}
Suppose Assumption~\ref{assumption:eigengap} holds, and suppose $T\geq T_0$, where $T_0$ is the first iterate such that $\|N_T-A\|\leq 3\|\Gamma(Z)\|$. Then on the event in Lemma \ref{lem:spectral_norm_concentration}, we have
\begin{align*}
    (\hat \lambda_r^{(T)})^2 - \tilde \lambda_{r+1}^2 \geq \lambda_r^2/2.
\end{align*}
In particular, 
\begin{align*}
    \hat \lambda_r^2 - \tilde \lambda_{r+1}^2 \geq \lambda_r^2/2.
\end{align*}
\end{lemma}

\begin{proof}
$\hat{\lambda}_r^2\geq \tilde{\lambda}_r^2-\|N_T-\tilde{A}\|\geq \tilde{\lambda}_r^2-4\|\Gamma(Z)\|$ once $T\geq T_0$. Then $\hat{\lambda}_r^2-\tilde{\lambda}_{r+1}^2\geq \tilde{\lambda}_r^2-\tilde{\lambda}_{r+1}^2-4\|\Gamma(Z)\|\geq \lambda_r^2-\lambda_{r+1}^2-6\|\Gamma(Z)\|=\lambda_r^2-6\|\Gamma(Z)\|$. On the event in  Lemma~\ref{lem:spectral_norm_concentration}, $\|\Gamma(Z)\|\leq C_{\mathrm{spectral}}(\sigma^2(n+\sqrt{nd})+\sigma\sqrt{n}\kappa\lambda_r),$ and under the signal-to-noise ratio condition of Assumption~\ref{assumption:eigengap}, $$\lambda_r^2\geq 12C_{\mathrm{spectral}}\left(\sigma^2(n+\sqrt{nd})+\sigma\sqrt{n}\kappa \lambda_r\right),$$ this gives $$\|\Gamma(Z)\|\leq \frac{1}{12}\lambda_r^2,$$ meaning that $$\hat{\lambda}_r^2-\tilde{\lambda}_{r+1}^2\geq \lambda_r^2/2.$$
\end{proof}

We can now prove Theorem \ref{thm:2_infty_deterministic}.

\begin{proof}[Proof of Theorem \ref{thm:2_infty_deterministic}]
We write:
\begin{align*}
    \hat U  - \tilde U \tilde U\t \hat U  &= P_{\tilde U}\tilde A[\hat U - \tilde U \tilde H] \hat \Lambda\inv + [N_T -  P_{\tilde U}\tilde A][ \hat U - \tilde U \tilde H] \hat \Lambda\inv \\&\qquad +  [N_T -  P_{\tilde U}\tilde A]\tilde U \tilde H \hat \Lambda\inv - \tilde U[ \tilde H \hat \Lambda- \tilde \Lambda \tilde H]\hat \Lambda\inv \\
    :&= J_1 + J_2 + J_3 + J_4.
\end{align*}
We bound each term successively.

\ \\ \noindent \textbf{The term $J_1$:} 
Note that $\|\hat{U}-\tilde{U}\tilde{H}\|\leq\sqrt{2}\|\sin\Theta(\hat{U},\tilde{U})\|.$  Therefore, we have that 
\begin{align}
    \| P_{\tilde U}\tilde A[ \hat U - \tilde U \tilde H]\hat \Lambda\inv \|_{2,\infty} &\leq   \| P_{\tilde U}\tilde A \|_{2,\infty }  \| \hat U - \tilde U \tilde H \| \|\hat \Lambda\inv \| \nonumber \\
    &\leq  \| \tilde U \tilde \Lambda \tilde U \|_{2,\infty}  \| \hat U - \tilde U \tilde H \| \hat \lambda_{r}^{-2} \nonumber \\
    &\leq \frac{\tilde \lambda_1^2}{\hat \lambda_r^2}   \| \tilde U \|_{2,\infty}  \| \hat U - \tilde U \tilde H \| \nonumber \\ %\\
    %&\leq 4  \kappa \| \tilde U \|_{2,\infty}   \| \hat U - \tilde U \tilde H \| \\
    %    &\leq 4 \kappa   \| \tilde U \|_{2,\infty}^2 
%\frac{8 b \tilde \lambda_{r+1}}{\tilde \lambda_r - \hat\lambda_{r+1}} \\
%    &\leq 64 \kappa b \| \tilde U \|_{2,\infty}^2 \frac{\tilde \lambda_{r+1}}{\lambda_r} \\
%    &\leq 64 \kappa b \| \tilde U \|_{2,\infty}^2\frac{\| \Gamma(Z) \|}{\lambda_r}.
&\leq \sqrt{2}\frac{\tilde \lambda_1^2}{\hat \lambda_r^2}   \| \tilde U \|_{2,\infty} \|\sin\Theta(\hat{U},\tilde{U})\|. \label{j1bound}
\end{align}
%by Lemma \ref{utilde_approx_lem}, Lemma \ref{lem:eigengaps}, and the fact that $\tilde \lambda_1/\hat\lambda_r \leq 4\kappa$.  

\ \\ \noindent \textbf{The term $J_2$:}
We decompose via
\begin{align*}
[N_T-P_{\tilde{U}}\tilde{A}][\hat{U}-\tilde{U}\tilde{H}]\hat{\Lambda}^{-1}&=\hat{U}[\hat{\Lambda}\hat{U}\t-\tilde{H}\t\tilde{\Lambda}\tilde{U}\t][\hat{U}-\tilde{U}\tilde{H}]\hat{\Lambda}^{-1}+[\hat{U}-\tilde{U}\tilde{H}]\tilde{H}\t\tilde{\Lambda}\tilde{U}\t[\hat{U}-\tilde{U}\tilde{H}]\hat{\Lambda}^{-1}\\
&\quad-\tilde{U}[I-\tilde{H}\tilde{H}\t]\tilde{\Lambda}\tilde{U}\t[\hat{U}-\tilde{U}\tilde{H}]\hat{\Lambda}^{-1}.
\end{align*}
Note that $\tilde{U}\t[\hat{U}-\tilde{U}\tilde{H}]=0$, by the definition of $\tilde{H}$, whence we have 
\begin{align*}
[N_T-P_{\tilde{U}}\tilde{A}][\hat{U}-\tilde{U}\tilde{H}]\hat{\Lambda}^{-1} &=\tilde{U}\tilde{H}[\hat{\Lambda}\hat{U}\t-\tilde{H}\t\tilde{\Lambda}\tilde{U}\t][\hat{U}-\tilde{U}\tilde{H}]\hat{\Lambda}^{-1}+[\hat{U}-\tilde{U}\tilde{H}][\hat{\Lambda}\hat{U}\t-\tilde{H}\t\tilde{\Lambda}\tilde{U}\t][\hat{U}-\tilde{U}\tilde{H}]\hat{\Lambda}^{-1}\\
&=\tilde{U}\tilde{H}\hat{\Lambda}[\hat{U}-\tilde{U}\tilde{H}]\t[\hat{U}-\tilde{U}\tilde{H}]\hat{\Lambda}^{-1}+[\hat{U}-\tilde{U}\tilde{H}]\hat{\Lambda}[I-\tilde{H}\t\tilde{H}]\hat{\Lambda}^{-1}.
\end{align*}
Taking norms, we have 
\begin{align*}
    \|[N_T-P_{\tilde{U}}\tilde{A}][\hat{U}-\tilde{U}\tilde{H}]\hat{\Lambda}^{-1}\|_{2,\infty}\leq \|\tilde{U}\|_{2,\infty}\hat{\lambda}_1^2\|\hat{U}-\tilde{U}\tilde{H}\|^2\hat{\lambda}_r^{-2}+\|\hat{U}-\tilde{U}\tilde{H}\|_{2,\infty}\hat{\lambda}_1^2\|I-\tilde{H}\t\tilde{H}\|\hat{\lambda}_r^{-2}
\end{align*}
Recall that $\|I-\tilde{H}\t\tilde{H}\|\leq \|\hat{U}\hat{U}\t-\tilde{U}\tilde{U}\t\|\leq 2\|\sin\Theta(\hat{U},\tilde{U})\|$, and $\|\hat{U}-\tilde{U}\tilde{H}\|\leq \sqrt{2}\|\sin\Theta(\hat{U},\tilde{U})\|.$  Consequently,
\begin{align}
    \| J_2 \|_{2,\infty} &\leq  2 \|\tilde{U}\|_{2,\infty} \frac{\hat{\lambda}_1^2}{\hat{\lambda}_r^{2}}\|\sin\Theta(\hat U, \tilde U)\|^2+2 \|\hat{U}-\tilde{U}\tilde{H}\|_{2,\infty}\frac{\hat{\lambda}_1^2}{\hat{\lambda}_r^{2}} \|\sin\Theta(\hat U, \tilde U \|. \label{j2bound}
\end{align}

\ \\ \noindent \textbf{The term $J_3$:} Again using the fact that $\tilde{U}\t[\hat{U}-\tilde{U}\tilde{H}]=0$,
\begin{align*}
[N_T-P_{\tilde{U}}\tilde{A}][\hat{U}-\tilde{U}\tilde{H}]\hat{\Lambda}^{-1}&=\hat{U}\hat{\Lambda}[I-\tilde{H}\t\tilde{H}]\hat{\Lambda}^{-1}\\
&=\tilde{U}\tilde{H}\hat{\Lambda}[I-\tilde{H}\t\tilde{H}]\hat{\Lambda}^{-1}+[\hat{U}-\tilde{U}\tilde{H}]\hat{\Lambda}[I-\tilde{H}\t\tilde{H}]\hat{\Lambda}^{-1}.
\end{align*}

This gives 
\begin{align}
    \|[N_T-P_{\tilde{U}}\tilde{A}][\hat{U}-\tilde{U}\tilde{H}]\hat{\Lambda}^{-1}\|_{2,\infty}&\leq \|\tilde{U}\|_{2,\infty}\hat{\lambda}_1^2\|I-\tilde{H}\t\tilde{H}\|\hat{\lambda}_r^{-2}+\|\hat{U}-\tilde{U}\tilde{H}\|_{2,\infty}\hat{\lambda}_1^2\|I-\tilde{H}\t \tilde{H}\|\hat{\lambda}_r^{-2} \nonumber \\
    &\leq 2 (\|\tilde{U}\|_{2,\infty}+\|\hat{U}-\tilde{U}\tilde{H}\|_{2,\infty}) \frac{\hat{\lambda}_1^2}{\hat{\lambda}_r^2} \|\sin \Theta(\hat{U},\tilde{U})\|. \label{j3bound}
\end{align}

\ \\ \noindent \textbf{The term $J_4$:} In $2,\infty$ norm, we note that
\begin{align*}
  \|  \tilde U[ \tilde H \hat \Lambda- \tilde \Lambda \tilde H]\hat \Lambda\inv \|_{2,\infty} &\leq  \frac{\|  \tilde U \|_{2,\infty}}{\hat \lambda_r^2} \|\tilde H \hat \Lambda- \tilde \Lambda \tilde H\|.
\end{align*}
Furthermore,
\begin{align*}
    \| \tilde H \hat \Lambda -\tilde \Lambda \tilde H \| &= \| \tilde U\t[ N_T - \tilde A] \hat U\| \\
    &\leq  \| \tilde U\t[ N_T - \tilde A] \|.
\end{align*}
Consequently,
\begin{align}
    \|J_4 \| &\leq \| \tilde U \|_{2,\infty} \frac{1}{\hat \lambda_r^2} \| \tilde U\t [ N_T - \tilde A] \|. \label{j4bound}
\end{align}

\ \\ \noindent \textbf{Putting it together:}
Collecting the bounds in \eqref{j1bound},\eqref{j2bound},\eqref{j3bound}, and \eqref{j4bound}, we see that 
\begin{align*}
\|\hat{U}-\tilde{U}\tilde{H}\|_{2,\infty}&\leq \frac{\tilde{\lambda}_1^2}{\hat{\lambda}_r^2}\|\tilde{U}\|_{2,\infty}\left[\sqrt{2}\|\sin\Theta(\hat{U},\tilde{U})\|+\frac{\hat{\lambda}_1^2}{\tilde{\lambda}_1^2}(2\|\sin\Theta(\hat{U},\tilde{U})\|^2+2\|\sin(\hat{U},\tilde{U})\|)+\frac{\tilde{K}_T}{\tilde{\lambda}_1^2}\right]\\
&\quad+\|\hat{U}-\tilde{U}\tilde{H}\|_{2,\infty}\frac{4\hat{\lambda}_1^2}{\hat{\lambda}_r^2}\|\sin\Theta(\hat{U},\tilde{U})\|,
\end{align*}
where $\tilde K_T := \|N_T - \tilde A \|$.   Applying Davis-Kahan, Lemma~\ref{lem:deterministic_spectral}, and Lemma~\ref{lem:eigengaps}, we see that 
\begin{align*}
    \|\sin\Theta(\hat{U},\tilde{U})\|&\leq \frac{\|\hat{A}-\tilde{A}\|}{\hat{\lambda}_r^2-\tilde{\lambda}_{r+1}^2} \\
    &\leq \frac{41\|U\|_{2,\infty}\|\Gamma(Z)\|}{\lambda_r^2/2}=:\tau.
\end{align*}
By Theorem~\ref{thm:2_infty_random}, for large enough $n$, $\|\tilde{U}\|_{2,\infty}\leq \|U\|_{2,\infty}$ since the bound in Theorem~\ref{thm:2_infty_random} is of the form $c \|U\|_{2,\infty}$, where $c < 1$ for $n$ sufficiently large.
This gives
$$
\|\hat{U}-\tilde{U}\tilde{H}\|_{2,\infty}\leq 2\frac{\tilde{\lambda}_1^2}{\hat{\lambda}_r^2}\|U\|_{2,\infty}\left[\sqrt{2}\tau+\frac{\hat{\lambda}_1^2}{\tilde{\lambda}_1^2}\left(2\tau^2+2\tau\right)+\frac{\tau \lambda_r^2/2}{\tilde{\lambda}_1^2}\right]+\|\hat{U}-\tilde{U}\tilde{H}\|_{2,\infty}\frac{4\hat{\lambda}_1^2}{\hat{\lambda}_r^2}\tau.
$$
The proof of Lemma~\ref{lem:eigengaps} reveals that $(11/12)\lambda_1^2\leq\tilde{\lambda}_1^2\leq (13/12)\lambda_1^2$, $(3/4)\lambda_r^2\leq \hat{\lambda}_r^2 \leq (5/4) \lambda_r^2$, with the same bounds holding for $\hat{\lambda}_1$, also. Thus $\tilde{\lambda}_1^2/\hat{\lambda}_r^2\leq (13/9)\kappa^2$, $\hat{\lambda}_1^2/\tilde{\lambda}_1^2\leq 15/11$, and $\hat{\lambda}_1^2/\hat{\lambda}_r^2\leq (5/3)\kappa^2$. This gives

$$
\|\hat{U}-\tilde{U}\tilde{H}\|_{2,\infty}\leq 2(13/9)\kappa^2\|U\|_{2,\infty}\left[\sqrt{2}\tau+(15/11)\left(2\tau^2+2\tau\right)+\frac{\tau \lambda_r^2/2}{(11/12)\lambda_1^2}\right]+\|\hat{U}-\tilde{U}\tilde{H}\|_{2,\infty}4(5/3)\kappa^2\tau.
$$

When $(20/3)\kappa^2 \tau\leq 1/2$, which occurs for $n$ sufficiently large under the event in Theorem \ref{thm:2_infty_random}, this gives

\begin{align*}
\|\hat{U}-\tilde{U}\tilde{H}\|_{2,\infty}&\leq \kappa^2\frac{500\|U\|_{2,\infty}^2\|\Gamma(Z)\|}{\lambda_r^2}\left[\sqrt{2}+(15/11)\left(2+\frac{3}{20\kappa^2}\right)+\frac{6}{11\kappa^2}\right]\\&=C_D \kappa^2 \frac{\|U\|_{2,\infty}^2\|\Gamma(Z)\|}{\lambda_r^2}
\end{align*}
as required.
% Now rearranging gives the final inequality
% $$\|\hat{U}-\tilde{U}\tilde{H}\|_{2,\infty}\leq \frac{\tilde{\lambda}_1}{\hat{\lambda}_r}\|\tilde{U}\|_{2,\infty}\left[\sqrt{2}\|\sin\Theta(\hat{U},\tilde{U})\|(1+\sqrt{2}(\hat{\lambda}_1/\tilde{\lambda}_1)(1+\|\sin\Theta(\hat{U},\tilde{U})\|))+\frac{K_T}{\tilde{\lambda}_1}\right]\left[1-\frac{4\hat{\lambda}_1}{\hat{\lambda}_r}\|\sin\Theta(\hat{U},\tilde{U})\|\right]^{-1}.$$
% Using $\|\sin\Theta(\hat{U},\tilde{U})\|\leq 1$ yields the stated upper bound.
\end{proof}

\section{Proof of Theorem \ref{thm:b-e}}
\label{sec:be_proof}
First, we justify Equation \eqref{mainresidual}.  Note that by Theorem \ref{thm:2_infty}, on that event $\|\hat U\|_{2,\infty} \leq C \|U\|_{2,\infty}$, which implies that $\hat U$ is as incoherent as $U$ up to constant factors.   Now following similarly in the first part of the proof Theorem \ref{thm:2_infty}, we have that for the same orthogonal matrix $\mathcal{O}_*$ as in Lemma \ref{orthogonal_matrix_lemma} that
\begin{align}
    e_i\t\bigg( \hat U \mathcal{O}_* - U \bigg) e_j =% e_i\t \hat U( W - \hat U\t \tilde U\tilde U\t U )e_j + e_i\t \bigg( \hat U \hat U\t \tilde U\tilde U\t  U - U\bigg) e_j \\
      e_i\t \hat U( \mathcal{O}_* - \hat U\t \tilde U\tilde U\t U )e_j + e_i\t \bigg( \hat U \hat U\t \tilde U\tilde U\t  U - \tilde U \tilde U\t U\bigg) e_j  + e_i\t \bigg( \tilde U \tilde U\t U - U \bigg) e_j. \label{hatexp}
\end{align}
As in the proof of Theorem \ref{thm:2_infty_random} (see Appendix \ref{sec:lemmas_random}), let $\tilde U_D $ be the matrix of eigenvectors of $MM\t + EM\t + ME\t + \Gamma(EE\t)$, and let $W:= EM\t + ME\t + \Gamma(EE\t)$.  

Now we again apply Lemma \ref{lem:xia} to $\tilde U_D \tilde U_D\t$.  First, recall the definition of $S_{MM\t,1}(W) = \per W U\Lambda^{-2}U\t + U\t \Lambda^{-2}U\t W \per$, and note that $S_{MM\t,1}(W)U = \per W U\Lambda^{-2}$.  Now, just as in the proof of Theorem \ref{thm:2_infty_random} we expand $\tilde U_D \tilde U_D\t$ as an infinite series in $W$ via
 \begin{align}
    e_i\t \bigg(\tilde U\tilde U\t U - U\bigg) e_j &=   e_i\t \bigg( \tilde U_D \tilde U_D\t U - U \bigg) e_j - e_i\t \bigg( \tilde U_D \tilde U_D\t - \tilde U \tilde U\t\bigg)Ue_j \nonumber\\
    &=  e_i\t S_{MM\t,1}(W) U e_j + e_i\t \sum_{k\geq 2} S_{MM\t,k}(W) Ue_j  -e_i\t \bigg( \tilde U_D \tilde U_D\t - \tilde U \tilde U\t\bigg)Ue_j \nonumber\\
    &=  e_i\t \per \bigg(EM\t + ME\t + \Gamma(EE\t)\bigg)\Lambda^{-2} e_j + e_i\t \sum_{k\geq 2} S_{MM\t,k}(W) Ue_j \nonumber \\
   &\qquad - e_i\t \bigg( \tilde U_D \tilde U_D\t - \tilde U \tilde U\t\bigg)Ue_j \nonumber\\
   &= e_i\t  \bigg(EM\t + \Gamma(EE\t)\bigg) U \Lambda^{-2} e_j  - e_i\t UU\t\bigg(EM\t + \Gamma(EE\t)\bigg) U \Lambda^{-2} e_j \nonumber\\
    &\qquad + e_i\t \sum_{k\geq 2} S_{MM\t,k}(W) Ue_j - e_i\t \bigg( \tilde U_D \tilde U_D\t - \tilde U \tilde U\t \bigg)Ue_j  \label{tildeexp}
\end{align}
where in the penultimate line we used the fact that $\per M = 0$.  Hence, plugging the expansion in \eqref{tildeexp} into \eqref{hatexp}, we see that
\begin{align*}
     e_i\t\bigg( \hat U \mathcal{O}_* - U \bigg) e_j &= e_i\t  EM\t U\Lambda^{-2} e_j  + e_i\t \Gamma(EE\t)U\Lambda^{-2} e_j - e_i\t UU\t \bigg( EM\t + \Gamma(EE\t) \bigg) U\Lambda^{-2} e_j  \\
     &\qquad  + e_i\t \sum_{k\geq 2} S_{MM\t,k}(W) Ue_j - e_i\t \bigg( \tilde U_D \tilde U_D\t - \tilde U \tilde U\t \bigg)Ue_j \\
     &\qquad +  e_i\t \hat U( \mathcal{O}_* - \hat U\t \tilde U\tilde U\t U )e_j + e_i\t \bigg( \hat U \hat U\t \tilde U\tilde U\t  U - \tilde U \tilde U\t U\bigg) e_j \\
     :&= e_i\t EM\t U\Lambda^{-2} e_j + R_0 + R_1 + R_2 + R_3 + R_4 + R_5,
\end{align*}
where
\begin{align*}
    R_0 :&= e_i\t \Gamma(EE\t) U\Lambda^{-2} e_j; \\
    R_1 :&= e_i\t U\bigg( EM\t + \Gamma(EE\t) \bigg) U\Lambda^{-2} e_j  ;\\
    R_2: &=e_i\t \bigg( \tilde U_D \tilde U_D\t - \tilde U \tilde U\t \bigg)Ue_j; \\
    R_3 :&=  e_i\t \sum_{k\geq 2} S_{MM\t,k}(W) Ue_j  ;\\
    R_4:&=   e_i\t \hat U( \mathcal{O}_* - \hat U\t \tilde U\tilde U\t U )e_j  ;\\
    R_5:&=  e_i\t \bigg( \hat U \hat U\t \tilde U\tilde U\t  U - \tilde U \tilde U\t U\bigg) e_j.
\end{align*}
We now characterize the residual terms. 
\begin{restatable}{lemma}{residualsrand} \label{lem:residuals1}
There exist universal constants $C_{6}$ and $C_7$ such that the residual terms $R_1$,$R_2$,  and $R_3$ satisfy, uniformly over $i$ and $j$,
\begin{align*}
\frac{1}{\sigma_{ij}} \bigg(|R_1| + |R_2| + |R_3| \bigg) 
&\leq C_{6} \kappa_{\sigma} \kappa^2 \mu_0 \sqrt{\frac{r\log(n\vee d)}{n}}+C_7 \kappa^3 \kappa_{\sigma} \mu_0 \frac{r\log(n\vee d)}{\snr}
\end{align*}
with probability at least $1 - 5(n\vee d)^{-4}$.
\end{restatable}

\begin{restatable}{lemma}{residualsdet}\label{lem:residuals2}
On the intersection of the events in Theorem \ref{thm:2_infty} and Lemma \ref{lem:spectral_norm_concentration} the residual terms $R_4$ and $R_5$ satisfy for all $i$ and $j$,
\begin{align*}
\frac{1}{\sigma_{ij}}  \bigg( | R_4 | + |R_5| \bigg)  &\leq  C_8 \kappa^3\kappa_{\sigma} \mu_0  \frac{1}{\snr} + C_9  \kappa^4 \kappa_{\sigma} \mu_0^2 \frac{r}{\sqrt{n}}.
\end{align*}
for some universal constants $C_{8}$ and $C_{9}$.
\end{restatable}

To bound $R_0$, we note that we can equivalently write
\begin{align*}
R_0 :=    \sum_{k\neq i} \langle E_i, E_k \rangle (U \Lambda^{-2})_{kj} = \sum_{k\neq i} \langle E_i, E_k \rangle U_{kj} \lambda_j^{-2}.  
\end{align*}
We have the following bound for $R_0$. 

\begin{restatable}{lemma}{finalres}\label{lem:final_res}
There exists a universal constant $C_{10}$ such that with probability at least $1 - 4(n\vee d)^{-4}$
\begin{align*}
    \frac{1}{\sigma_{ij}}\bigg| \sum_{k\neq i} \langle E_i, E_k \rangle U_{kj} \lambda_j^{-2} \bigg|  &\leq C_{10} \mu_0 \kappa_{\sigma} \frac{\log(n\vee d)}{\snr},
\end{align*}
where the probability is uniform over $i$ and $j$.
\end{restatable}

Let $R := R_0 + R_1 + R_2 + R_3 + R_4 + R_5$.  This argument leaves us with
\begin{align*}
   e_i\t(\hat U\mathcal{O}_* - U) e_j &=  e_i\t EM\t U\Lambda^{-2} e_j + R.
\end{align*}
Note in addition that  $M\t U\Lambda^{-2} = V\Lambda U\t U \Lambda^{-2} =  V\Lambda^{-1}$ by definition.  Hence, the term $e_i\t EM\t U\Lambda^{-2} e_j$ can be equivalently written as $\langle E_i, V_{\cdot j}\rangle \lambda_j\inv$
where $V_{\cdot j}$ is the $j$'th column of the matrix $V$, which justifies equation \eqref{mainresidual}.

Now, combining the bounds for the residuals in Lemmas \ref{lem:residuals1}, \ref{lem:residuals2}, and \ref{lem:final_res}, we see that with probability at least $1 - 10(n\vee d)^{-4} - 4(n\vee d)^{-6} \geq 1 - (n\vee d)^{-3}$,
\begin{align*}
    \frac{|R|}{\sigma_{ij}} &\leq C_6 \kappa_{\sigma} \kappa^2 \mu_0 \sqrt{\frac{r\log(n\vee d)}{n}} + C_7 \kappa^3 \kappa_{\sigma} \mu_0 \frac{r\log(n\vee d)}{\snr} + C_8 \kappa^3 \kappa_{\sigma} \mu_0 \frac{1}{\snr} \\
    &\qquad + C_9 \kappa^4\kappa_{\sigma} \mu_0^2 \frac{r}{\sqrt{n}} + C_{10} \mu_0 \kappa_{\sigma} \frac{\log(n\vee d)}{\snr}  \\
    &\leq \tilde C_1 \kappa^3 \kappa_{\sigma} \mu_0 \frac{r\log(n\vee d)}{\snr} + \tilde C_2 \kappa^2 \kappa_{\sigma} \mu_0 \sqrt{\frac{r}{n}} \bigg(  \sqrt{\log(n\vee d)} + \mu_0 \kappa^2 \sqrt{r} \bigg) \\
    =&: B.
\end{align*}

We can now complete the proof of Theorem \ref{thm:b-e}. By the classical Berry-Esseen Theorem \citep{berry_accuracy_1941}, for any $x \in \R$, denoting $Y_{i\alpha}$ as the $\alpha$ entry of $Y_i$, 
\begin{align*}
    \bigg|\p\bigg( \frac{\langle E_i, V_{\cdot j} \rangle}{\|\Sigma_i^{1/2} V_{\cdot j}\|} > x \bigg) - \Phi(x) \bigg| &\leq C \frac{\sum_{\alpha} | (\Sigma_i^{1/2}V_{\cdot j})_\alpha|^3 \E | Y_{i\alpha}|^{3}}{\|\Sigma_i^{1/2}V_{\cdot j}\|^3} \\
&\leq C \frac{ \|\Sigma_i^{1/2} V_{\cdot j}\|_3^3}{\|\Sigma_i^{1/2}V_{\cdot j}\|^3}.
\end{align*}
Hence, by the Lipchitz property of $\Phi$, (e.g. \citet{xia_normal_2019})
\begin{align*}
\p\bigg( \frac{1}{\sigma_{ij}} e_i\t \bigg( \hat U\mathcal{O}_* - U\bigg)e_j \leq x \bigg) &\leq \p\Bigg\{\frac{\langle E_i, V_{\cdot j} \rangle}{\|\Sigma_i^{1/2} V_{\cdot j}\|} \leq x + B \bigg\} + (n\vee d)^{-3} \\
&\leq  \Phi(x + B) + C \frac{ \|\Sigma_i^{1/2} V_{\cdot j}\|_3^3}{\|\Sigma_i^{1/2}V_{\cdot j}\|^3} + (n\vee  d)^{-3} \\
&\leq \Phi(x) +  C \frac{ \|\Sigma_i^{1/2} V_{\cdot j}\|_3^3}{\|\Sigma_i^{1/2}V_{\cdot j}\|^3} + B + (n\vee d)^{-3}.
\end{align*}
A similar bound for the left tail also holds.  %Let $t = 4 \log(n\vee d)$.  Then we bound each term after the the $3-$norm bound.  Note that with this choice,
 Therefore, after relabeling constants and noting that $\kappa^2 \kappa_{\sigma} \sqrt{r/n} \geq (n\vee d)^{-3}$, we conclude that
 \begin{align*}
     \sup_{x\in \R} \bigg| \p\bigg( \frac{1}{\sigma_{ij}} e_i\t \bigg(\hat U\mathcal{O}_* - U\bigg) e_j \leq x \bigg) - \Phi(x) \bigg| &\leq C \frac{ \|\Sigma_i^{1/2} V_{\cdot j}\|_3^3}{\|\Sigma_i^{1/2}V_{\cdot j}\|^3} + B + (n\vee d)^{-3} \\
     &\leq C_1 \frac{ \|\Sigma_i^{1/2} V_{\cdot j}\|_3^3}{\|\Sigma_i^{1/2}V_{\cdot j}\|^3} + C_2 \kappa^3 \kappa_{\sigma} \mu_0 \frac{r\log(n\vee d)}{\snr} \\
     &\qquad + C_3 \kappa^2 \kappa_{\sigma} \mu_0 \sqrt{\frac{r}{n}} \bigg(  \sqrt{\log(n\vee d)} + \mu_0 \kappa^2 \sqrt{r} \bigg).
 \end{align*}

 \section{Proof of Corollaries in Section \ref{sec:applications}}
 \begin{proof}[Proof of Corollary \ref{cor:mixture}]
By the proof of Theorem \ref{thm:b-e}, we have that
\begin{align*}
    e_i\t \bigg( \hat U \mathcal{O}_* - U \bigg)e_j &= \langle E_i, V_{\cdot j} \rangle \lambda_j\inv + R_{ij},
\end{align*}
Hence. the $i$'th row of $\hat U$ satisfies
\begin{align*}
    e_i\t \bigg( \hat U \mathcal{O}_* - U \bigg) &= E_i\t V \Lambda\inv + R_{i}.
\end{align*}
We now analyze $(S_i)^{-1/2} R_i$.  However, by Lemmas \ref{lem:residuals1}, \ref{lem:residuals2}, and \ref{lem:final_res}, we see that $R_{ij}$ satisfies with probability at least $1 - (n\vee d)^{-3}$
\begin{align*}
    |R|_{ij} &\leq C_2 \sigma_{ij}  \kappa^3 \kappa_{\sigma} \mu_0 \frac{r \log(n\vee d)}{\snr} + C_3 \sigma_{ij}\kappa^2 \kappa_{\sigma} \mu_0 \sqrt{\frac{r}{n}} \bigg( \sqrt{\log(n\vee d)} + \mu_0 \kappa^2 \sqrt{r}\bigg).
\end{align*}
In addition, note that
\begin{align*}
    \|S_{i}^{-1/2}\| \sigma_{ij} &\leq \kappa \kappa_{\sigma}.
\end{align*}
Therefore,
\begin{align*}
    S_i^{-1/2} R_i \to 0
\end{align*}
in probability (and almost surely) as $n$ and $d $ tend to infinity, since $\|S_i^{-1/2}\| \sigma_{ij} = O(1)$ when $\kappa$ and $\kappa_{\sigma}$ are bounded.  
Furthermore, we note that
\begin{align*}
    \E  \bigg(\langle E_i, V_{\cdot j} \rangle \langle E_i, V_{\cdot k} \rangle \lambda_j\inv \lambda_k\inv \bigg) &= (S_i)_{jk}.
\end{align*}
Hence, the result holds by the Cramer-Wold device and Slutsky's Theorem.    
\end{proof}

\begin{proof}[Proof of Corollary \ref{cor:consistent_estimate}]
 Without loss of generality assume that $C_k = \{1, \cdots, n_k\}$, or else reorder the matrix.  Furthermore, we assume that the set of indices for community $k$ is known; under the assumptions for Theorem \ref{thm:2_infty} this will be true for sufficiently large $n,d$ since $\| \hat U - U \mathcal{O}_* \| \ll \|U\|_{2,\infty}$ and each row of $U$ reveals the community memberships by Lemma 2.1 of \citet{lei_consistency_2015}.

In what follows, recall that %for convenience we consider 
$\Lambda\inv V\t E_i$ is an $r$-dimensional column vector and $E_i\t V \Lambda\inv$ is a row vector.  For convenience, we let $\bar U^{(k)}$ denote the rank one matrix whose rows are all just $\bar U^{(k)}$ .  
By the expansion in the proof Theorem \ref{thm:b-e}, we have that
\begin{align*}
    \frac{1}{n_k} \sum_{i=1}^{n_k} \bigg(\hat U  - \bar{U}^{(k)} \bigg)_{i\cdot} \bigg(\hat U  - \bar{U}^{(k)} \bigg)_{i\cdot}\t &=  \frac{1}{n_k} \sum_{i=1}^{n_k} \bigg(\hat U  - U\mathcal{O}_*\t \bigg)_{i\cdot} \bigg(\hat U  - U\mathcal{O}_*\t \bigg)_{i\cdot}\t \\
    &\qquad +  \frac{1}{n_k} \sum_{i=1}^{n_k} \bigg(\hat U  - U\mathcal{O}_*\t \bigg)_{i\cdot}\bigg(U\mathcal{O}_*\t - \bar{U}^{(k)} \bigg)_{i\cdot}\t \\
    &\qquad +  \frac{1}{n_k} \sum_{i=1}^{n_k}\bigg(U\mathcal{O}_*\t - \bar{U}^{(k)} \bigg)_{i\cdot} \bigg(\hat U  - U\mathcal{O}_*\t \bigg)_{i\cdot}\t\\
    &\qquad + \frac{1}{n_k} \sum_{i=1}^{n_k} \bigg(\bar{U}^{(k)}  - U \mathcal{O}_*\t \bigg)_{i\cdot} \bigg(\bar{U}^{(k)}  - U\mathcal{O}_*\t \bigg)_{i\cdot}\t \\
    % &= \frac{1}{n_k} \sum_{i=1}^{n_k} \bigg(  \mathcal{O}_*(\Lambda\inv V\t E_i) +\mathcal{O}_* R_i \bigg)\bigg(  \mathcal{O}_*(\Lambda\inv V\t E_i) +\mathcal{O}_* R_i \bigg)\t \\
    % &\qquad +  \frac{1}{n_k} \sum_{i=1}^{n_k}\bigg(U\mathcal{O}_*\t - \bar{U}^{(k)} \bigg)_{i\cdot} \bigg(\hat U  - U\mathcal{O}_*\t \bigg)_{i\cdot}\t\\
    % &\qquad + \frac{1}{n_k} \sum_{i=1}^{n_k} \bigg(\bar{U}^{(k)}  - U \mathcal{O}_*\t \bigg)_{i\cdot} \bigg(\bar{U}^{(k)}  - U\mathcal{O}_*\t \bigg)_{i\cdot}\t \\
    % &\qquad + \frac{1}{n_k} \sum_{i=1}^{n_k} \bigg(\bar{U}^{(k)}  - U \mathcal{O}_*\t \bigg)_{i\cdot} \bigg(\bar{U}^{(k)}  - U\mathcal{O}_*\t \bigg)_{i\cdot}\t \\
  %  &= \frac{1}{n_k} \sum_{i=1}^{n_k} \mathcal{O}_* (\Lambda\inv V\t E_i E_i\t V \Lambda\inv)  \mathcal{O}_*\t \\
  %  &\qquad + \frac{1}{n_k} \sum_{i=1}^{n_k} \mathcal{O}_* R_i (E_i\t V \Lambda\inv)  \mathcal{O}_*\t  + \frac{1}{n_k} \sum_{i=1}^{n_k} \mathcal{O}_* (\Lambda\inv  V\t E_i )R_i  \mathcal{O}_*\t \\
   % &\qquad + \frac{1}{n_k} \sum_{i=1}^{n_k} \mathcal{O}_* R_i R_i\t \mathcal{O}_*\t \\
  %  &\qquad + \frac{1}{n_k} \sum_{i=1}^{n_k} \bigg(\bar{U}^{(k)} - U\mathcal{O}_*\t  \bigg)_{i\cdot} \bigg(\bar{U}^{(k)}  - U\mathcal{O}_*\t  \bigg)_{i\cdot}\t \\
    :&=J_1 +J_2 + J_3 %+ J_4,
\end{align*}
where
\begin{align*}
    J_1 :&=  \frac{1}{n_k} \sum_{i=1}^{n_k} \bigg(\hat U  - U\mathcal{O}_*\t \bigg)_{i\cdot}\bigg(\hat U  - U\mathcal{O}_*\t \bigg)_{i\cdot}\t; \\ 
   J_2 :&= \frac{1}{n_k} \sum_{i=1}^{n_k} \bigg(\bar{U}^{(k)}  - U \mathcal{O}_*\t \bigg)_{i\cdot} \bigg(\bar{U}^{(k)}  - U\mathcal{O}_*\t \bigg)_{i\cdot}\t; \\
    J_3 :&=  \frac{1}{n_k} \sum_{i=1}^{n_k} \bigg(\hat U  - U\mathcal{O}_*\t \bigg)_{i\cdot}\bigg(U\mathcal{O}_*\t - \bar{U}^{(k)} \bigg)_{i\cdot}\t  +  \frac{1}{n_k} \sum_{i=1}^{n_k}\bigg(U\mathcal{O}_*\t - \bar{U}^{(k)} \bigg)_{i\cdot} \bigg(\hat U  - U\mathcal{O}_*\t \bigg)_{i\cdot}\t.
\end{align*}
We will show that $(S^{(k)})\inv J_1$ converges to $I_r$ in probability and that the other terms tend to zero in probability.  
\\ \ \\ \noindent
\textbf{The term $J_1$:} Using the same expansion as in the proof for Corollary \ref{cor:mixture} we expand out $\hat U - U\mathcal{O}_*\t$ via
\begin{align*}
    J_1 :&= \frac{1}{n_k} \sum_{i=1}^{n_k} \bigg(\hat U  - U\mathcal{O}_*\t \bigg)_{i\cdot}\bigg(\hat U  - U\mathcal{O}_*\t \bigg)_{i\cdot}\t \\
    &= \frac{1}{n_k}\sum_{i=1}^{n_k} \bigg(  \mathcal{O}_*(\Lambda\inv V\t E_i) +\mathcal{O}_* R_i \bigg)\bigg(  \mathcal{O}_*(\Lambda\inv V\t E_i) +\mathcal{O}_* R_i \bigg)\t \\
    &=  \frac{1}{n_k} \sum_{i=1}^{n_k} \mathcal{O}_* (\Lambda\inv V\t E_i E_i\t V \Lambda\inv)  \mathcal{O}_*\t + \frac{1}{n_k} \sum_{i=1}^{n_k} \mathcal{O}_* R_i (E_i\t V \Lambda\inv)  \mathcal{O}_*\t  + \frac{1}{n_k} \sum_{i=1}^{n_k} \mathcal{O}_* (\Lambda\inv  V\t E_i )R_i  \mathcal{O}_*\t \\
    &\qquad + \frac{1}{n_k} \sum_{i=1}^{n_k} \mathcal{O}_* R_i R_i\t \mathcal{O}_*\t \\
    :&= J_{11} + J_{12} + J_{13},
\end{align*}
where
\begin{align*}
    J_{11} :&= \frac{1}{n_k} \sum_{i=1}^{n_k} \mathcal{O}_* (\Lambda\inv V\t E_i E_i\t V \Lambda\inv)  \mathcal{O}_*\t; \\
    J_{12} :&= \frac{1}{n_k} \sum_{i=1}^{n_k} \mathcal{O}_* R_i (E_i\t V \Lambda\inv)  \mathcal{O}_*\t  + \frac{1}{n_k} \sum_{i=1}^{n_k} \mathcal{O}_* (\Lambda\inv  V\t E_i )R_i  \mathcal{O}_*\t; \\
    J_{13} :&= \frac{1}{n_k} \sum_{i=1}^{n_k} \mathcal{O}_* R_i R_i\t \mathcal{O}_*\t.
\end{align*}
Since $E_i = (\Sigma^{(k)})^{1/2} Y_i$, the term $J_{11}$ satisfies
\begin{align*}
 \mathcal{O}_*  ( S^{(k)})\inv \mathcal{O}_*\t J_{11} = \mathcal{O}_*  ( S^{(k)})\inv \frac{1}{n_k} \sum_{i=1}^{n_k} \Lambda\inv V\t \Sigma_i^{1/2} Y_i Y_i\t \Sigma_i^{1/2} V \Lambda\inv \mathcal{O}_*\t.
\end{align*}
The random variable $ \Lambda\inv V\t \Sigma_i^{1/2} Y_i Y_i\t \Sigma_i^{1/2} V \Lambda\inv$ is an $r \times r$ matrix with expectation $S^{(k)}$, and hence $ \|\mathcal{O}_* (S^{(k)})\inv \mathcal{O}_*\t J_1 - I_r \| \to 0$ in probability by the strong law of large numbers, the rotation invariance of the spectral norm, and the fact that each $Y_i$ has independent subgaussian components. We now show the other terms all tend to zero in probability.

As for $J_{12}$, we analyze the term
\begin{align*}
     \mathcal{O}_*  ( S^{(k)})\inv \mathcal{O}_*\t \frac{1}{n_k} \sum_{i=1}^{n_k} \mathcal{O}_* R_i (E_i\t V \Lambda\inv)  \mathcal{O}_*\t &=  \mathcal{O}_*  ( S^{(k)})\inv\frac{1}{n_k} \sum_{i=1}^{n_k}  R_i (E_i\t V \Lambda\inv)  \mathcal{O}_*\t.
\end{align*}
The other term is similar.  By the rotational invariance of the spectral norm, we may ignore the orthogonal matrices henceforth.  By the residual bounds in Lemmas \ref{lem:residuals1}, \ref{lem:residuals2}, and \ref{lem:final_res}, we have that with probability at least $1 - (n\vee d)^{-3}$,
\begin{align*}
  \max_{i} \|(S^{(k)})^{-1/2}  R_{i}\| \lesssim  \frac{\log(n\vee d)}{\snr} +  \sqrt{\frac{\log(n\vee d)}{n}}, 
\end{align*}
where we let the implicit constants depend on $\kappa, \mu_0, \kappa_{\sigma},$ and $r$, since they are assumed bounded in $n$ and $d$.  Let this event be denoted $\mathcal{E}$.  Then
\begin{align*}
  \p\bigg( \|   ( S^{(k)})\inv\frac{1}{n_k} \sum_{i=1}^{n_k}  R_i (E_i\t V \Lambda\inv)  \| > t \bigg) &\leq \p\bigg( \|   ( S^{(k)})\inv\frac{1}{n_k} \sum_{i=1}^{n_k}  R_i (E_i\t V \Lambda\inv)   \| > t \cap\mathcal{E} \bigg) + (n\vee d)^{-3},
\end{align*}
so it suffices to analyze this term on the event $\mathcal{E}$.  Note that the vector $E_i\t V \Lambda\inv$ is an $r-$dimensional random variable with covariance matrix $S^{(k)}$.  Note that the condition number of $S^{(k)}$ is at most $\kappa^2 \kappa_{\sigma}^2$, and hence $(S^{(k)})^{1/2}$ has condition number at most $\kappa \kappa_{\sigma}$.  Therefore,
%Since the nonzero eigenvalues of the matrix $AB$ are the same as the nonzero eigenvalues of $BA$, we have that
\begin{align*}
    \|   ( S^{(k)})\inv\frac{1}{n_k} \sum_{i=1}^{n_k}  R_i (E_i\t V \Lambda\inv)  \| &\leq \kappa \kappa_{\sigma}  \|   ( S^{(k)})^{-1/2}\frac{1}{n_k} \sum_{i=1}^{n_k}  R_i (E_i\t V \Lambda\inv)   ( S^{(k)})^{-1/2} \| \\
    &\leq  r \kappa \kappa_{\sigma}  \|  \frac{1}{n_k} \sum_{i=1}^{n_k}  ( S^{(k)})^{-1/2}  R_i (E_i\t V \Lambda\inv)   ( S^{(k)})^{-1/2} \|_{\max}.
\end{align*}
Now consider the $j,l$ entry of the above matrix, which can be written as
\begin{align*}
  \frac{1}{n_k} \sum_{i=1}^{n_k}  \bigg( ( S^{(k)})^{-1/2} R_i\bigg)_j \bigg[(E_i\t V \Lambda\inv)   ( S^{(k)})^{-1/2} \bigg]_l.
\end{align*}
By (the restricted) Markov's inequality,
\begin{align*}
    \p\bigg( \bigg| \frac{1}{n_k} \sum_{i=1}^{n_k}  \bigg( ( S^{(k)})^{-1/2} R_i\bigg)_j \bigg[(E_i\t V \Lambda\inv)   &( S^{(k)})^{-1/2} \bigg]_l \bigg| > t \cap \mathcal{E} \bigg) \\
    &\leq \frac{1}{t} \E \mathbb{I}_{\mathcal{E}} \bigg| \frac{1}{n_k} \sum_{i=1}^{n_k}  \bigg( ( S^{(k)})^{-1/2} R_i\bigg)_j \bigg[(E_i\t V \Lambda\inv)   ( S^{(k)})^{-1/2} \bigg]_l \bigg| \\
    &\leq \frac{1}{t} \max_i   \E \mathbb{I}_{\mathcal{E}} \bigg|  \bigg( ( S^{(k)})^{-1/2} R_i\bigg)_j \bigg[(E_i\t V \Lambda\inv)   ( S^{(k)})^{-1/2} \bigg]_l \bigg| \\
    &\lesssim \frac{1}{t} \bigg( \frac{\log(n\vee d)}{\snr} +  \sqrt{\frac{\log(n\vee d)}{n}}\bigg)  \E  \bigg| \bigg[(E_i\t V \Lambda\inv)   ( S^{(k)})^{-1/2} \bigg]_l \bigg| \\
    &\lesssim\frac{1}{t} \bigg( \frac{\log(n\vee d)}{\snr} +  \sqrt{\frac{\log(n\vee d)}{n}}\bigg),
\end{align*}
where the final inequality is due to the fact that $E_i\t V \Lambda\inv (S^{(k)})^{-1/2}$ is an isotropic $r$-dimensional subgaussian random variable, and hence has moments are bounded by $O(1)$.  Therefore, we conclude that since the $j,l$ entry converges to zero in probability, since $r$ is fixed, we conclude that $(S^{(k)})\inv J_{12}$ converges to zero in probability.

%Moreover, the $l$'th entry of the random variable $E_i\t V \Lambda\inv (S^{(k)})^{-1/2}$ has $\psi_2$ norm bounded by $O(1)$ since it is an isotropic random variable.  Therefore, on the event $\mathcal{E}$, it is straightforward to see that
%\begin{align*}
%  \bigg\|   \bigg( ( S^{(k)})^{-1/2} R_i\bigg)_j \bigg[(E_i\t V \Lambda\inv)   ( S^{(k)})^{-1/2} \bigg]_l \bigg\|_{\psi_2} &\lesssim   \frac{\log(n\vee d)}{\snr} +  \sqrt{\frac{\log(n\vee d)}{n}} \\
%  :&= B.
%\end{align*}
%Therefore, by Hoeffding's inequality, we have that
%\begin{align*}
%    \p\bigg(  \frac{1}{n_k} \sum_{i=1}^{n_k}  \bigg( ( S^{(k)})^{-1/2} R_i\bigg)_j \bigg[(E_i\t V \Lambda\inv)   ( S^{(k)})^{-1/2} \bigg]_l > t \cap \mathcal{E} \bigg) &\leq \exp\bigg( -c \frac{ n^2 t^2}{B^2}\bigg).
%\end{align*}
%Therefore, taking $t = C B \frac{\log(n)}{n}$ shows that this converges to zero in probability (note that we could also have calculated the moment generating function directly, as we do in the proof of Lemma \ref{lem:2infty_general}).  Since $r$ is fixed, we see that $(S^{(k)})\inv J_2$ tends to zero in probability.
For $J_{13}$, after accounting for orthogonal matrices, we note that 
\begin{align*}
     \frac{1}{n_k} \bigg\| (S^{(k)})\inv \sum_{i=1}^{n_k} R_i R_i\t \bigg\| &\leq  \max_i \| (S^{(k)})\inv R_i R_i\t \|.
\end{align*}
The matrix $R_i R_i\t$ is rank one and $S^{(k)}$ is assumed to be positive definite by Assumption \ref{assumption:be_assumptions}.  Therefore this term satisfies
\begin{align*}
    \max_i \| (S^{(k)})\inv R_i R_i\t \| &= \max_i R_i\t (S^{(k)})\inv R_i\t \\
    &= \max_i \| (S^{(k)})^{-1/2} R_i \|^2.
\end{align*}
By the argument for the same term in the proof of Corollary \ref{cor:b-e}, we conclude that this term tends to zero in probability, where we have implicitly used the fact that the bounds in Lemmas \ref{lem:residuals1}, \ref{lem:residuals2} and \ref{lem:final_res} are uniform over $i$.  
\\ \ \\ \noindent
\textbf{The term $J_2$:} We note that $\bar U^{(k)}$ is the same for all $i$ and by Lemma 2.1 of \citet{lei_consistency_2015}, the term $U_{i\cdot}$ is the same for all $i$ belonging to community $k$.  Hence, again using the asymptotic expansion as in the proof of Corollary \ref{cor:b-e}, we have that
\begin{align*}
    \bigg( \bar U^{(k)}  - U^{(k)}\mathcal{O}_*\t \bigg)\bigg( \bar U^{(k)} - U^{(k)}\mathcal{O}_*\t \bigg)\t    &=  \bigg( \frac{1}{n_k} \sum_{i=1}^{n_k} (\hat U_i - U_i \mathcal{O}_*\t) \bigg) \bigg( \frac{1}{n_k} \sum_{i=1}^{n_k} (\hat U_i - U_i \mathcal{O}_*\t) \bigg)\t \\
    &=   \bigg( \frac{1}{n_k} \sum_{i=1}^{n_k} \mathcal{O}_*(\Lambda\inv V\t E_i) +\mathcal{O}_* R_i   \bigg) \bigg( \frac{1}{n_k} \sum_{i=1}^{n_k} \mathcal{O}_*(\Lambda\inv V\t E_i) +\mathcal{O}_* R_i  \bigg)\t \\
    &= \mathcal{O}_* \bigg( \frac{1}{n_k} \sum_{i=1}^{n_k} \Lambda\inv V\t E_i \bigg) \bigg( \frac{1}{n_k} \sum_{i=1}^{n_k} \Lambda\inv V\t E_i \bigg)\t \mathcal{O}_*\t \\&\qquad +  \mathcal{O}_* \bigg( \frac{1}{n_k} \sum_{i=1}^{n_k} R_i  \bigg) \bigg( \frac{1}{n_k} \sum_{i=1}^{n_k} \Lambda\inv V\t E_i   \bigg)\t \mathcal{O}_*\t
    \\&\qquad +  \mathcal{O}_* \bigg( \frac{1}{n_k} \sum_{i=1}^{n_k} \Lambda\inv V\t E_i   \bigg)\bigg( \frac{1}{n_k} \sum_{i=1}^{n_k} R_i  \bigg)\t \mathcal{O}_*\t\\
    &\qquad + \mathcal{O}_* \bigg( \frac{1}{n_k} \sum_{i=1}^{n_k} R_i  \bigg)\bigg( \frac{1}{n_k} \sum_{i=1}^{n_k} R_i  \bigg)\t\mathcal{O}_*\t \\
    :&= J_{21} + J_{22} + J_{23} + J_{24}.
\end{align*}
The term $J_{21}$ satisfies
\begin{align*}
    \| \mathcal{O}_* (S^{(k)})\inv \mathcal{O}_*\t J_{21} \| &\leq \kappa \kappa_{\sigma} \frac{1}{n_k^2} \| (S^{(k)})^{-1/2} \sum_{i=1}^{n_k} \Lambda\inv V\t E_i \|^2.
\end{align*}
Therefore, by Markov's inequality, since $\kappa$ and $\kappa_{\sigma}$ are assumed bounded, by including them in the implicit constants, we have that
\begin{align*}
      \p \bigg( \| \mathcal{O}_* (S^{(k)})\inv\mathcal{O}_*\t J_{21} \| > t \bigg) &\lesssim \frac{ \E\bigg(  \| (S^{(k)})^{-1/2} \sum_{i=1}^{n_k} \Lambda\inv V\t E_i \|^2 \bigg)}{n_k^2 t} \\
      &= \frac{ \E   \bigg(\sum_{i=1}^{n_k} \Lambda\inv V\t E_i \bigg)\t (S^{(k)})\inv   \bigg(\sum_{j=1}^{n_k} \Lambda\inv V\t E_i \bigg) \bigg)}{n_k^2 t} \\
      &=\sum_{i=1}^{n_k} \frac{ \E \bigg[  E_i\t V \Lambda\inv (S^{(k)})\inv   \bigg(\sum_{j=1}^{n_k} \Lambda\inv V\t E_j \bigg) \bigg]}{n_k^2 t} \\
      &=\sum_{i=1}^{n_k} \frac{ \E \tr \bigg(\sum_{j=1}^{n_k} \Lambda\inv V\t E_j \bigg) \bigg[  E_i\t V \Lambda\inv (S^{(k)})\inv    \bigg]}{n_k^2 t} \\
      &= \sum_{i=1}^{n_k} \frac{  \tr \Lambda\inv V\t \Sigma^{(k)} V \Lambda\inv (S^{(k)})\inv    }{n_k^2 t} \\
      &= \sum_{i=1}^{n_k} \frac{r}{n_k^2 t} \\
      &= \frac{r}{n_k t},
\end{align*}
which implies that $\| (S^{(k)})\inv J_{21} \|$ converges to zero in probability.  

Note that $J_{22}$ is a rank one matrix.  Therefore
\begin{align}
      \|\mathcal{O}_*  (S^{(k)})\inv \mathcal{O}_*\t J_{22} \| &\leq \frac{\kappa\kappa_\sigma}{n_k^2}\bigg| \bigg\langle  (S^{(k)})^{-1/2} \sum_{i=1}^{n_k} R_i, (S^{(k)})^{-1/2} \sum_{i=1}^{n_k} E_i\t V \Lambda\inv \bigg\rangle \bigg| \nonumber\\
      &\lesssim \frac{1}{n_k^2} \bigg\| (S^{(k)})^{-1/2} \sum_{i=1}^{n_k} R_i \bigg\| \bigg\|(S^{(k)})^{-1/2} \sum_{i=1}^{n_k} E_i\t V\Lambda\inv \bigg\| \nonumber \\
      &\lesssim   \bigg\|(S^{(k)})^{-1/2} \sum_{i=1}^{n_k} E_i\t V\Lambda\inv \bigg\| \frac{\max_i \| (S^{(k)})^{-1/2} R_i\|}{n_k}. \label{init1}
\end{align}
By the same calculation as in $J_{21}$, for any $t > 0$,
\begin{align*}
    \p\bigg(  \bigg\|(S^{(k)})^{-1/2} \sum_{i=1}^{n_k} E_i\t V\Lambda\inv \bigg\| > t \bigg) &\leq \frac{r n_k }{ t^2 }.
\end{align*}
Setting $t = n^{1/2} \sqrt{\log(n)}$ shows that
\begin{align}
    \bigg\|(S^{(k)})^{-1/2} \sum_{i=1}^{n_k} E_i\t V\Lambda\inv \bigg\| &\leq \sqrt{n\log(n)}\label{init2}
\end{align}
with probability at least $1 - o(1)$.  Furthermore, by the same argument as in the proof of Corollary \ref{cor:b-e}, (i.e. the residual bounds on $R_i$ (as in Lemmas \ref{lem:residuals1}, \ref{lem:residuals2}, and \ref{lem:final_res}), we have that with probability at least $1 - (n\vee d)^{-3}$
\begin{align}
    \max_{i} \|(S^{(k)})^{-1/2}  R_{i}\| \lesssim  \frac{\log(n\vee d)}{\snr} +  \sqrt{\frac{\log(n\vee d)}{n}}, \label{init3}
\end{align}
where we again let the implicit constants depend on $\kappa, \mu_0, \kappa_{\sigma},$ and $r$, since they are assumed bounded in $n$ and $d$.  Hence, combining \eqref{init1}, \eqref{init2}, and \eqref{init3}, we have that with probability at least $1 - o(1)$, 
\begin{align*}
     \| \mathcal{O}_* (S^{(k)})\inv  \mathcal{O}_*\t J_{22} \| &\leq \frac{\sqrt{n\log(n)}}{n_k} \bigg(  \frac{\log(n\vee d)}{\snr} +  \sqrt{\frac{\log(n\vee d)}{n}} \bigg)
\end{align*}
which converges to zero.  The same exact proof works for $J_{23}$.  For $J_{24}$, we note that by Cauchy-Schwarz, the term tends to zero by a similar method as in the bound for $J_1$.  
\\ \ \\ \noindent
\textbf{The term $J_3$:} Recall that $J_3$ consists of two terms, one of which is the transpose of the other.  We analyze only the first as the other is similar, again using the same expansion as in $J_1$ and $J_2$.  We have that  
\begin{align*}
  \bigg\|   \mathcal{O}_* [S^{(k)}]\inv \mathcal{O}_*\t \bigg[ &\frac{1}{n_k} \sum_{i=1}^{n_k} \bigg( \hat U - U \mathcal{O}_*\t\bigg)_{i\cdot}\bigg( U\mathcal{O}_*\t - \bar U^{(k)} \bigg)_{i\cdot}\t \bigg]\bigg\| \\
  &= \bigg\| \frac{1}{n_k} \sum_{i=1}^{n_k} [S^{(k)}]\inv \mathcal{O}_*\t \bigg( \mathcal{O}_* \Lambda\inv V\t E_i + \mathcal{O}_* R_i \bigg)\bigg( U\mathcal{O}_*\t - \bar U^{(k)} \bigg)_{i\cdot}\t \bigg] \mathcal{O}_* \bigg\| \\
  &\leq \kappa \kappa_{\sigma}\bigg\| \frac{1}{n_k} \sum_{i=1}^{n_k} [S^{(k)}]^{-1/2} \bigg(  \Lambda\inv V\t E_i +  R_i \bigg)\bigg( U\mathcal{O}_*\t - \bar U^{(k)} \bigg)_{i\cdot}\t \mathcal{O}_*  [S^{(k)}]^{-1/2} \bigg]\bigg\| \\
  &\leq \kappa \kappa_{\sigma}\bigg\| \frac{1}{n_k} \sum_{i=1}^{n_k} [S^{(k)}]^{-1/2} \bigg(  \Lambda\inv V\t E_i +  R_i \bigg)\bigg( \frac{1}{n_k}\sum_{m=1}^{n_k} U\mathcal{O}_*\t -  \hat U \bigg)_{m\cdot}\t \mathcal{O}_*  [S^{(k)}]^{-1/2} \bigg]\bigg\| \\
  &\leq \kappa \kappa_{\sigma}\bigg\| \frac{1}{n_k} \sum_{i=1}^{n_k} [S^{(k)}]^{-1/2} \bigg(  \Lambda\inv V\t E_i +  R_i \bigg)\bigg( \frac{1}{n_k}\sum_{m=1}^{n_k} \mathcal{O}_* \Lambda\inv V\t E_m + \mathcal{O}_* R_m \bigg)\t \mathcal{O}_*  [S^{(k)}]^{-1/2} \bigg]\bigg\| \\
   &\leq \kappa \kappa_{\sigma}\bigg\| \frac{1}{n_k} \sum_{i=1}^{n_k} [S^{(k)}]^{-1/2} \bigg(  \Lambda\inv V\t E_i +  R_i \bigg)\bigg( \frac{1}{n_k}\sum_{m=1}^{n_k} E_m\t V \Lambda\inv +  R_m\t \bigg)   [S^{(k)}]^{-1/2} \bigg]\bigg\| \\
   &\lesssim \bigg\| \frac{1}{n_k} \sum_{i=1}^{n_k} [S^{(k)}]^{-1/2} \bigg(  \Lambda\inv V\t E_i  \bigg)\bigg( \frac{1}{n_k}\sum_{m=1}^{n_k} E_m\t V \Lambda\inv  \bigg)   [S^{(k)}]^{-1/2} \bigg]\bigg\| \\
   &\qquad + \bigg\| \frac{1}{n_k} \sum_{i=1}^{n_k} [S^{(k)}]^{-1/2} \bigg(  \Lambda\inv V\t E_i  \bigg)\bigg( \frac{1}{n_k}\sum_{m=1}^{n_k}   R_m\t \bigg)   [S^{(k)}]^{-1/2} \bigg]\bigg\| \\
   &\qquad + \bigg\| \frac{1}{n_k} \sum_{i=1}^{n_k} [S^{(k)}]^{-1/2} R_i \bigg( \frac{1}{n_k}\sum_{m=1}^{n_k} E_m\t V \Lambda\inv  \bigg)   [S^{(k)}]^{-1/2} \bigg]\bigg\| \\
   &\qquad + \bigg\| \frac{1}{n_k} \sum_{i=1}^{n_k} [S^{(k)}]^{-1/2} R_i \bigg( \frac{1}{n_k}\sum_{m=1}^{n_k}  R_m\t \bigg)   [S^{(k)}]^{-1/2} \bigg]\bigg\| \\
   :&= J_{31} + J_{32} + J_{33} + J_{34}.
\end{align*}
Since the term inside of $J_{31}$ is a product of two rank-one matrices, its spectral norm is equal to 
\begin{align*}
    \bigg\| \frac{1}{n_k} \sum_{i=1}^{n_k}  [S^{(k)}]^{-1/2} \Lambda\inv V\t E_i \bigg\|^2 &= \frac{1}{n_k^2}  \bigg\| \sum_{i=1}^{n_k}  [S^{(k)}]^{-1/2} \Lambda\inv V\t E_i \bigg\|^2.
\end{align*}
Note that
\begin{align*}
\E \bigg(  \sum_{i=1}^{n_k}  [S^{(k)}]^{-1/2} \Lambda\inv V\t E_i \bigg) \bigg( \ \sum_{m=1}^{n_k} E_m\t V \Lambda\inv [S^{(k)}]^{-1/2} \bigg) &= \sum_{i=1}^{n_k}  I_r \\
&= n_k I_r.
\end{align*}
Therefore, by Markov's inequality,
\begin{align*}
    \p\bigg( \frac{1}{n_k^2}  \bigg\| \sum_{i=1}^{n_k}  [S^{(k)}]^{-1/2} \Lambda\inv V\t E_i \bigg\|^2 > t \bigg) &\leq \frac{\E  \bigg\| \sum_{i=1}^{n_k}  [S^{(k)}]^{-1/2} \Lambda\inv V\t E_i \bigg\|^2}{n_k^2 t} \\
    &=\frac{n_k \tr  I_r}{n_k^2 t} \\
    &\leq \frac{r}{n_k t},
\end{align*}
so that $J_{31}$ tends to zero in probability.  
%The term $J_{31}$ can be analyzed by noting that
%\begin{align*}
%\E \bigg( \frac{1}{n_k} \sum_{i=1}^{n_k}  [S^{(k)}]^{-1/2} \Lambda\inv V\t E_i \bigg) \bigg( \frac{1}{n_k} \sum_{m=1}^{n_k} E_m\t V \Lambda\inv [S^{(k)}]^{-1/2} \bigg) %\\ 
%&= \frac{1}{n_k^2 }\sum_{i=1}^{n_k}  I_r \\
%&= \frac{1}{n_k} I_r.
%\end{align*}
%Consequently, this tends to zero as $n$ and $d$ tend to infinity by the law of large numbers and the fact that each vector $(S^{(k)})^{-1/2} \Lambda\inv V\t E_i$ is an $r$-dimensional Subgaussian isotropic random vector, so hence has uniformly bounded second moment in $n$ and $d$.  

For $J_{32}$, using the inequality $|ab| \leq |a| |b|$, we have that on the event $\mathcal{E}$,
\begin{align*}
    \bigg|\bigg(\frac{1}{n_k} \sum_{i=1}^{n_k} &[S^{(k)}]^{-1/2} \big(  \Lambda\inv V\t E_i  \big) \bigg)\bigg( \frac{1}{n_k}\sum_{i=1}^{n_k}   R_i\t \bigg)   [S^{(k)}]^{-1/2} \bigg)_{jl} \bigg|\\
    &\leq \max_{m} \bigg| \bigg ([S^{(k)}]^{-1/2} R_m \bigg)_l \bigg| \bigg| \bigg(\frac{1}{n_k} \sum_{i=1}^{n_k} [S^{(k)}]^{-1/2} \big(  \Lambda\inv V\t E_i  \big) \bigg)_j \bigg| \\
    &\lesssim \bigg( \frac{\log(n\vee d)}{\snr} + \sqrt{\frac{\log(n\vee d)}{n}} \bigg) \bigg| \bigg(\frac{1}{n_k} \sum_{i=1}^{n_k} [S^{(k)}]^{-1/2} \big(  \Lambda\inv V\t E_i  \big) \bigg)_j \bigg|.
\end{align*}
We have already shown that the outermost term tends to zero in probability, implying that $J_{32}$ tends to zero in probability since $r$ is fixed.  The same argument also works for $J_{33}$.  For $J_{34}$, the same argument as $J_{13}$ suffices, and hence $J_{3}$ tends to zero in probability, which completes the proof.
\end{proof}

\section{Proofs of Lemmas in Section \ref{sec:randompart}} \label{sec:lemmas_random}

In this section we prove the additional Lemmas required for the proof of Theorem \ref{thm:2_infty_random}. 

\subsection{Proof of Lemmas \ref{lem:spectral_norm_concentration} and \ref{lem:diag}}

We first prove the spectral norm concentration bound in Lemma \ref{lem:spectral_norm_concentration}.  We restate it here.

\spectralnormconcentration*

\begin{proof}[Proof of Lemma \ref{lem:spectral_norm_concentration}]
We will follow the proof in Theorem 2 and Lemma 3 in \citet{amini_concentration_2021}. More specifically, defining $\nu := d + \|M\|^2/\sigma^2$, we will show that there exists a universal constant $c$ such that for any $u \geq 0$, with probability at least $1 - 4(n\vee d)^{-6} \exp(-u^2)$,
\begin{align}
    \| \Gamma(EM\t + ME\t +EE\t) \| \leq 2\sigma^2 \nu \max(\delta^2,\delta), \label{claim1}
\end{align}
for $\delta = \sqrt{\frac{6\log(n\vee d)/c + n\log(9)/c + u^2/c}{\nu}}$.  To obtain the final result, note that by choosing $C$ sufficiently large, we have that when $u = 0$,
\begin{align*}
    \delta \leq C \sqrt{n/\nu}. 
\end{align*}
Furthermore, $\max(\delta^2, \delta) \leq \delta^2 + \delta$ for all $\delta$.  Then the result reads that with probability at least $1 - 4(n\vee d)^{-6}$,
\begin{align*}
    \| \Gamma(EM\t + ME\t +EE\t) \| &\leq 2 \sigma^2 \nu( \delta^2 + \delta ) \\
    &\leq 2 \sigma^2 \nu ( C^2 n/\nu + C\sqrt{n/\nu} ) \\
    &\leq 2 \sigma^2 C ( Cn +  \sqrt{n\nu} ) \\
    &\leq 2 \sigma^2 C( Cn + \sqrt{n (d + \| M\|^2 /\sigma^2 )} ) \\
    &\leq 2 \sigma^2 C( Cn + \sqrt{nd}) + 2C \sqrt{n} \sigma \lambda_1  \\
    &\leq C_{\text{spectral}} \bigg( \sigma^2 (n + \sqrt{nd}) + \sqrt{n} \sigma \kappa \lambda_r \bigg)
\end{align*}
by taking $C_{\text{spectral}}$ sufficiently large and recalling that $\kappa = \lambda_1/\lambda_r$.

We now prove the claim \eqref{claim1}.   Let $z \in \R^n$ be a unit vector, and define $Y_z := z\t \Gamma(Z) z$.  Recall $Z = EM\t + ME\t + EE\t$, where the rows of $E$ are of the form $\Sigma_{i}^{1/2} W_i$ for vectors $W_i$ with independent mean-zero coordinates with unit $\psi_2$ norm.  Define $\Vec{X} \in \R^{nd}$ as the vector obtained by stacking the vectors $X_i$.  Then 
\begin{align*}
    z\t (A + \Gamma( Z)) z &= \| Xz \|^2 - \sum_{i}z_i^2 \| X_i \|^2 + z\t G(A) z.
\end{align*}
Define the matrix $\Xi_z := z\t \otimes I_d \in \R^{d \times nd}$.  Then $\Xi_z \Vec{X} = X z$, where $X$ is the matrix whose columns are $X_i$.  Let $\Sigma$ be the block-diagonal matrix whose $i$'th block is $\Sigma_i^{1/2}$.  Then
\begin{align*}
    Xz = \Xi_z \Vec{\mu} + \Xi_z \Sigma^{1/2} \Vec{W} = \Xi_z \Sigma^{1/2} \Vec{\xi},
\end{align*}
for $\xi = \Sigma^{-1/2} \Vec{\mu} + \Vec{W}$.  Similarly, note that
\begin{align*}
    \sum_{i=1}^{n} z_i^2  \| X_i\|^2 &= \| \diag( z_i I_d )  \Vec{X} \|^2 = \| \diag(z_i I_d ) \Sigma^{1/2} \Vec{\xi} \|^2.
\end{align*}
Therefore, we have that $z\t Az-z\t G(A)z=z\t\Gamma(A)z$ and
\begin{align*}
    z\t ( \Gamma(Z)) z &= z\t ( A + \Gamma(Z)) z - z\t A z \\
    &= \| \Vec{X}z \|^2 - \sum_{i=1}^n z_i^2 \| X_i \|^2 +z\t G(A)z - z\t A z \\
    &= \| \Xi_z \Sigma^{1/2} \Vec{\xi} \|^2 - \| \diag(z) \Sigma^{1/2} \Vec{\xi} \|^2 - z\t \Gamma(A)z \\
    &= \Vec{\xi}\t B_z \Vec{\xi} - \Vec{\xi}\t C_z \Vec{\xi} - z\t \Gamma(A) z \\
    &= \Vec{\xi}\t (B_z - C_z) \Vec{\xi} - \E\bigg( \Vec{\xi}\t (B_z - C_z) \Vec{\xi} \bigg),
\end{align*}
where $B_z := \Sigma^{1/2} \Xi_z\t \Xi_z \Sigma^{1/2}$; $C_z := \Sigma^{1/2} \diag(z_i I_d) \diag(z_i I_d) \Sigma^{1/2}$.  We now apply the extension of the Hanson-Wright inequality (Theorem 6 in \citet{amini_concentration_2021}) to this quadratic form to determine that \begin{align}
    \p\bigg( |Y_z | \geq t \bigg) \leq 4 \exp\bigg( -c \min\bigg( \frac{t^2}{\| B_z - C_z \|_F^2 + \| \tilde M (B_z - C_z)\|_F^2}, \frac{t}{\| B_z - C_z \|} \bigg) \bigg), \label{hansonwright}
\end{align}
where $\tilde M := (\Sigma^{-1/2} \Vec{\mu})\t$. We now bound the denominators.  

We have that $B_z - C_z = \Sigma^{1/2}(\Xi_z\t \Xi_z - \diag(z_i^2 I_d)) \Sigma^{1/2}$.  Hence $\| B_z - C_z \|_F = \| \Sigma^{1/2}(\Xi_z\t \Xi_z - \diag(z_i^2 I_d)) \Sigma^{1/2}\|_F \leq \| \Sigma^{1/2} \|^2 \| (\Xi_z\t \Xi_z - \diag(z_i^2 I_d)) \|_F$.  Furthermore, by the parallelogram law and the fact that $z$ is unit norm,
\begin{align*}
    \| \Xi_z\t \Xi_z - \diag(z_i^2 I_d) \|_F^2 &= 2\| \Xi_z\t \Xi_z \|_F^2 + 2 \| \diag(z_i^2 I_d) \|_F^2 - \| \Xi_z\t \Xi_z + \diag(z_i^2 I_d) \|_F^2 \\
    &\leq 2d + 2d \\
    &\leq 4d,
\end{align*}
so that $\|B_z - C_z\|_F^2 \leq 4d \sigma^4,$ where $\sigma^2 = \| \Sigma \| = \max_i \|\Sigma_i\|$.  Additionally,
\begin{align*}
    \| \tilde M(B_z - C_z) \|_F^2 &= \| (\Sigma^{-1/2} \Vec{\mu} )\t \Sigma^{1/2}( \Xi_z\t \Xi_z - \diag(z_i^2 I_d) ) \Sigma^{1/2} \|_F^2 \\
    &= \| \Vec{\mu} ( \Xi_z\t \Xi_z - \diag(z_i^2 I_d) ) \Sigma^{1/2} \|_F^2 \\
    &= \| \Sigma^{1/2}( \Xi_z\t \Xi_z - \diag(z_i^2 I_d) ) \Vec{\mu} \|_2^2 \\
    &\leq \sigma^2 \|( \Xi_z\t \Xi_z - \diag(z_i^2 I_d) ) \Vec{\mu} \|_2^2.
\end{align*}
Note that 
\begin{align*}
    \|( \Xi_z\t \Xi_z - \diag(z_i^2 I_d) ) \Vec{\mu} \|_2 &\leq \| \Xi_z\t \Xi_z \Vec{\mu} \| + \|\diag(z_i^2 I_d)\Vec{\mu} \| \\
    &\leq 2 \| M \|
\end{align*}
using the definition of $\Vec{\mu}$.  Finally, for the operator norm, we have that
\begin{align*}
    \| B_z - C_z \|&\leq \| B_z \| + \|C_z \| \\
    &\leq \sigma^2 \| \Xi_z \|^2 + \sigma^2 \| \diag(z_i^2 I_d)\|^2 \\
    &\leq 2 \sigma^2.
\end{align*}
In summary,
\begin{align}
    \| B_z - C_z \|_F^2 &\leq 4\sigma^4 d; \label{b1}\\
    \| \tilde M(B_z - C_Z)\|_F^2 &\leq 4 \sigma^2 \| M \|^2; \label{b2}\\
    \| B_z - C_z \| &\leq 4 \sigma^2. \label{b3}
\end{align}
Plugging \eqref{b1}, \eqref{b2}, and \eqref{b3} into Equation \ref{hansonwright} and absorbing $1/4$ into the constant $c$ yields
\begin{align*}
     \p\bigg( |Y_z | \geq t \bigg) \leq 4 \exp\bigg( -c \min\bigg( \frac{t^2}{\sigma^4d  + \sigma^2 \|M\|^2}, \frac{t}{\sigma^2} \bigg) \bigg).
\end{align*}
Define $\tilde t := \sigma^2 t$.  Then the inequality above becomes 
\begin{align*}
     \p\left( |Y_z | \geq t \sigma^2 \right) \leq 4 \exp\bigg( -c \min\bigg( \frac{t^2}{d  + \sigma^{-2}\|M\|^2},t \bigg)\bigg).
\end{align*}
 By changing $t$ to $\nu t$, the above concentration can be written via
\begin{align*}
    \p\bigg( |Y_z| \geq t \sigma^2 \nu \bigg) \leq 4 \exp(-c\nu \min(t^2,t)).
\end{align*}
Define $\delta := \sqrt{ \frac{6\log(n\vee d) +n \log(9) + u^2}{\nu c}}$.  Note that regardless of the value of $\delta$,
\begin{align*}
    \min( \max(\delta^2, \delta)^2 , \max(\delta^2, \delta)) \leq \delta^2.
\end{align*}
Taking $t = \max(\delta^2,\delta)$, we arrive at
\begin{align*}
    \p( |Y_z| \geq \sigma^2 \nu \max(\delta^2,\delta)) &\leq 4 \exp(-c \nu\min( \max(\delta^2, \delta)^2 , \max(\delta^2, \delta)) ) \\
    &\leq 4 \exp( -c \nu\delta^2) \\
    &\leq 4 \exp( -c \frac{6\log(n\vee d) +n \log(9) + u^2}{c} ) \\
    &\leq 4 \exp( -6\log(n\vee d) - n\log(9) - u^2).  
\end{align*}
Now we follow the proof of Theorem 2 in \citet{amini_concentration_2021} via an $\eps$-net. Let $\mathcal{N}$ be a $1/4$ net of the $n$-sphere, and hence that $| \mathcal{N}| \leq 9^n$.  Then $\| \Gamma(Z) \| \leq 2 \max_{z\in \mathcal{N}} |Y_z|$, so that by a union bound,
\begin{align*}
    \p\bigg( \| \Gamma(Z) \| \geq 2 \sigma^2 \nu \max(\delta^2,\delta) \bigg) &\leq 4 \cdot 9^n \exp(-6\log(n\vee d) - n\log(9) - u^2) \\
    &\leq 4 \exp( n\log(9) - 6\log(n\vee d) - n\log(9) - u^2) \\
    &\leq 4 (n\vee d)^{-6} \exp(-u^2).
\end{align*}
This proves the result. 

\end{proof}

\diagonals*

\begin{proof}[Proof of Lemma \ref{lem:diag}]
This follows because \begin{align*}
\| \diag (EM\t + ME\t) \|_2 &=   \max_{i} |(EM\t+ ME\t)_{ii}| \\
&= 2\max_{i} |\langle E_{i}. M_{i} \rangle| \\
&= 2\max_{i} | \sum_{\alpha} Y_{i\alpha} (\Sigma_{i}^{1/2} M_{i})_\alpha |.
\end{align*}
This is a sum of $d$ independent random variables with $\psi_2$ norm bounded by $2 \max_{i} \| \Sigma_{i}^{1/2} M_{i} \|  \leq 2 \sigma \max_{i} \|M_{i}\| \leq 2\sigma \lambda_1 \|U\|_{2,\infty}$.  Consequently, Hoeffding's inequality and a union bound that $\|\diag(EM\t + ME\t)\| \leq 2 \lambda_1 \sigma \|U\|_{2,\infty}\sqrt{6 \log(n \vee d)}$ with probability at least $1- 2(n\vee d)^{-5}$. By the Davis-Kahan theorem,
\begin{align*}
  \|\tilde U_D \tilde U_D\t - \tilde U \tilde U\t\| &\leq  C \frac{2 \lambda_1}{\lambda_r^2} \sigma \|U\|_{2,\infty}\sqrt{ \log(n\vee d)} \\
  &\leq \frac{C\delta_1}{\lambda_r^2} \|U\|_{2,\infty}
\end{align*}
where we have used Weyl's inequality and Assumption \ref{assumption:eigengap}. 
\end{proof}

\subsection{Proof of Lemma \ref{lem:2infty_general}}
In order to prove Lemma \ref{lem:2infty_general}, we will also require the following additional lemmas.

\begin{restatable}{lemma}{leminftyzero}
\label{lem:2infty_0}
There exists an absolute constant $C_0$ such that with probability at least $1 - (n\vee d)^{-5}$ it holds that 
\begin{align*}
   \| \per W U\|_{2,\infty} &\leq C_{0}\bigg(\sqrt{rnd}\log( n\vee d) \sigma^2 \|U\|_{2,\infty} + \sqrt{ rn\log( n\vee d)} \lambda_1 \sigma \bigg)\|U\|_{2,\infty},
\end{align*}
\end{restatable}

\begin{proof}[Proof of Lemma \ref{lem:2infty_0}]
Note that
\begin{align}
    \|\per W U\|_{2,\infty} &\leq \| U U\t (ME\t + EM\t + \Gamma(EE\t)) U \|_{2,\infty} + \| (EM\t + ME\t + \Gamma(EE\t) )U\|_{2,\infty}  \nonumber \\
    &\leq \|U\|_{2,\infty} \| W \| + \| (EM\t + ME\t + \Gamma(EE\t))U\|_{2,\infty} \label{lem:2infty_0_1}
\end{align}
We will analyze the second term.  
Note that the $i,j$ entry of $W$ can be written as
\begin{align*}
     E_i\t M_j + M_i\t E_j +  E_i \t E_j(1 - \delta_{ij}).
\end{align*}
Furthermore, 
note that
\begin{align*}
\| W U \|_{2,\infty} &\leq \sqrt{r} \| WU \|_{\max}.
\end{align*}
Fix an index $i,j$.  then this shows that
\begin{align*}
(WU)_{ij} &= \sum_{k=1}^n \bigg( \langle Y_i, \Sigma_i^{1/2} M_k \rangle + \langle Y_k, \Sigma_k^{1/2} M_i \rangle +(1 - \delta_{ik}) \langle Y_i, \Sigma_i^{1/2} \Sigma_k^{1/2} Y_k \rangle \bigg) U_{kj}.
\end{align*}
Define
\begin{align*}
\xi_{ik} := \begin{cases}  \langle Y_i, \Sigma_i^{1/2} M_k \rangle + \langle Y_k, \Sigma_k^{1/2} M_i \rangle + \langle Y_i, \Sigma_i^{1/2} \Sigma_k^{1/2} Y_k \rangle & i\neq k \\
2 \langle Y_i, \Sigma_i^{1/2} M_i \rangle  & i = k. \end{cases}
\end{align*}
Then 
\begin{align*}
(WU)_{ij} = \sum_{k=1}^n  \xi_{ik} U_{kj}
 \end{align*}

Expanding out $\xi_{ik}$, we have that 
\begin{align*}
\xi_{ik} &= \begin{cases} \sum_{\alpha =1}^{d} (Y_{i\alpha} [ \Sigma_i^{1/2} M_k]_{\alpha} + Y_{k\alpha} (\Sigma_{k}^{1/2} M_i)_{\alpha} ) + \sum_{\alpha=1}^d \sum_{\beta=1}^d (\Sigma_k^{1/2} \Sigma_{i}^{1/2})_{\alpha \beta} Y_{i\alpha} Y_{k\beta} & i \neq k \\
2 \sum_{\alpha =1}^{d} (Y_{i\alpha} [ \Sigma_i^{1/2} M_i]_{\alpha}  & i = k, \end{cases}
\end{align*}
% Then this sum is of the form
%\begin{align*}
%\sum_{\alpha =1}^{d} (Y_{i\alpha} [ \Sigma_i^{1/2} M_k]_{\alpha}] + \sum_{\alpha=1}^d Y_{k\alpha} (\Sigma_{k}^{1/2} M_i)_{\alpha} ) + \sum_{\alpha=1}^d \sum_{\beta=1}^d (\Sigma_k^{1/2} \Sigma_{i}^{1/2})_{\alpha \beta} Y_{i\alpha} Y_{k\beta}
%\end{align*}
where $Y_{i\alpha}$ denotes the $\alpha$ coordinate of the $i$'th random vector $Y_i$. We want to write this in terms of the independent collection of random variables $Y$. 
We have that
%The above sum can be written via
%\begin{align*}
%\sum_{k=1}^n  \xi_{ik} U_{kj} &= \sum_{k=1}^n  U_{kj} \bigg( \sum_{\alpha =1}^{d} (Y_{i\alpha} [ \Sigma_i^{1/2} M_k]_{\alpha}] + \sum_{\alpha=1}^d Y_{k\alpha} (\Sigma_{k}^{1/2} M_i)_{\alpha} ) + \sum_{\alpha=1}^d \sum_{\beta=1}^d (\Sigma_k^{1/2} \Sigma_{i}^{1/2})_{\alpha \beta} Y_{i\alpha} Y_{k\beta} \bigg).
%\end{align*}
%By the triangle inequality,
\begin{align*}
\bigg| \sum_{k=1}^n  \xi_{ik} U_{kj}\bigg| &\leq 2\bigg| \sum_{k=1}^n  U_{kj}\sum_{\alpha =1}^{d} (Y_{i\alpha} [ \Sigma_i^{1/2} M_k]_{\alpha}]  \bigg| + \bigg| \sum_{k=1,k\neq i}^n  U_{kj} \sum_{\alpha=1}^d Y_{k\alpha} (\Sigma_{k}^{1/2} M_i)_{\alpha} ) \bigg| \\
&\qquad + \bigg| \sum_{k=1,k\neq i}^n  U_{kj}\sum_{\alpha=1}^d \sum_{\beta=1}^d (\Sigma_k^{1/2} \Sigma_{i}^{1/2})_{\alpha \beta} Y_{i\alpha} Y_{k\beta} \bigg| \\
&= (I) + (II) + (III).
\end{align*}
In what follows, each term is bounded separately.

\ \\ \noindent \textbf{The term (I):}
The first term can be written via
\begin{align*}
\sum_{\alpha=1}^{d} Y_{i\alpha} \bigg[ \sum_{k=1}^{n} U_{kj}[\Sigma_i^{1/2} M_k]_{\alpha} \bigg],
\end{align*}
 which is a sum of $d$ mean-zero random variables.  So it suffices to bound the coefficients in order to apply the Hoeffding concentration inequality for subgaussians.  The coefficients can be found via
 \begin{align*}
 a_{(I)}(\alpha) := \sum_{k=1}^{n} U_{kj}[\Sigma_i^{1/2} M_k]_{\alpha} 
 \end{align*}
 for $\alpha$ ranging from $1$ to $d$.  %To bound it, it suffices to bound the squared euclidean norm and the maximum.  %For convenience, define $\tilde M_{ik\alpha} := [\Sigma_i^{1/2} M_k]_{\alpha}$.  Then by Holder,
 %\begin{align*}
 %\max_{\alpha}  | \sum_{k=1,k\neq i}^{n} U_{kj}\tilde M_{ik\alpha} | &\leq \max_{k,j,\alpha} |U_{kj}| \sum_{k=1,k\neq i}^{n} | \tilde M_{ik\alpha} | \\
% &\leq \| U \|_{2,\infty} \lambda_i^{1/2} \sum_{k\neq i,k=1}^n \| M_k \|\\
% &\leq \| U \|_{2,\infty} \sigma n \| M \|_{2,\infty}
% \end{align*}
% where the final inequality is by Cauchy-Schwarz.  Alternatively,
%Then
% \begin{align*}
% \max_{\alpha} \bigg| \sum_{k=1,k\neq i}^{n} U_{kj} (\Sigma_{i}^{1/2} M_k)_{\alpha} | &\leq \max_{\alpha} \bigg(\sum_{k=1,k\neq i}^{n}|U_{kj}|^2 \bigg)\bigg( \sum_{k=1,k\neq i}^{n} (\Sigma_{i}^{1/2} M_k)_{\alpha}^2\bigg) \\
% &\leq  \| (\Sigma_i^{1/2} M\t) \|_F.
% \end{align*}
% A tighter bound is likely possible.  
   Furthermore, 
 \begin{align*}
 \sum_{\alpha=1}^d \bigg( \sum_{k=1}^{n} U_{kj}\tilde M_{ik\alpha} \bigg)^2 &\leq \sum_{\alpha=1}^{d} \bigg( \sum_{k=1}^{n} U_{kj}^2 \bigg)  \bigg( \sum_{k=1,k\neq i}^{n} \tilde M_{ik\alpha}^2 \bigg) \\
 &\leq \sum_{\alpha=1}^d  \bigg( \sum_{k=1}^{n}  (\Sigma_i^{1/2} M_k)_{\alpha}^2 \bigg) \\
 &\leq \| \Sigma_{i}^{1/2} M\t \|^2_F \\
 &=  \| M \Sigma_{i}^{1/2} \|^2_F \\
 &\leq n \| M \Sigma_i^{1/2} \|_{2,\infty}^2 \\
 &\leq n \|U\|_{2,\infty}^2 \lambda_1^2 \sigma^2.
 \end{align*}
 By the generalized Hoeffding inequality (Theorem 2.6.3 in \citet{vershynin_high-dimensional_2018}) we obtain
 \begin{align*}
 \p\bigg( | (I) | > t \bigg) \leq 2\exp\bigg[ -c \bigg( \frac{t^2}{ n \|U\|_{2,\infty}^2 \lambda_1^2 \sigma^2} \bigg) \bigg]
 \end{align*}
 Taking $t =  C\|U\|_{2,\infty} \lambda_1 \sigma \sqrt{ n (2\gamma \log(n\vee d) + 2\log(r))}$ shows that this holds with probability at least $1 - 2r\inv n^{-\gamma}$.  Therefore, we derive the first bound,
 \begin{align*}
     (I) &\leq  C \|U\|_{2,\infty} \lambda_1 \sigma \sqrt{ n (2\gamma \log(n\vee d) + 2\log(r))} \\
     &\leq C \|U\|_{2,\infty} \lambda_1 \sigma \sqrt{\gamma n  \log(n\vee d) },
 \end{align*}
 since $\log(r) \leq \log(n\vee d)$.  
 
 \ \\ \noindent \textbf{The term (II):}
By a similar argument, it suffices to bound the norm of the coefficient vector, where the coefficient vector ranges over $\alpha$ and $k\neq i$ with
\begin{align*}
(a_2)_{\alpha,k} := U_{kj}  (\Sigma_{k}^{1/2} M_i)_{\alpha} ) .
\end{align*}
Therefore, we see that
\begin{align*}
\| a_2\|_2^2 :&= \sum_{\alpha=1}^d  \sum_{k=1,k\neq i}^{n} U_{kj}^2 (\Sigma_{k}^{1/2} M_i)_{\alpha} ) ^2 \\
&\leq \sum_{\alpha=1}^d  \sum_{k=1,k\neq i}^{n} (\Sigma_{k}^{1/2} M_i)_{\alpha} ) ^2 \\
&\leq    \sum_{k=1,k\neq i}^{n} \| \Sigma_{k}^{1/2} M_i \|^2 \\
&\leq   n \max_{k} \| M_i\t \Sigma_k^{1/2} \|^2 \\
&\leq   n \| M_i \|^2 \max_k \| \Sigma_k^{1/2} \|^2 \\
&\leq  n \|M\|_{2,\infty} \max_k \| \Sigma_k^{1/2} \|^2 \\
&\leq   n \sigma^2 \| M \|_{2,\infty}^2 \\
&\leq \|U\|_{2,\infty}^2  n \sigma^2 \lambda_1^2.
\end{align*}
Therefore, we see that with probability at least $1- 2r\inv n^{-\gamma},$
\begin{align*}
    (II) &\leq \sqrt{2} \|U\|_{2,\infty}\sigma \lambda_1 \sqrt{\gamma\log(n\vee d) n},
\end{align*}
which matches the previous bound.

\ \\ \noindent \textbf{The term (III):}
The final quantity is of the form
 \begin{align*}
 (III) := \sum_{k=1,k\neq i}^n U_{kj} \sum_{\alpha=1}^d \sum_{\beta=1}^d (\Sigma_k^{1/2} \Sigma_i^{1/2})_{\alpha\beta} Y_{i\alpha}Y_{k\beta}.
 \end{align*}
We  use a conditioning argument.  First, consider the event
 \begin{align*}
\mathcal{A} := \{ \sum_{\alpha=1}^{d} Y_{i\alpha} U_{kj} (\Sigma_k^{1/2} \Sigma_i^{1/2})_{\alpha\beta} > s \text{ for any $k$ and $\beta$}\}
\end{align*}
for some $s > 0$ to be determined later.  Note that this event depends on the collection $\{Y_{i\alpha}\}_{\alpha}$ only. Conditional on this event, we see that the sum is a sum of independent mean-zero subguassian random variables with $\psi_2$ norm bounded by $s$.  Since there are $(n-1)d$ such random variables, the generalized Hoeffding inequality shows that
\begin{align*}
\p( | (III) | > t | \mathcal{A} ) \leq 2\exp \bigg( -\frac{1}{2} \frac{t^2}{(n-1)ds^2} \bigg)
\end{align*}
so taking $t = \sqrt{2}\sqrt{(n-1)ds^2(\gamma\log(n\vee d) +\log(r))}$ shows that this holds with probability at least $1 - 2r\inv n^{-\gamma}$. Now I will find a probabilistic bound on $s$.

Note that the sum in the event $\mathcal{A}$ is a sum over $d$ independent random variables, so it suffices to estimate the term
\begin{align*}
\sum_{\alpha=1}^{d}  (U_{kj} (\Sigma_k^{1/2} \Sigma_i^{1/2})_{\alpha\beta})^2
\end{align*}
uniformly over $k,\beta$.  We see that
\begin{align*}
\max_{k,\beta} \sum_{\alpha=1}^{d}  (U_{kj} (\Sigma_k^{1/2} \Sigma_i^{1/2})_{\alpha\beta})^2 &\leq \max_{k,\beta} U_{kj}^2 \sum_{\alpha=1}^{d} (\Sigma_k^{1/2} \Sigma_i^{1/2})_{\alpha\beta}^2 \\
&\leq \max_{k} \| \Sigma_{i}^{1/2} \Sigma_k^{1/2} \|_{2,\infty}^2  U_{kj}^2 \\
&\leq  \| \Sigma_{i}^{1/2}\|_{2,\infty}^2  \|U\|_{2,\infty}^2 \max_{k}\| \Sigma_k^{1/2} \|^2 \\
&\leq  \| \Sigma_{i}^{1/2}\|_{2,\infty}^2  \|U\|_{2,\infty}^2 \sigma^2.
\end{align*}
Therefore, we have that since there are at most $nd$ terms with $k$ and $\beta$,
\begin{align*}
\p( \mathcal{A}) &\leq nd \max_{k,\beta} \p\bigg( \sum_{\alpha=1}^{d} Y_{i\alpha} U_{kj} (\Sigma_k^{1/2} \Sigma_i^{1/2})_{\alpha\beta} > s \text{ for fixed $k$ and $\beta$} \bigg) \leq 2nd \exp \bigg( - \frac{1}{2} \frac{s^2}{\| \Sigma_{i}^{1/2}\|_{2,\infty}^2  \|U\|_{2,\infty}^2 \sigma^2}\bigg),
\end{align*}
so taking $s = \sqrt{2} \sqrt{(\gamma + 1)\log(n\vee d) + \log(d) + \log(r)} \|U\|_{2,\infty} \| \Sigma_{i}^{1/2}\|_{2,\infty} \sigma$ shows that this holds with probability at least $1 - 2 r\inv n^{-\gamma}$.  Therefore, with this fixed choice of $s$, we see that
\begin{align*}
\p( (III) \leq t ) &\leq \p( (III) \leq t | \mathcal{A} ) + \p(\mathcal{A}^c) \\
&\leq 2r\inv n^{-\gamma} + 2r\inv n^{-\gamma}
\end{align*}
Doing the algebra, we see that
\begin{align*}
t &= \sqrt{2}\sqrt{(n-1)d} s \sqrt{\gamma\log(n\vee d) + \log(r)} \\
&= \sqrt{2}\sqrt{(n-1)d}\sqrt{\gamma\log(n\vee d) + \log(r)} \bigg( \sqrt{2} \sqrt{(\gamma + 1)\log(n\vee d) + \log(d) + \log(r)} \|U\|_{2,\infty} \| \Sigma_{i}^{1/2}\|_{2,\infty} \sigma\bigg) \\
&\leq 4\sqrt{nd} \sqrt{(\gamma + 1) \log(n\vee d)} \sqrt{(\gamma + 1)\log(n \vee d)} \|U\|_{2,\infty} \| \Sigma_{i}^{1/2}\|_{2,\infty} \sigma \\
&\leq 8 \sqrt{nd} \|U\|_{2,\infty}  \sigma^2  \gamma\log( n\vee d)
\end{align*}
Letting $\gamma \geq 20$ and taking a union over all $nr$ entries shows that with probability at least $1 - (n\vee d)^{-10}.$,
\begin{align*}
    \| WU\|_{2,\infty} &\leq C_0 \bigg(\sqrt{rnd}\log(n\vee d) \sigma^2  +  \sqrt{ rn\log(n\vee d)} \lambda_1 \sigma \bigg) \|U\|_{2,\infty}.
\end{align*}
Hence, the result holds by also applying Lemma \ref{lem:spectral_norm_concentration} on the term $\| W \|$ in Equation \eqref{lem:2infty_0_1}.  Since the spectral norm bound is smaller than the bound above and holds with probability at least $1 - 4(n\vee d)^{-6}$, the result holds by increasing the constant $C_0$ by a factor of $2$.  
\end{proof}

\leminftygeneral*

\begin{proof}[Proof of Lemma \ref{lem:2infty_general}]
We will prove the result by induction.  When $p = 1$, the result holds by Lemma \ref{lem:2infty_0}, where we take $C_1$ and $C_2$ in the statement of the lemma to be large, fixed constants.  We now fix these constants $C_1$ and $C_2$.  

Let $ p > 1$.  Assume that with probability at least $1 -p (n\vee d)^{-5}$ that for all $1 \leq p_0 \leq p-1$ that 
\begin{align*}
    \|\per (\per W \per)^{p_0-1} W U\|_{2,\infty} &\leq C_1 (C_2\delta)^{p_0} \|U\|_{2,\infty}.
\end{align*}
Note that by the definition of $W := EM\t + ME\t + \Gamma(EE\t)$, we have the identity $\per W \per = \per \Gamma(EE\t) \per$. Let $\mathcal{B}$ be the event that
\begin{align*}
       \mathcal{B}:= \bigg\{\|\per (\per W \per)^{p_0-1} W U\|_{2,\infty} &\leq C_1 (C_2\delta)^{p_0} \|U\|_{2,\infty} \text{ for all $1 \leq p_0 \leq p-1$} \bigg\} \\
    &\bigcap \bigg\{ \|W \| \leq \frac{\delta}{\sqrt{r\log(n\vee d)}} \bigg\} \\
    &\bigcap \bigg\{ \|\Gamma(EE\t) \| \leq \frac{\delta}{\sqrt{r\log(n\vee d)}} \bigg\} .
\end{align*}
Note that $\p(\mathcal{B}) \geq 1 - p(n\vee d)^{-5} -12(n\vee d)^{-6}$ by the induction hypothesis and Lemmas \ref{lem:spectral_norm_concentration} and \ref{lem:diag} (to get the bound on $\Gamma(EE\t)$, we apply Lemma \ref{lem:spectral_norm_concentration} with $M = 0$).

We note that
\begin{align}
\| (\per W \per )^{p-1} W U \|_{2,\infty} &\leq  \| UU\t \Gamma(EE\t)  (\per W \per )^{p-2} W U\|_{2,\infty} \nonumber\\
    &\qquad + \|\Gamma(EE\t) (\per W \per )^{p-2} W U\|_{2,\infty} \nonumber \\
    &\leq  \|U\|_{2,\infty} \|\Gamma(EE\t)\|\|W\|^{p-1} +  \| \Gamma(EE\t)(\per W \per )^{p-2} W U\|_{2,\infty} \nonumber\\
   % &\leq \|U\|_{2,\infty} \|W\|^{k+1} + \|W (\per W \per )^{k-1} W U\|_{2,\infty} \nonumber \\
    &\leq  \|U\|_{2,\infty} \|\Gamma(EE\t)\|\|W\|^{p-1}+ \sqrt{r}  \|\Gamma(EE\t) (\per W \per )^{p-2} W U\|_{\max}. \label{lem:2infty_general_1}
\end{align}
We first will analyze the $(i,j)$ entry of the matrix $\Gamma(EE\t)(\per W \per)^{p-2} W U$.  Fix an index $i$ and consider the auxiliary matrix $X^{-i}$ defined via
\begin{align*}
    X^{-i} := (\per (I - e_i e_i\t)W (I - e_i e_i\t) \per)^{p-2} (I - e_i e_i\t)W (I - e_i e_i\t)  U.
\end{align*}
Note that $X^{-i}$ is independent of the random variable $E_i$.  Define also the matrix $X$ via
\begin{align*}
    X :&= (\per W \per )^{p-2} W U.
\end{align*}
We note that the $i,j$ entry of the matrix $\Gamma(EE\t) X$ can be written as
\begin{align*}
   \sum_{k\neq i} \langle E_i, E_k \rangle X_{kj}^{-i} + %\sum_{k} \langle M_i, E_k\rangle X_{kj}^{-i} 
     \sum_{k\neq i} \langle E_i, E_k \rangle (X_{kj} - X_{kj}^{-i}) %+  \sum_{k} \langle M_i, E_k\rangle (X_{kj} - X_{kj}^{-i})
    := I_1^{ij} + I_2^{ij} 
\end{align*}
where 
\begin{align*}
    I_1^{ij} :&=  \sum_{k\neq i} \langle E_i, E_k \rangle X_{kj}^{-i}; \\ %+ \sum_{k} \langle M_i, E_k\rangle X_{kj}^{-i}; \\
    I_2^{ij} :&= % \sum_{k\neq i} \langle E_i, E_k \rangle (X_{kj}^{-i} - X_{kj}^{-i-k}); \\
      %I_3^{ij} :&= 
      \sum_{k\neq i} \langle E_i, E_k \rangle (X_{kj} - X_{kj}^{-i}).  %+  \sum_{k} \langle M_i, E_k\rangle (X_{kj} - X_{kj}^{-i}).
\end{align*}
Let $\mathcal{A}_{ij}$ be the event that for all $k$, $|X_{kj}^{-i}| \leq 2 C_1 (C_2\delta)^{p-1} \|U\|_{2,\infty}=:A$.  We first study $I_1^{ij}$ on the event $\mathcal{A}_{ij}$. % Note that
% \begin{align*}
%     \p\bigg( |I_1^{ij}| > t/2  \cap \mathcal{A}_{ij}\bigg) &\leq \p\bigg( \bigg| \sum_{k\neq i} \langle E_i, E_k \rangle X_{kj}^{-i} \bigg| > t/4  \cap \mathcal{A}_{ij} \bigg) %+ \p\bigg( \bigg|\sum_{k} \langle M_k, E_i \rangle X_{kj}^{-i} \bigg| > t/4   \cap \mathcal{A}_{ij} \bigg).
% \end{align*}

% By Lemma \ref{lem:events}, on the event $\mathcal{A}_{ij}$ it holds for all $k$ that $|X_{kj}^{-i}| \leq  C_1 (C_2 \delta)^{p-1} \|U\|_{2,\infty}$.  
Note that
 \begin{align*}
  I_1^{ij} &= (E_i)\t\bigg(\sum_{k\neq i} E_k X_{kj}^{-i}\bigg) \\
   &= Y_i\t  \Sigma_i^{1/2} \bigg( (E^{-i})\t X^{-i} e_j \bigg),
 \end{align*}
where $E^{-i}$ is the matrix $E$ with the $i$'th   row  set to zero. Recall that by construction $\mathcal{A}_{ij}$ does not depend on $E_i$.  Hence, conditional on $E_k$ for $k\neq i$, this is a sum of $d$ independent mean-zero random variables $Y_{i\alpha}$ with coefficients indexed by $\alpha$.  Define the vector $v := \Sigma_i^{1/2} (E^{-i})\t X^{-i} e_j \in \R^{d}$.   Note that
\begin{align*}
    \|v\| &= \bigg\| \Sigma_i^{1/2}(E^{-i})\t X^{-i}e_j\bigg\|\\
    &\leq \|\Sigma_i^{1/2}\|\|E^{-i}\|\|X^{-i}e_j\|_2\\
    &\leq \sigma_i\| E^{-i} \| \| X^{-i} e_j \|_2,
\end{align*}
which always holds.  Moreover, on the event $\mathcal{A}_{ij}$ it holds that 
\begin{align*}
 \| X^{-i} e_j \|_2 &\leq \sqrt{n} \max_{k} |X_{kj}^{-i}| \\
 &\leq \sqrt{n} A.
\end{align*}
Suppose for the moment that $\mathcal{E}$ is the event that $\|E^{-i} \| \leq B$, for some bound $B$.  Then by independence of this event with $E_i$, it holds that
\begin{align*}
    \| Y_i\t v \|_{\psi_2}^2 &= \| v\|_2^2.
\end{align*}
Therefore, by Hoeffding's inequality, for some universal constant $C_3$,
\begin{align}
    \p\bigg(  |Y_i\t \Sigma_i^{1/2} (E^{-i})\t X^{-i} e_j | > C_3 \sqrt{n}\sigma_i A B \sqrt{20\log(n\vee d)} \cap \mathcal{A}_{ij} \cap \mathcal{E} \bigg) &\leq 2 (n\vee d)^{-20}.\label{eq:lem6hoeffcalc}
\end{align}
We now deduce a bound $B$ for the spectral norm of $E^{-i}$.  Note that for any deterministic unit vectors $a \in \R^{n}$ and $b \in \R^d$, by independence of the $Y_{k\alpha}$'s it holds that
\begin{align*}
 \|   a\t E^{-i} b \|_{\psi_2}^2  &= \| \sum_{k\neq i} \sum_{j} a_k E^{-i}_{kj} b_j \|_{\psi_2}^2 \\
    &=\bigg\| \sum_{k\neq i} \sum_{j} a_k (\Sigma_k^{1/2} Y_k)_{j} b_j \bigg\|_{\psi_2}^2 \\
    &=  \sum_{k\neq i} a_k^2 \bigg\| \sum_{j} (\Sigma_k^{1/2} Y_k)_{j} b_j\bigg\|_{\psi_2}^2 \\
  &=  \sum_{k\neq i} a_k^2\bigg\| \sum_{j} \sum_{\alpha} (\Sigma_k^{1/2})_{j\alpha} Y_{k\alpha} b_j\bigg\|_{\psi_2}^2 \\
  &= \sum_{k\neq i} a_k^2 \sum_{\alpha} \bigg\| Y_{k\alpha} \sum_{j} (\Sigma_k^{1/2})_{j\alpha} b_j\bigg\|_{\psi_2}^2 \\
  &=  \sum_{k\neq i} a_k^2 \sum_{\alpha} \| \sum_{j} (\Sigma_k^{1/2})_{j\alpha} b_j \|_2^2 \\
  &=  \sum_{k\neq i} a_k^2 \sum_{\alpha}  [ \Sigma_k^{1/2} b]_{\alpha}^2 \\
  &= \sum_{k\neq i} a_k^2  \| \Sigma_k^{1/2} b \|_2^2 \\
  &\leq \sigma^2  \sum_{k\neq i} a_k^2  \|b\|^2 \\
  &\leq \sigma^2.
\end{align*}
Therefore, by a standard $\eps$-net argument (e.g. the argument in the proof of Theorem 4.4.5 in \citet{vershynin_high-dimensional_2018}), it holds that there exists a universal constant $C_4$ such that
\begin{align*}
    \p\bigg( \|E^{-i} \| > C_4 \sigma ( \sqrt{d} + u) \bigg) &\leq 2 \exp(-u^2).
\end{align*}
Here we implicitly used Assumption \ref{assumption:dimension}, or that $d\geq c n$.  Define the event
\begin{align*}
    \mathcal{E} := \bigg\{ \|E^{-i} \| \leq C_4 \sigma ( \sqrt{d} + \sqrt{20\log(n\vee d)}) \bigg\},
\end{align*}
so that $\p(\mathcal{E}^c) \leq 2(n\vee d)^{-20}$.   On the event $\mathcal{E}$ it holds that $\|E^{-i}\| \leq C_4 \sigma \sqrt{20 d\log(n\vee d)}$.
By Equation~(\ref{eq:lem6hoeffcalc}),
\begin{align*}
    \p\bigg(  |Y_i\t \Sigma_i^{1/2} (E^{-i})\t X^{-i} e_j | > & 20 C_3 C_4 \sigma^2 \sqrt{nd}A \log(n\vee d)  \cap \mathcal{A}_{ij} \bigg) \\ &\leq   \p\bigg(  |Y_i\t \Sigma_i^{1/2} (E^{-i})\t X^{-i} e_j | > 20 C_3 C_4 \sigma^2 \sqrt{nd}A \log(n\vee d)  \cap \mathcal{A}_{ij} \cap \mathcal{E} \bigg) + \p( \mathcal{E}^c ) \\
    &\leq 4(n\vee d)^{-20}.
\end{align*}
\noindent
Furthermore, $A := 2C_1(C_2\delta)^{p-1}\|U\|_{2,\infty}$. % and $\|M\|_{2,\infty} \leq  \lambda_1$.  
  Recall that $\delta$ satisfies, for some sufficiently large absolute constant $C_0$,
 \begin{align*}
     \delta &= C_0 \bigg( \sqrt{rnd} \log(n\vee d) \sigma^2 + \sqrt{rn \log(n\vee d)} \sigma\lambda_1 \bigg),
 \end{align*}
 so that $\frac{\delta}{\sqrt{r}} \geq \log(n\vee d) \sqrt{nd} \sigma^2 + \lambda_1 \sqrt{n \log(n\vee d)} \sigma $.   Therefore,
 \begin{align*}
     \p\bigg( |I_1^{ij} | > \frac{C_1 (C_2\delta)^{p}}{4\sqrt{r}} \|U\|_{2,\infty} \cap \mathcal{A}_{ij} \bigg) \leq 4(n\vee d)^{-20}
 \end{align*}
 as long as $C_2 \geq 20 C_3 C_4$, which is true by taking $C_2$ large since $C_3$ and $C_4$ are universal constants.  Therefore, from equation \eqref{lem:2infty_general_1}, 
\begin{align}
\p\bigg( &\|U\|_{2,\infty} \|\Gamma(EE\t)\|\| W \|^{p-1} + \sqrt{r} \|\Gamma(EE\t)(\per W \per)^{p-2} W U\|_{\max} > C_1(C_2\delta)^{p} \|U\|_{2,\infty}\bigg) \nonumber\\
&\leq \p\bigg( \|U\|_{2,\infty} \|\Gamma(EE\t)\|\| W \|^{p-1}+ \sqrt{r} \|\Gamma(EE\t)(\per W \per)^{p-2} W U\|_{\max} > C_1(C_2\delta)^{p} \|U\|_{2,\infty}\cap \mathcal{B} \bigg) + \p(\mathcal{B}^c) \nonumber\\
&\leq  \p\bigg( \|\Gamma(EE\t)(\per W \per)^{p-2} W U\|_{\max} > \frac{C_1}{2\sqrt{r}}(C_2\delta)^{p} \|U\|_{2,\infty}\cap \mathcal{B} \bigg) + \p(\mathcal{B}^c) \nonumber \\
&\leq \p\bigg( \bigcup_{i,j}  \bigg|e_i\t \Gamma(EE\t)(\per W \per)^{p-2} W U e_j\bigg| > \frac{C_1}{2\sqrt{r}}(C_2\delta)^{p} \|U\|_{2,\infty}\cap \mathcal{B} \bigg) + \p(\mathcal{B}^c) \nonumber\\
&\leq nr \max_{i,j} \p\bigg(   |\Gamma(EE\t)(\per W \per)^{p-2} W U|_{ij} > \frac{C_1}{2\sqrt{r}}(C_2\delta)^{p} \|U\|_{2,\infty}\cap \mathcal{B} \bigg) + \p(\mathcal{B}^c) \nonumber\\
&\leq nr \max_{i,j} \p \bigg( |I_1^{ij} + I_2^{ij} | > \frac{C_1}{2\sqrt{r}}(C_2\delta)^{p} \|U\|_{2,\infty}\cap \mathcal{B} \bigg) + \p(\mathcal{B}^c) \nonumber\\
&\leq nr \max_{i,j} \p \bigg( |I_1^{ij}| >  \frac{C_1}{4\sqrt{r}}(C_2\delta)^{p} \|U\|_{2,\infty}\cap \mathcal{B} \bigg)  + nr \max_{i,j} \p \bigg( |I_2^{ij}| >  \frac{C_1}{4\sqrt{r}}(C_2\delta)^{p} \|U\|_{2,\infty}\cap \mathcal{B} \bigg) + \p(\mathcal{B}^c) \nonumber \\
    &\leq nr \max_{i,j} \p\bigg( | I_1^{ij} \mathcal| > \frac{C_1}{4\sqrt{r}}(C_2\delta)^{p} \|U\|_{2,\infty}\cap \mathcal{B} \cap \mathcal{A}_{ij} \bigg) + nr \max_{i,j} \p ( \mathcal{B} \cap \mathcal{A}_{ij}^c) \nonumber \\
    &\qquad   + nr \max_{i,j} \p \bigg( |I_2^{ij}| >  \frac{C_1}{4\sqrt{r}}(C_2\delta)^{p} \|U\|_{2,\infty}\cap \mathcal{B} \bigg) + \p(\mathcal{B}^c). %\nonumber \\
 %   &\leq nr \max_{i,j} \p\bigg( | I_1^{ij}| > \frac{C_1}{4\sqrt{r}}(C_2\delta)^{p} \|U\|_{2,\infty} \cap \mathcal{A}_{ij} \bigg) + nr \max_{i,j} \p\bigg( | I_2^{ij}| > \frac{C_1}{4\sqrt{r}}(C_2\delta)^{p} \|U\|_{2,\infty} \cap \mathcal{A}_{ij} \cap \mathcal{B} \bigg) \nonumber \\
  %  &\qquad + 
   % nr \max_{i,j} \p ( \mathcal{B} \cap \mathcal{A}_{ij}^c) 
   %+ \p(\mathcal{B}^c).  
   \label{probabilitysums}
\end{align}
In Lemma \ref{lem:events} we show for all $i$ and $j$ that $\p(\mathcal{B} \cap \mathcal{A}_{ij}^c ) = 0$ and
\begin{align*}
\p\bigg( | I_2^{ij} | >\frac{C_1}{4\sqrt{r}}(C_2\delta)^{p} \|U\|_{2,\infty} \cap \mathcal{B} \bigg) &= 0.
\end{align*}
% We also show that
% \begin{align*}
%     nr\p\bigg( |I_2^{ij} | > \frac{C_1}{4\sqrt{r}} (C_2\delta)^p \|U\|_{2,\infty} \cap \mathcal{A}_{ij} \cap \mathcal{B} \bigg) &\leq 4 (n\vee d)^{-20}.
% \end{align*}
In addition, from our previous analysis,
\begin{align*}
    \p\bigg( |I_1^{ij} | > \frac{C_1 (C_2\delta)^{p}}{4\sqrt{r}} \|U\|_{2,\infty} \cap \mathcal{A}_{ij} \bigg) \leq 4(n\vee d)^{-20}.
\end{align*}
Finally, $\p(\mathcal{B}^c) \leq p(n\vee d)^{-5} + 12(n\vee d)^{-6}$ by the induction hypothesis.  Plugging  these results into Expression \eqref{probabilitysums}, we obtain
\begin{align*}
    nr \p\bigg( | I_1^{ij}| > \frac{C_1}{4\sqrt{r}}(C_2\delta)^{p} \|U\|_{2,\infty} \cap \mathcal{A}_{ij} \bigg) + nr & \p ( \mathcal{B} \cap \mathcal{A}_{ij}^c)  + \p\bigg( | I_2^{ij} | >\frac{C_1}{4\sqrt{r}}(C_2\delta)^{p} \|U\|_{2,\infty} \cap \mathcal{B} \bigg) + \p(\mathcal{B}^c) \\
    &\leq 4(nr)(n\vee d)^{-20} + p(n\vee d)^{-5} + 12(n\vee d)^{-6} \\
    &\leq (p+1)(n\vee d)^{-5}
\end{align*}
as desired.  
\end{proof}

\begin{lemma}\label{lem:events}
Define $X^{-i}$ and $X$ as in the proof of Lemma \ref{lem:2infty_general}.  %Define
%\begin{align*}
%    X_{kj} := e_k\t (\per \Gamma(EE\t)\per)^{k-1} (EM\t + \Gamma(EE\t))Ue_j.
%\end{align*}
Let $\mathcal{A}_{ij}$ and $\mathcal{B}$ be the events defined via 
\begin{align*}
\mathcal{A}_{ij}:= \bigg\{    \max_{k} | X_{kj}^{-i} | \leq 2 C_1(C_2\delta)^{p-1}\|U\|_{2,\infty} \bigg\}& \\
        \mathcal{B}:= \bigg\{\|\per (\per W \per)^{p_0-1} W U\|_{2,\infty} &\leq C_1 (C_2\delta)^{p_0} \|U\|_{2,\infty} \text{ for all $1 \leq p_0 \leq p-1$} \bigg\} \\
    &\bigcap \bigg\{ \|W \| \leq \frac{\delta}{\sqrt{r\log(n\vee d)}} \bigg\} \\
    &\bigcap \bigg\{ \|\Gamma(EE\t) \| \leq \frac{\delta}{\sqrt{r\log(n\vee d)}} \bigg\}. 
\end{align*}
Then $\mathcal{B} \cap \mathcal{A}^c_{ij}$ is empty for all $i$ and $j$.  Define also $I_2^{ij} := e_i\t \Gamma(EE\t) (X - X^{-i})e_j.$ as in the proof of Lemma \ref{lem:2infty_general}. Then  again for all $i$ and $j$,
\begin{align*}
    \p\bigg( |I_2^{ij}| > \frac{C_1}{4\sqrt{r}} (C_2 \delta)^{p}\|U\|_{2,\infty} \cap \mathcal{B} \bigg) = 0.
\end{align*}
\end{lemma}

\begin{proof}[Proof of Lemma \ref{lem:events}]
The proof will be split up into steps, the first of which will be expanding out the difference $X - X^{-i}$ in terms of a matrix telescoping series, the second of which will be bounding individual terms, and the final step will prove the two results.

\ \\ \noindent 
\textbf{Step 1: A useful expansion}\\ \ \\
We have that
\begin{align*}
    X &= (\per W\per)^{p-2} WU \\
    X^{-i} &= (\per (I - e_i e_i\t) W (I - e_i e_i\t) \per)^{p-2} (I - e_i e_i\t) W (I - e_i e_i\t)  U.
\end{align*}
Define the matrix $\xi := W - (I- e_ie_i\t) W (I - e_i e_i\t)$.  We note that
\begin{align*}
    X - X^{-i} &= (\per W \per )^{p-2} W U - (\per (I - e_i e_i\t) W (I - e_i e_i\t) \per)^{p-2} (I - e_i e_i\t) W (I - e_i e_i\t)  U \\
    &= (\per W \per )^{p-2} W U -  (\per (I - e_i e_i\t) W (I - e_i e_i\t) \per)^{p-2} ( W - \xi) U \\
    &= \bigg[ (\per W \per )^{p-2} - (\per (I - e_i e_i\t) W (I - e_i e_i\t) \per)^{p-2} \bigg] W U \\
    &\qquad + (\per (I - e_i e_i\t) W (I - e_i e_i\t) \per)^{p-2} \xi U \\
    &= \bigg[ (\per W \per )^{p-2} - (\per (I - e_i e_i\t) W (I - e_i e_i\t) \per)^{p-3}(\per (W - \xi) \per) \bigg] W U
    \\
    &\qquad + (\per (I - e_i e_i\t) W (I - e_i e_i\t) \per)^{p-2} \xi U \\
    &= \bigg[  \bigg( ( \per W \per)^{p-3} - (\per (I - e_i e_i\t) W (I - e_i e_i\t) \per)^{p-3} \bigg) (\per W \per ) \bigg] W U \\
    &\qquad +  (\per (I - e_i e_i\t) W (I - e_i e_i\t) \per)^{p-3} (\per \xi \per) W U \\
    &\qquad + (\per (I - e_i e_i\t) W (I - e_i e_i\t) \per)^{p-2} \xi U.
\end{align*}
Note that if $ p = 2$ then we simply have $\xi U$.  Define the matrices
\begin{align*}
    S_{\xi} &:= \per \xi \per \\
    S_{-i} :&= (\per(I-e_i e_i\t)W(I-e_i e_i\t)\per) \\
    S :&= \per W \per.
\end{align*}
Then iterating the process above, it holds that
\begin{align*}
    X - X^{-i} &= \sum_{m=1}^{p-2}(\per(I-e_i e_i\t)W(I-e_i e_i\t)\per)^{m-1}(\per\xi\per)(\per W\per)^{p-m-2}WU\\
&+(\per(I-e_i e_i\t)W(I-e_i e_i\t)\per)^{p-2}\xi U \\
&= \sum_{m=1}^{p-2} S_{-i}^{m-1} S_{\xi} S^{p-m-2} W U + S_{-i}^{p-2} \xi U,
\end{align*}
where the sum is understood to be the empty sum if $p = 2$.  

%\begin{align*}
%\sum_{m=1}^{p-1}&(\per(I-e_i e_i\t)W(I-e_i e_i\t)\per)^{m-1}(\per\xi\per)(\per W\per)^{p-1-m}WU\\
%&+(\per(I-e_i e_i\t)W(I-e_i e_i\t)\per)^{p-1}\xi U.
%\end{align*}

%Define the matrix $S_{\xi} := (\per \xi \per)$ and the matrix $S := (\per W \per)$.  Iterating this process, we %see that we have the expansion
%\begin{align*}
%    X - X^{-i} &= \sum_{\substack{0 \leq t_i \leq k-1\\ t_1 \neq p-1\\ t_1 + \cdots + t_{k} = k-1}} (-1)^{\tau(t)} S^{t_1} S_{\xi}^{t_2} \cdots S_{\xi}^{t_{k-1}} S^{t_k} W U + (\per (I - e_i e_i\t) W (I - e_i e_i\t) \per)^{p-1} \xi U,
%\end{align*}
%where $\tau(t)$ is a power depending on the number of nonzero entries in the vector $t = (t_1, \dots, t_{k-1})$.  Note that each element of this sum has at least one factor of $\xi$ appearing.  Also note that when $k = 1$, this expansion holds by considering the empty sum (so we just have $\xi U$).  

\ \\ \noindent
\textbf{Step 2: Analyzing each term in the sum} \\ \ \\
We now analyze each individual term in the sum on the event $\mathcal{B}$, where we also analyze each row of $W(X - X^{-i})$. Let the matrix $B \in \{\Gamma(EE\t), I\}$ be fixed.  We ignore the boundary term $BS_{-i}^{p-2} \xi U$ for now.

Note that the $k$'th row of such term can be written as
\begin{align*}
    e_k\t  B S_{-i}^{m-1} S_{\xi} S^{p-m-2} W U,
\end{align*}
Recall that
\begin{align*}
    S_{\xi} &= \per \xi \per \\
    &= \per \bigg( e_i e_i\t W +  W e_i e_i\t - e_i e_i\t W e_i e_i\t \bigg) \per \\
    &= \per e_i e_i\t W \per + \per W e_i e_i\t \per - \per e_i e_i\t W e_i e_i\t \per.
\end{align*}
Therefore, by homogeneity of the vector norm, we have that
\begin{align*}
  \|  e_k\t  B &S_{-i}^{m-1} S_{\xi}S^{p-m-2} W U \| \\
  &\leq  \| e_k\t  B S_{-i}^{m-1} \per e_i e_i\t W \per S^{p-m-2} W U \| +  \|  e_k\t  B S_{-i}^{m-1} \per W e_i e_i\t \per S^{p-m-2} W U \| 
  \\
  &\qquad + \|  e_k\t B S_{-i}^{m-1} \per e_i e_i\t W e_i e_i\t \per S^{p-m-2} W U \|  \\
  &\leq \| e_k\t B S_{-i}^{m-1} \per  e_i \| \| e_i\t W \per S^{p-m-2} W U\|  + \| e_k\t B S_{-i}^{m-1} \per W e_i \| \| e_i\t \per S^{p-m-2} W U \| \\
  &\qquad + \| e_k\t B S_{-i}^{m-1}\per e_i \| \| e_i\t W e_i \| \|e_i\t \per S^{p-m-2} W U\| \\
  &\leq \| e_k\t B S_{-i}^{m-1} \per  e_i \|\bigg( \| e_i\t \per W \per S^{p-m-2} W U\| + \| e_i\t U U\t W \per S^{p-m-2} W U\| \bigg) \\
  &\qquad + \| e_k\t B S_{-i}^{m-1} \per W e_i \| \| e_i\t \per S^{p-m-2} W U \|  + \| e_k\t B S_{-i}^{m-1}\per e_i \| \| e_i\t W e_i \| \|e_i\t \per S^{p-m-2} W U\| \\
  &\leq \| B S_{-i}^{m-1} \| \| \per W \per S^{p-m-2} W U\|_{2,\infty}  + \|B S_{-i}^{m-1} \| \|U\|_{2,\infty} \| W \per S^{p-m-2} W U\| \\
  &\qquad + \| B S_{-i}^{m-1} \per W \| \| \per S^{p-m-2} W U \|_{2,\infty}  + \| B S_{-i}^{m-1} \| \| W \| \|\per S^{p-m-2} W U\|_{2,\infty} \\
  &\leq \| B S_{-i}^{m-1} \| \|S^{p -m-1} W U\|_{2,\infty} + \|U \|_{2,\infty} \|B S_{-i}^{m-1} \| \|W \|^{p-m} + \| B S_{-i}^{m-1} \| \| W \| \| \per S^{p-m-2} W U\|_{2,\infty} \\
  &\qquad + \| B S_{-i}^{m-1} \| \| W \| \|\per S^{p-m-2} W U\|_{2,\infty} \\
  &\leq \| B S_{-i}^{m-1} \| \bigg( C_1 (C_2 \delta)^{p-m} \|U\|_{2,\infty} + \|U\|_{2,\infty} \| W\|^{p-m} + 2\|W \| C_1 (C_2 \delta)^{p-m-1} \|U\|_{2,\infty} \bigg) \\
  &\leq \|B S_{-i}^{m-1} \| \|U\|_{2,\infty} \bigg( C_1 (C_2 \delta)^{p-m} +  \bigg[\frac{\delta}{\sqrt{r\log(n \vee d)}} \bigg]^{p-m} + 2\frac{\delta}{\sqrt{r\log(n \vee d)}} C_1 (C_2 \delta)^{p-m-1} \bigg).
\end{align*}
But for $\|B S_{-i}^{m-1}\|$, this bound is uniform in $i$ and $j$. Finally, for the boundary term, we note that the same strategy can be applied in precisely the same manner, yielding the same bound.  

\ \\ \noindent
\textbf{Step 3: Putting it together}\\ \ \\
First we show the upper bound on $I_2^{ij}$ on $\mathcal{B}$.  Recall that
\begin{align*}
    I_2^{ij} = e_i\t \Gamma(EE\t) (X - X^{-i}) e_j.
\end{align*}
By the bounds on each term, we have that
\begin{align*}
    | e_i\t  \Gamma(EE\t) (X - X^{-i}) e_j| &\leq \| e_i\t  \Gamma(EE\t) (X - X^{-i}) \| \\
    &\leq \sum_{m=1}^{p-2}  \| \Gamma(EE\t)\| \|W\|^{m-1} \|U\|_{2,\infty} \bigg( C_1 (C_2 \delta)^{p-m} +  \bigg[\frac{\delta}{\sqrt{r\log(n \vee d)}} \bigg]^{p-m} + 2\frac{\delta}{\sqrt{r\log(n \vee d)}} C_1 (C_2 \delta)^{p-m-1} \bigg) \\
    &\leq 4 C_1 \|U\|_{2,\infty} \sum_{m=1}^{p-2}   \| \Gamma(EE\t)\| \|W\|^{m-1} (C_2 \delta)^{p-m} \\
    &\leq 4 C_1 \|U\|_{2,\infty}\sum_{m=1}^{p-2}  \bigg( \frac{\delta}{\sqrt{r\log(n \vee d)}}\bigg)^m (C_2 \delta)^{p-m} \\
    &\leq \frac{4 C_1 \|U\|_{2,\infty} }{\sqrt{r\log(n\vee d)}} (C_2 \delta)^{p} \sum_{m=1}^{p-2}  C_2^{-m} \\
    &\leq \frac{C_1}{4 \sqrt{r}} (C_2 \delta)^{p} \|U\|_{2,\infty}
\end{align*}
as long as $C_2 \geq 48$.

Next we show that $\mathcal{B} \cap \mathcal{A}^c_{ij}$ is empty for all $i$ and $j$.  More specifically, we show that
\begin{align*}
    | X_{kj}^{-i} - X_{kj} | &\leq C_1 (C_2 \delta)^{p-1} \|U\|_{2,\infty},
\end{align*}
Again upper bounding the $k,j$ entry by the $k$'th row norm, we note that the $k$'th row of $X - X^{-i}$ can be written as
%\begin{align*}
$ e_k\t ( X - X^{-i} ).
$ %\end{align*}
 Using the expansion and the bounds on each term in the summation, we have that
\begin{align*}
    \| e_k\t (X - X^{-i}) \| &\leq \sum_{m=1}^{p-1}\|W\|^{m-1} \|U\|_{2,\infty} \bigg( C_1 (C_2 \delta)^{p-m} +  \bigg[\frac{\delta}{\sqrt{r\log(n \vee d)}} \bigg]^{p-m} + \frac{\delta}{\sqrt{r\log(n \vee d)}} C_1 (C_2 \delta)^{p-m-1} \bigg) \\
    &\leq 4 C_1 \|U\|_{2,\infty}  (C_2 \delta)^{p-1} \sum_{m=1} C_2^{-m} \\
    &\leq C_1 \|U\|_{2,\infty} (C_2 \delta)^{p-1},
\end{align*}
and hence on $\mathcal{B}$ we have that 
\begin{align*}
    \| e_k\t  X^{-i} e_j\| &\leq \|X\|_{2,\infty} + \| e_k\t (X - X^{-i})\|  \\
    &\leq 2 C_1 ( C_2 \delta)^{p-1} \|U\|_{2,\infty}
\end{align*}
as desired.  Therefore $\mathcal{B} \cap \mathcal{A}^c_{ij}$ is empty.

\end{proof}

\section{Proof of Lemmas in Section \ref{sec:deterministicpart}}

\subsection{Proof of Lemma \ref{lem:deterministic_spectral}}
We will need the following lemma, adapted from \citet{zhang_heteroskedastic_2021}.  

\begin{lemma}[Lemma 1 in \citet{zhang_heteroskedastic_2021}]\label{lemma1_spectral}
Let $U,V \in \mathbb{O}(n,r)$ and let $A$ be a fixed $n\times n$ matrix.  Then
\begin{align*}
\|G(UU^{\t} A)\|&\leq \|U\|_{2,\infty}\|A\|\\
    \| G(AVV\t) \| &\leq \| V\|_{2,\infty} \| A \|. 
\end{align*}
\end{lemma}

\begin{proof}[Proof of Lemma \ref{lem:deterministic_spectral}]
%\textcolor{red}{
First, if $T_0 \geq C \log\left(\frac{\lambda_r^2}{\|\Gamma(Z)\|} \right)$, then by \citet{zhang_heteroskedastic_2021} it holds that $\|N_T - A \| \leq 3 \|\Gamma(Z)\|$. In addition, supposing the result holds and that $\rho \leq \frac{1}{2}$, then when $T - T_0 \geq C \log\left( \frac{1}{\|U\|_{2,\infty}}\right)$ it holds that
\begin{align*}
    \| \hat A - \tilde A\| \leq 41 \|U\|_{2,\infty} \|\Gamma(Z)\|.
\end{align*}
Hence, we must have that
\begin{align*}
    T  &\geq C \log\left( \frac{1}{\|U\|_{2,\infty}}\right) + T_0 \\
    &\geq C \bigg( \log\left( \frac{1}{\|U\|_{2,\infty}}\right) + \log \left(\frac{\lambda_r^2}{\|\Gamma(Z)\|} \right) \bigg) \\
    &\geq C \bigg(  \log \left(\frac{\lambda_r^2}{\|U\|_{2,\infty}\|\Gamma(Z)\|} \right) \bigg).
\end{align*}
This proves the ``consequently'' part.  We now show that
\begin{align*}
    \tilde K_T &\leq \frac{4}{\rho^{T - T_0}} \|\Gamma(Z)\| + \frac{20}{1-\rho} \|U\|_{2,\infty}\|\Gamma(Z)\|.
\end{align*}
%}
We have that
\begin{align*}
\| N_T - \tilde A \| &= \| G(N_T - \tilde A) \| \\
&\leq \left\|G\left(\left(P_{U^{ T-1}}-P_{\tilde{U}}\right)\left(N_{T-1}-\tilde{A}\right)\right)\right\|+\left\|G\left(P_{\tilde U}\left(N_{T-1}-\tilde{A}\right)\right)\right\|+\left\|G\left(P_{U_{\perp}^{T-1}} \tilde{A}\right)\right\| \\
&\leq \left\|G\left(\left(P_{U ^{T-1}}-P_{\tilde{U}}\right)\left(N_{T-1}-\tilde{A}\right)\right)\right\|+\left\|G\left(P_{\tilde U}\left(N_{T-1}-\tilde{A}\right)\right)\right\|+\left\|G\left(P_{U_{\perp}^{T-1}}A\right)\right\| + \| G (P_{U_{\perp}^{T-1}} \Gamma(Z) ) \| \\
&\leq \left\|G\left(\left(P_{U^{ T-1}}-P_{\tilde{U}}\right)\left(N_{T-1}-\tilde{A}\right)\right)\right\|+\left\|G\left(P_{\tilde U}\left(N_{T-1}-\tilde{A}\right)\right)\right\|+\left\|G\left(P_{U_{\perp}^{T-1}}A\right)\right\| \\
&\qquad + \| G (P_{U_{\perp}^{T-1}} - P_{\tilde U_{\perp}}) \Gamma(Z) ) \|  + \| G (P_{\tilde U_{\perp}} \Gamma(Z) ) \| \\
&= J_1 + J_2 + J_3 + J_4 + J_5.
\end{align*}
we now bound each term.

\ \\ \noindent
\textbf{The term $J_1$:}
We use the restricted-rank operator norm of $G$ to bound this term, since $\mathrm{rank}(P_{U^{T-1}}-P_{\tilde{U}})\leq 2r$:
$$\|G((P_{U^{T-1}}-P_{\tilde{U}})(N_{T-1}-\tilde{A}))\|\leq  \|P_{U^{T-1}}-P_{\tilde{U}}\|\|N_{T-1}-\tilde{A}\|.$$ 

\ \\ \noindent
\textbf{The term $J_2$:}  By Lemma~\ref{lemma1_spectral},  $$\|G(P_{\tilde{U}}(N_{T-1}-\tilde{A}))\|\leq \|\tilde{U}\|_{2,\infty}\|N_{T-1}-\tilde{A}\|.$$

\ \\ \noindent
\textbf{The term $J_3$:}  By Lemmas~\ref{lemma1_spectral} and \ref{lemma7_spectral}, $$\|G(P_{U_\perp^{T-1}}A)\|=\|G(P_{U_\perp^{T-1}}A P_{U})\|\leq \|U\|_{2,\infty}\|P_{U_\perp^{T-1}}A\|\leq 2\|U\|_{2,\infty}\|N_{T-1}-A\|.$$

\ \\ \noindent
\textbf{The term $J_4$:}  Since $P_{U_{\perp}^{T-1}}-P_{\tilde{U}_{\perp}}=(I-P_{\tilde{U}_{\perp}})-(I-P_{U_{\perp}^{T-1}})=P_{\tilde{U}}-P_{U^{T-1}},$ we proceed as we did for $J_1$, obtaining $$\|G((P_{U_{\perp}^{T-1}}-P_{\tilde{U}_{\perp}})\Gamma(Z)\|\leq \|P_{U^{T-1}}-P_{\tilde{U}}\| \|\Gamma(Z)\|.$$

\ \\ \noindent
\textbf{The term $J_5$:}  We have $G(P_{\tilde U_{\perp}} \Gamma(Z)) = G( \Gamma(Z)) - G(P_{\tilde U} \Gamma(Z)) = - G(P_{\tilde U} \Gamma(Z)) $. By Lemma~\ref{lemma1_spectral},  $$\|G(P_{\tilde{U}_{\perp}}\Gamma(Z))\|\leq \|\tilde{U}\|_{2,\infty}\|\Gamma(Z)\|.$$

\ \\ \noindent
\textbf{Putting it together:} 
Let $K_{T-1} := \| N_{T-1} - A\|$ and let $\tilde K_{T-1} := \| N_{T-1} - \tilde A\|$.  Compiling these bounds, we have that
\begin{align*}
    J_1 &\leq  \| P_{U^{T-1}} - P_{\tilde U} \| \tilde K_{T-1} \\
    J_2 &\leq \|\tilde U\|_{2.\infty} \tilde K_{T-1} \\
    J_3 &\leq 2 \|U\|_{2,\infty} K_{T-1} \\
    J_4 &\leq  \| P_{U^{T-1}} - P_{\tilde U} \| \| \Gamma(Z) \| \\
    J_5 &\leq \|\tilde U\|_{2,\infty} \|\Gamma(Z) \|.
\end{align*}
These bounds hold regardless of $T$.  Hence, we may take $T$ such that $\|N_{T-1} - A\|\leq 3 \| \Gamma(Z)\|$ %\textcolor{red}{
(by \citet{zhang_heteroskedastic_2021}, we may take $T \geq C \log \left( \frac{\lambda_r^2}{\|\Gamma(Z)\|} \right)$.%}
%(which exists by \citet{zhang_heteroskedastic_2021} under our assumptions). 
The proof of Lemma~\ref{lem:eigengaps} shows that $\|\Gamma(Z)\|\leq \lambda_r^2/12$, so $\tilde{K}_{T-1}\leq 4(\lambda_r^2/12)\leq \lambda_r^2/2$, and by the Davis-Kahan theorem, we have that
\begin{align*}
    \| P_{U^{T-1}} - P_{\tilde U} \| &\leq 2\frac{ \tilde K_{T-1} }{\lambda_r^2}.
\end{align*}
Applying this to the above bounds, we see that
\begin{align*}
        J_1 &\leq 2 \frac{\tilde K_{T-1}^2}{\lambda_r^2} \\
        J_4 &\leq 2 \frac{\tilde K_{T-1}}{\lambda_r^2} \| \Gamma(Z) \|.
\end{align*}
Moreover, we have the trivial bound
\begin{align*}
    K_{T-1} &\leq \tilde K_{T-1} + \| \Gamma(Z)\|.
\end{align*}
Hence, we see that
\begin{align*}
     J_1 &\leq 2 \frac{\tilde K_{T-1}^2}{\lambda_r^2} \\
    J_2 &\leq \|\tilde U\|_{2.\infty} \tilde K_{T-1} \\
    J_3 &\leq 2 \|U\|_{2,\infty} \bigg(\tilde K_{T-1} + \| \Gamma(Z)\| \bigg)\\
    J_4 &\leq  2 \frac{\tilde K_{T-1}}{\lambda_r^2} \| \Gamma(Z) \|\\
    J_5 &\leq \|\tilde U\|_{2,\infty} \|\Gamma(Z) \|.
\end{align*}
Now, for $T_0$ such that $K_{T_0} \leq 3 \| \Gamma(Z)\|$, $\tilde{K}_{T_0}\leq 4\|\Gamma(Z)\|,$ and we see that we have the initial bound
\begin{align*}
    \tilde K_{T_0 + 1} &\leq  40 \frac{ \| \Gamma(Z) \|^2}{\lambda_r^2} + 5\| \tilde U\|_{2,\infty} \|\Gamma(Z) \|  + 10 \|U\|_{2,\infty}\|\Gamma(Z) \|.
\end{align*}
On the event in Theorem~\ref{thm:2_infty_random} and Assumption \ref{assumption:eigengap}, once $n$ and $d$ are large enough, $\|\tilde{U}-UU\t\tilde{U}\|_{2,\infty}\leq \|U\|_{2,\infty}$, so $$\|\tilde{U}\|_{2,\infty}\leq \|\tilde{U}-UU\t\tilde{U}\|_{2,\infty}+\|UU\t\tilde{U}\|_{2,\infty}\leq 2\|U\|_{2,\infty},
$$ since $\|U\t\tilde{U}\|\leq 1$. This gives $$\tilde{K}_{T_0+1}\leq 40 \frac{\|\Gamma(Z)\|^2}{\lambda_r^2}+20\|U\|_{2,\infty}\|\Gamma(Z)\|.$$
Let $\rho=10 \|\Gamma(Z)\|/\lambda_r^2$, and suppose that for $ T - 1\geq T_0$ we have the bound
\begin{align*}
    \tilde K_{T-1} &\leq 4 \rho^{T-1-T_0} \|\Gamma(Z)\| + \frac{20}{1-\rho} \|U\|_{2,\infty} \|\Gamma(Z)\|.
\end{align*}
Clearly for $T_0$ we have this bound. With the recursion, and using $\tilde{K}_{T-1}\leq 4\|\Gamma(Z)\|$ and $\|\tilde{U}\|_{2,\infty}\leq 2\|U\|_{2,\infty}$, we get
\begin{align*}
\tilde K_T & \leq 10 \frac{\tilde{K}_{T-1}}{\lambda_r^2}\|\Gamma(Z)\|+20\|U\|_{2,\infty}\|\Gamma(Z)\|\\
&\leq \frac{10}{\lambda_r^2}\left[4\rho^{T-1-T_0}\|\Gamma(Z)\|+\frac{20}{1-\rho}\|U\|_{2,\infty}\|\Gamma(Z)\|\right]\|\Gamma(Z)\|+20\|U\|_{2,\infty}\|\Gamma(Z)\|\\
&\leq 4\rho^{T-T_0}\|\Gamma(Z)\|+\left(1+\frac{\rho}{1-\rho}\right)20\|U\|_{2,\infty}\|\Gamma(Z)\|\\
&= 4\rho^{T-T_0}\|\Gamma(Z)\|+\frac{20}{1-\rho}\|U\|_{2,\infty}\|\Gamma(Z)\|
\end{align*}
as required.
\end{proof}

\section{Proof of Lemmas in Section \ref{sec:be_proof}}

Recall Lemma \ref{lem:residuals1}.  
\residualsrand*

\begin{proof}[Proof of Lemma \ref{lem:residuals1}]
Recall the definitions of $R_1$, $R_2$ and $R_3$ via
\begin{align*}
    R_1 :&= e_i\t UU\t\bigg( EM\t + \Gamma(EE\t) \bigg) U\Lambda^{-2} e_j  ;\\
    R_2: &=e_i\t \bigg( \tilde U_D \tilde U_D\t - \tilde U \tilde U\t \bigg)Ue_j; \\
    R_3 :&=  e_i\t \sum_{k\geq 2} S_{MM\t,k}(W) Ue_j  .
\end{align*}
We analyze each in turn.

\ \\ \noindent \textbf{The term $R_1$:} 
We will split this into two terms, $R_{11}$ and $R_{12}$.  From the identity $M\t U = V \Lambda$, the $i,j$ entry of $R_1$ can be written as
\begin{align*}
 \frac{1}{\lambda_j} e_i\t UU\t E V_{\cdot j} + \frac{1}{\lambda_j^2}\sum_{k\neq l} (UU\t)_{ik} U_{lj} \langle E_k, E_l \rangle % &= \frac{1}{\lambda_j} \sum_{k=1}^{n} (UU\t)_{ik} \langle E_k, V_{\cdot j} \rangle + \frac{1}{\lambda_j^2} \sum_{k < l} ( U_{ik}U_{lj} + U_{il}U_{kj}) \langle E_k, E_l \rangle \\
 :&= R_{11} + R_{12},
\end{align*}
where
\begin{align*}
    R_{11} :&= \frac{1}{\lambda_j} \sum_{k=1}^{n} (UU\t)_{ik} \langle E_k, V_{\cdot j} \rangle; \\
    R_{12} :&= \frac{1}{\lambda_j^2}\sum_{k\neq l} (UU\t)_{ik} U_{lj} \langle E_k, E_l \rangle.
\end{align*}
Dividing $R_{11}$ by $\sigma_{ij}$ reveals it is of the form
\begin{align*}
    \frac{1}{\| \Sigma_i^{1/2}V_{\cdot j}\|} \sum_{k=1}^n (UU\t)_{ik} \langle E_k, V_{\cdot j} \rangle.
\end{align*}
To calculate an upper bound, we need to calculate the $\psi_2$ norm squared:
\begin{align*}
   \frac{1}{\| \Sigma_i^{1/2}V_{\cdot j}\|^2}  \sum_{k=1}^{n} (UU\t)_{ik}^2 V_{\cdot j}\t \Sigma_k V_{\cdot j} &\leq  \frac{\sigma^2}{\| \Sigma_i^{1/2}V_{\cdot j}\|^2}  \sum_{k=1}^{n} (UU\t)_{ik}^2 \\
    &\leq \kappa_{\sigma}^2\|U\|_{2,\infty}^2
\end{align*}
which by Hoeffding's inequality shows that this is less than $\tilde C_{R_1} \kappa_{\sigma}\|U\|_{2,\infty}t$ with probability at least $1 - 2 \exp(-ct^2)$.  Hence, we obtain the bound $C_{R_1} \kappa_{\sigma} \|U\|_{2,\infty} \sqrt{\log(n\vee d)}$ with probability at least $1 - 2(n\vee d)^{-4}$.

We now analyze $R_{12}$.  Note that
\begin{align*}
R_{12} := \frac{1}{\lambda_j^2}\sum_{k\neq l} (UU\t)_{ik} U_{lj} \langle E_k, E_l \rangle
\end{align*}
resembles the random variable in the Hanson-Wright inequality (e.g. \citet{vershynin_concentration_2020,chen_hanson-wright_2020}).  By the generalized Hanson-Wright inequality (e.g. Exercise 6.2.7 in \citet{vershynin_high-dimensional_2018}), we have that
\begin{align*}
    \p\bigg\{ \bigg| \sum_{k\neq l} B_{kl} \langle E_k, E_l \rangle \bigg| \geq t \bigg) &\leq 2 \exp\bigg[ -c \min\bigg( \frac{t^2}{\sigma^4 d \|B\|_F^2}, \frac{t}{\sigma^2 \| B\|} \bigg) \bigg]
\end{align*}
where $B_{kl} := (UU\t)_{il} U_{kj}$.  Note that its Frobenius norm can be evaluated via
\begin{align*}
    \| B \|_F^2 &= \sum_{k \neq l} (UU\t)_{il}^2 U_{kj}^2 \\
    &\leq \sum_{l=1}^n(UU\t)_{il}^2 \sum_{k=1}^{n} U_{kj}^2 \\
    &\leq \sum_{l=1}^n(UU\t)_{il}^2 \\
    &\leq \| U\|_{2,\infty}^2.
\end{align*}
Similarly,
\begin{align*}
\| B \| :&= \sup_{\|x\| = 1, \|y\|= 1} \sum_{k=1}^{n} x_k (UU\t)_{ik}  \sum_{l=1}^{n} U_{lj} y_l \\
&\leq  \sup_{\|x\| = 1} \sum_{k=1}^{n} x_k (UU\t)_{ik} \\
&= \| (UU\t)_i \|_2\\
&\leq \|U\|_{2,\infty}.
\end{align*}
Therefore,
\begin{align*}
     \p\bigg\{ \bigg| \sum_{k\neq l} B_{kl} \langle E_k E_l \rangle \bigg| \geq t \bigg) &\leq 2 \exp\bigg[ -c \min\bigg( \frac{t^2}{\sigma^4 d \|U\|_{2,\infty}^2}, \frac{t}{\sigma^2 \|U\|_{2,\infty}} \bigg) \bigg].
\end{align*}
Taking $t = s \sqrt{d}\|U\|_{2,\infty} \sigma^2$ shows that
\begin{align*}
     \p\bigg\{ \bigg| \sum_{k\neq l} B_{kl} \langle E_k E_l \rangle \bigg| \geq s \sqrt{d}\|U\|_{2,\infty} \sigma^2 \bigg) &\leq 2 \exp\bigg( -c \min\bigg( s^2, s\sqrt{d} \bigg) \bigg),
\end{align*}
and hence taking $s = \frac{1}{\sqrt{c}}\sqrt{4 \log(n\vee d)}$ we see that with probability at  least $1 -2 (n\vee d)^{-4}$
\begin{align*} 
\bigg| \frac{1}{\lambda_j^2}\sum_{k\neq l} B_{kl} \langle E_k E_l \rangle \bigg| &\leq C_{R_1} \frac{ \sigma^2}{\lambda_j^2}   \|U\|_{2,\infty} \sqrt{d \log(n\vee d)}.
\end{align*}
Dividing by $\sigma_{ij}$ reveals that with this same probability,
\begin{align*}
    |R_{22}/\sigma_{ij} | &\leq C_{R_1} \kappa_{\sigma} \kappa \frac{\sigma \sqrt{d}}{\lambda_j} \|U\|_{2,\infty} \sqrt{\log(n\vee d)} \\
&\leq C_{R_1} \kappa_{\sigma} \kappa \mu_0 \frac{\sigma \sqrt{rd}}{\lambda_r} \sqrt{\frac{\log(n\vee d)}{n}} \\
&\leq C_{R_1} \kappa_{\sigma} \kappa \mu_0 \frac{1}{\snr}  \sqrt{\frac{\log(n\vee d)}{n}}.
\end{align*}

\ \\ \noindent \textbf{The term $R_2$:} Note that $\tilde U_D $ and $\tilde U$ were already analyzed in Lemma \ref{lem:diag}, which shows that %are eigenvectors of a matrix whose entries only differ in the diagonal matrix whose entries are $2\langle E_i, M_i\rangle$.  In fact, this is straightforwardly bounded already by Lemma \ref{lem:diag}, which shows that
\begin{align*}
    \| \tilde U_D \tilde U_D\t - \tilde U \tilde U\t \| &\leq \frac{C_{R_2} \lambda_1 \sigma \sqrt{ \log(n\vee d)}}{\lambda_r^2} \|U\|_{2,\infty} \\
    &\leq C_{R_2} \|U\|_{2,\infty} \kappa \frac{\sigma }{\lambda_r} \sqrt{ \log(n\vee d)}
\end{align*}
with probability at least $1- 2(n\vee d)^{-4}$.  Multiplying by $\frac{1}{\sigma_{ij}}$ yields the upper bound 
\begin{align*}
    C_{R_2} \kappa_{\sigma}\|U\|_{2,\infty} \kappa^2 \sqrt{\log(n\vee d)} &\leq    C_{R_2} \kappa_{\sigma}\mu_0 \kappa^2 \sqrt{\frac{r \log(n\vee d)}{n}}.
\end{align*}
with probability at least $1 -2 (n\vee d)^{-4}$.

\ \\ \noindent \textbf{The term $R_3$:} 
First, note that
\begin{align*}
\bigg| e_i\t \sum_{k\geq 2} S_{MM\t,k}(W) U e_j \bigg| &\leq \| \sum_{k\geq 2} S_{MM\t,k} (W)U\|_{2,\infty} \\
&\leq \sum_{k\geq 2}\|  S_{MM\t,k} (W)U\|_{2,\infty}.
\end{align*}
Examining the proof of Theorem \ref{thm:2_infty_random} shows that the exact same result holds, only  now we start the summation at $k =2$.  Consequently, using the definition of $\delta$ as in the proof of Theorem \ref{thm:2_infty_random}, by Lemma \ref{lem:2infty_general} we have that with probability $1 - c \log(n\vee d )(n\vee d)^{-5}$ that
\begin{align*}
    \sum_{k=2}^{\infty} \|  S_{MM\t,k} (W)U\|_{2,\infty} &\leq \|U\|_{2,\infty}\sum_{k=2}^{c\log(n\vee d)} C_1 \bigg(\frac{4 C_2 \delta}{\lambda_r^2} \bigg)^{k} + \sum_{k= c\log(n\vee d)}^{\infty}\bigg(\frac{4 \delta}{\lambda_r^2} \bigg)^{k} \\
    &\leq C_{R_3} \|U\|_{2,\infty} \frac{\delta^2}{\lambda_r^4}.
\end{align*}
Hence, with probability at least $1 - (n\vee d)^{-4}$,   
\begin{align*}
    \frac{1}{\sigma_{ij}} |R_3| &\leq C_{R_3} \|U\|_{2,\infty}\frac{ \delta^2}{\lambda_r^4}   \frac{1}{\sigma_{ij}} \\
    &\leq C_{R_3}\|U\|_{2,\infty} \frac{1}{\sigma_{ij} \lambda_r^4} \bigg( \sqrt{r n d} \log (\max (n, d)) \sigma^{2}+ \sqrt{r n \log (n)} \lambda_{1} \sigma\bigg)^2 \\
    &\leq C_{R_3} \|U\|_{2,\infty} \frac{rnd \log^2(d) \sigma^4 + rn\log(n\vee d) \lambda_1^2 \sigma^2 + rn\sqrt{d} \log^{3/2}(n) \sigma^3 }{\sigma_{ij} \lambda_r^4}   \\
    &\leq C_{R_3} \kappa\kappa_{\sigma} \|U\|_{2,\infty} \frac{rnd \log^2(d) \sigma^3 + rn\log(n\vee d) \lambda_1^2 \sigma + rn\sqrt{d} \log^{3/2}(n) \sigma^2 }{ \lambda_r^3} \\
    &\leq C_{R_3} \kappa \kappa_{\sigma}\|U\|_{2,\infty} \bigg( \frac{rnd \log^2(n)\sigma^3+ rn\sqrt{d}\log^{3/2}(n)\sigma^2}{\lambda_r^3} + \frac{rn\log(n\vee d) \kappa^2 \sigma}{\lambda_r} \bigg) \\
    &\leq C_{R_3} \kappa \kappa_{\sigma} \mu_0 \bigg( \frac{r^{3/2}\sqrt{n} d \sigma^3}{\lambda_r^3} \log^2(n\vee d) + \frac{r^{3/2} \sqrt{nd} \sigma^2}{\lambda_r^3} \log^{3/2}(n\vee d) + \frac{r^{3/2} \sqrt{n} \kappa^2 \sigma}{\lambda_r} \log(n\vee d) \bigg) \\
    &\leq C_{R_3} \kappa \kappa_{\sigma} \mu_0 \bigg(  \frac{\log^2(n)}{\snr^3} + \frac{\sqrt{r}}{\lambda_r} \frac{\log^{3/2}(n\vee d)}{\snr^2} + \frac{\kappa^2 r \log(n\vee d)}{\snr} \bigg) \\
    &\leq  C_{R_3} \kappa^3 \kappa_{\sigma} \mu_0 \frac{r\log(n\vee d)}{\snr},
\end{align*}
where we have absorbed extra constants into $C_{R_3}$ since $\snr \geq \kappa \sqrt{\log(n\vee d)}$ by Assumption \ref{assumption:eigengap}. 
Therefore, summing up the probabilities and absorbing the constants, we see that with probability at least $1 - 5(n\vee d)^{-4}$ that 
\begin{align*}
\frac{1}{\sigma_{ij}} \bigg(|R_1| + |R_2| + |R_3| \bigg) &\leq  C_{R_1} \kappa_{\sigma } \kappa \mu_0 \frac{1}{\snr} \sqrt{\frac{\log(n\vee d)}{n}}+  C_{R_2} \kappa_{\sigma}\mu_0 \kappa^2 \sqrt{\frac{r\log(n\vee d)}{n}}+C_{R_3} \kappa^3 \kappa_{\sigma} \mu_0 \frac{r\log(n\vee d)}{\snr} \\
&\leq C_{6} \kappa_{\sigma} \kappa^2 \mu_0 \sqrt{\frac{r\log(n\vee d)}{n}}+C_7 \kappa^3 \kappa_{\sigma} \mu_0 \frac{r\log(n\vee d)}{\snr}
\end{align*}
as required.
\end{proof}

\residualsdet*

\begin{proof}[Proof of Lemma \ref{lem:residuals2}]
Recall the definitions of $R_4$ and $R_5$ via
\begin{align*}
R_4:&=   e_i\t \hat U( \mathcal{O}_* - \hat U\t \tilde U\tilde U\t U )e_j  ;\\
R_5:&=  e_i\t \bigg( \hat U \hat U\t \tilde U\tilde U\t  U - \tilde U \tilde U\t U\bigg) e_j.
\end{align*}
On the event in Theorem \ref{thm:2_infty}, $\| \hat U \|_{2,\infty} \leq \| U\|_{2,\infty} + \inf_{\mathcal{O}_*} \|\hat U - U\mathcal{O}_*\|_{2,\infty} \leq C \|U\|_{2,\infty}$ by Assumption \ref{assumption:eigengap}.  Therefore, the term $R_4$ can be bounded in a similar manner to the proof of Lemma \ref{orthogonal_matrix_lemma} (see Appendix \ref{sec:lems_aux}) via
\begin{align*}
\bigg| e_i\t \hat U( \mathcal{O}_* - \hat U\t \tilde U\tilde U\t U )e_j\bigg| &\leq \| \hat U\|_{2,\infty} \| \mathcal{O}_* - \hat U\t \tilde U\tilde U\t U \|  \\
&\leq C\|U\|_{2,\infty} \bigg( \|\sin(\hat U, \tilde U)\|^2 + \|\sin(\tilde U,U)\|^2\bigg) \\
&\leq C \|U\|_{2,\infty} \frac{\|\Gamma(Z)\|^2}{\lambda_r^4},
\end{align*}
where the final inequality is by the Davis-Kahan Theorem and Lemmas \ref{lem:spectral_norm_concentration} and \ref{lem:deterministic_spectral}.  On the event in Lemma \ref{lem:spectral_norm_concentration}, the numerator can be bounded by
\begin{align*}
    C_{\mathrm{spectral}}^2 \bigg( \sigma^2 ( n + \sqrt{nd}) + \sigma \lambda_1 \sqrt{n} \bigg)^2.
\end{align*}
Consequently, there exists a universal constant $C_{R_4}$ such that
\begin{align*}
    \frac{1}{\sigma_{ij}}| R_4| &\leq \frac{1}{\sigma_{ij}}C_{R_4} \|U\|_{2,\infty} \frac{ \sigma^3 n\sqrt{d}  + \sigma^2 \lambda_1^2 n + \sigma^4 nd}{\lambda_r^4} \\
    &=C_{R_4} \frac{\lambda_j}{\|\Sigma_i^{1/2}V_{\cdot j}\|} \|U\|_{2,\infty} \frac{ \sigma^3 n\sqrt{d}  + \sigma^2 \lambda_1^2 n + \sigma^4 nd}{\lambda_r^4} \\
    &\leq C_{R_4} \kappa\kappa_{\sigma} \|U\|_{2,\infty} \frac{\sigma^2 n\sqrt{d} + \sigma \lambda_1^2 n + \sigma^3 nd}{\lambda_r^3} \\
    &\leq  C_{R_4} \kappa\kappa_{\sigma} \mu_0 \frac{\sigma^2 \sqrt{rnd} + \sigma \lambda_1^2 \sqrt{r n} + \sigma^3 \sqrt{rn} d}{\lambda_r^3} \\
    &\leq C_{R_4}  \kappa\kappa_{\sigma} \mu_0 \bigg( \frac{1}{\lambda_r \snr^2} + \frac{\kappa^2}{\snr} + \frac{1}{\snr^3} \bigg) \\
    &\leq C_8 \kappa^3\kappa_{\sigma} \mu_0  \frac{1}{\snr}.
\end{align*}
The term $R_5$ satisfies
\begin{align*}
\frac{1}{\sigma_{ij}} \bigg| e_i\t \bigg( \hat U \hat U\t \tilde U\tilde U\t  U - \tilde U \tilde U\t U\bigg) e_j \bigg| &\leq\frac{1}{\sigma_{ij}} \| \hat U \hat U\t   \tilde U\tilde U\t U - \tilde U \tilde U\t U \|_{2,\infty} \\
&\leq \frac{1}{\sigma_{ij}} \|\hat U \hat U\t \tilde U - \tilde U\|_{2,\infty}.
\end{align*}
The definition of $\tilde H$ in the proof of Theorem \ref{thm:2_infty_deterministic} shows that $\tilde H\t = \hat U\t \tilde U$.  Define $$\mathcal{O}_1:=\argmin_{\mathcal{O} \in \mathbb{O}(r)} \| \hat U - \tilde U \mathcal{O} \|_F.$$
Then
\begin{align*}
    \frac{1}{\sigma_{ij}} \|\hat U \hat U\t \tilde U - \tilde U\|_{2,\infty} &=  \frac{1}{\sigma_{ij}} \| \hat U \tilde H\t - \tilde U \|_{2,\infty} \\
   % &\leq \frac{1}{\sigma_{ij}} \| \hat U \tilde H\t - \tilde U \|_{2,\infty} \\
    &\leq \frac{1}{\sigma_{ij}} \bigg( \| \hat U \tilde H\t - \hat U \mathcal{O}_1\t \|_{2,\infty}  +\| \hat U \mathcal{O}_1\t - \tilde U \|_{2,\infty} \bigg) \\
    &=\frac{1}{\sigma_{ij}} \bigg( \| \hat U \tilde H\t - \hat U \mathcal{O}_1\t \|_{2,\infty}  +\| \hat U  - \tilde U \mathcal{O}_1  \|_{2,\infty} \bigg)\\
    &\leq \frac{1}{\sigma_{ij}} \bigg( \| \hat U (\tilde H\t - \mathcal{O}_1\t) \|_{2,\infty}  +\| \hat U  - \tilde U \tilde H   \|_{2,\infty} + \| \tilde U (\mathcal{O}_1 - \tilde H) \|_{2,\infty} \bigg) \\
    &\leq \frac{1}{\sigma_{ij}} \bigg( \|\hat U \|_{2,\infty} \| \tilde H - \mathcal{O}_1 \| + \|\tilde U \|_{2,\infty} \| \mathcal{O}_1 - \tilde H \| + \|\hat U - \tilde U \tilde H \|_{2,\infty} \bigg) \\
    &\leq \frac{1}{\sigma_{ij}} \bigg( 2 C \|U\|_{2,\infty} \| \sin(\hat U, \tilde U ) \|_2^2 + \|\hat U - \tilde U \tilde H \|_{2,\infty} \bigg) \\
    &\leq \frac{1}{\sigma_{ij}} \bigg( 2 C \|U\|_{2,\infty}\frac{\|\Gamma(Z)\|^2}{\lambda_r^4} + \|\hat U - \tilde U \tilde H \|_{2,\infty} \bigg),
    \end{align*}
    where the final line follows from the Davis-Kahan Theorem.  By Theorem \ref{thm:2_infty_deterministic} we have that
    \begin{align}
         \|\hat U - \tilde U \tilde H \|_{2,\infty} &\leq C_D \kappa^2 \|U\|_{2,\infty}^2 \frac{ \|\Gamma(Z)\|}{\lambda_r^2}. \label{lems_dets_ub}
    \end{align}
    We have already shown that
    \begin{align*}
        \frac{1}{\sigma_{ij}} \| U \|_{2,\infty} \frac{\|\Gamma(Z)\|^2}{\lambda_r^4} &\leq C_8 \frac{\kappa^3 \kappa_{\sigma} \mu_0}{\snr},
    \end{align*}
    which matches the bound for $R_4$, so by increasing the constant $C_8$, we need only bound the term in Equation     \eqref{lems_dets_ub}.
    We have that
    \begin{align*}\frac{1}{\sigma_{ij}} C_D \kappa^2 \|U\|_{2,\infty}^2 \frac{ \|\Gamma(Z)\|}{\lambda_r^2}  %&\leq   \frac{C_D \kappa^2}{\sigma_{ij}} \|U\|_{2,\infty}^2 \frac{\|\Gamma(Z)\|}{\lambda_r^2} \\
&\leq \frac{C_9}{\sigma_{ij}} \kappa^2 \|U\|_{2,\infty}^2 \frac{\sigma^2 (n + \sqrt{nd}) + \sigma \lambda_1 \sqrt{n}}{\lambda_r^2} \\
&\leq C_9 \kappa_{\sigma}\kappa^3 \|U\|_{2,\infty}^2 \frac{ \sigma (n + \sqrt{nd}) + \lambda_1 \sqrt{n}}{\lambda_r} \\
&\leq C_9 \kappa_{\sigma} \kappa^3 \mu_0^2 \frac{ r\sigma (n + \sqrt{nd}) + \lambda_1 r\sqrt{n}}{n\lambda_r} \\
&\leq C_9 \kappa_{\sigma}\kappa^3 \mu_0^2 \bigg( \frac{r\sigma}{\lambda_r} + \frac{r \sigma \sqrt{d}}{\sqrt{n} \lambda_r} + \kappa \frac{r}{\sqrt{n}}\bigg) \\
&\leq C_9  \kappa^4 \kappa_{\sigma} \mu_0^2 \frac{r}{\sqrt{n}}
\end{align*}
which is the desired upper bound.
\end{proof}

\finalres*

\begin{proof}[Proof of Lemma \ref{lem:final_res}]
First, conditionally on $E_i$ the sum is a sum of independent random variables each with $\psi_2$ norm bounded by 
\begin{align*}
    \bigg\|\langle E_i, E_k \rangle U_{kj} \lambda_j^{-2} \bigg\|_{\psi_2} &\leq  \max_{k} \| \langle E_i, E_k \rangle \|_{\psi_2} |U_{kj}|\lambda_j^{-2} \\
    &\leq  C \|E_i\| \sigma    \|U\|_{2,\infty}\lambda_j^{-2},
\end{align*}
where $C$ is a universal constant. Hence, for any $t \geq 0$, we have that 
\begin{align*}
    \bigg| \sum_{k\neq i} \langle E_i, E_k \rangle U_{kj} \lambda_j^{-2} \bigg| &\leq C \sigma \sqrt{n} \| E_i\| \|U\|_{2,\infty} \lambda_j^{-2} t 
\end{align*}
with probability at least $1 - 2\exp(-c t^2)$.  Furthermore, for some other universal constant $C$, $\|E_i\| \leq C\sigma_i \sqrt{d} s$ with probability at least $1 - 2 \exp(-cs^2)$ (uniformly over $i$).  Hence, 
\begin{align*}
 \bigg| \sum_{k\neq i} \langle E_i, E_k \rangle U_{kj} \lambda_j^{-2} \bigg|  &\leq C_{10} \sigma \sqrt{nd} \lambda_j^{-2} \|U\|_{2,\infty} \sigma_i t
\end{align*}
with probability at least $1 - 4\exp(-c t)$.  Recall $\sigma^2_{ij} := \|\Sigma_i^{1/2} V_{\cdot j}\|^2 \lambda_j^{-2}$. Then
\begin{align*}
\frac{1}{\sigma_{ij}} \bigg| \sum_{k\neq i} \langle E_i, E_k \rangle U_{kj} \lambda_j^{-2} \bigg| &\leq\frac{C_{10}}{\sigma_{ij}} \sigma \sqrt{nd} \lambda_j^{-2} \|U\|_{2,\infty} \sigma_i t \\
&\leq C_{10}\frac{\sigma \sqrt{d}}{\lambda_j} \frac{\sigma_i}{\|\Sigma_i^{1/2}V_{\cdot j}\|} \sqrt{n}\|U\|_{2,\infty} t \\
&\leq C_{10} \kappa_{\sigma}\frac{\sigma \sqrt{d}}{\lambda_j}\sqrt{n}\|U\|_{2,\infty} t
\end{align*}
with probability at least $1 - 4e^{-ct}$, since $\sigma_i/\|\Sigma_i^{1/2}V_{\cdot j}\| \leq \kappa_{\sigma}$ remains bounded away from zero and infinity.    Taking $t = C (4 /c) \log(n\vee d)$ and absorbing the constants shows that with probability at least $1 - 4(n\vee d)^{-4}$, uniformly over $i$ and $j$ that
\begin{align*}
    \frac{1}{\sigma_{ij}} \bigg| \sum_{k\neq i} \langle E_i, E_k \rangle U_{kj} \lambda_j^{-2} \bigg| &\leq C_{10} \kappa_{\sigma}\frac{\sigma \sqrt{d}}{\lambda_j}\sqrt{n}\|U\|_{2,\infty} \log(n\vee d) \\
    &\leq C_{10} \mu_0 \kappa_{\sigma} \frac{\log(n\vee d)}{\snr}
\end{align*}
as required.  
\end{proof}

\section{Proof of Auxiliary Lemmas}
\label{sec:lems_aux}
%\subsection{Proof of Lemma \ref{orthogonal_matrix_lemma}}
First, recall Lemma \ref{orthogonal_matrix_lemma}.
\orthogonalmatrixlemma*
\begin{proof}[Proof of Lemma \ref{orthogonal_matrix_lemma}]
We have that for any orthogonal matrices $\mathcal{O}$ and $\tilde{\mathcal{O}}$ that
\begin{align*}
   \| UU\t \utilde\utilde\t \uhat  - U\mathcal{O} \|_{2,\infty} &\leq \|U\|_{2,\infty} \|U\t \utilde \utilde \t \uhat - \mathcal{O}\| \\
   &\leq \|U\|_{2,\infty}\bigg( \|U\t \utilde \utilde \t \uhat - \tilde{\mathcal{O}} \utilde\t \uhat\| + \|\mathcal{O} - \tilde{\mathcal{O}}\utilde\t \uhat\|\bigg)\\
   &\leq \|U\|_{2,\infty}\bigg( \|U\t \utilde  - \tilde{\mathcal{O}} \| + \|\mathcal{O} - \tilde{\mathcal{O}} \utilde\t \uhat\|\bigg).
\end{align*}
Let $\mathcal{O}^{(1)}$ and $\mathcal{O}^{(2)}$ be the orthogonal matrices satisfying
\begin{align*}
   \mathcal{O}^{(1)} :&= \argmin \|U - \utilde \mathcal{O}\|_F; \\
    \mathcal{O}^{(2)} :&= \argmin \|\utilde - \uhat\mathcal{O}\|_F.
\end{align*}
Define $\mathcal{O}_* := (\mathcal{O}^{(1)} \mathcal{O}^{(2)})\t.$  
Letting $\tilde{\mathcal{O}}= \mathcal{O}^{(1)}$ and $\mathcal{O} = \mathcal{O}_*\t$ implies that this is less than or equal to
\begin{align*}
    \|U\|_{2,\infty} \bigg( \|\sin\Theta(U, \utilde)\|^2 + \|\sin\Theta(\utilde, \uhat)\|^2\bigg).  
\end{align*}
By the Davis-Kahan Theorem, under Assumptions \ref{assumption:eigengap} and \ref{assumption:incoherence} and under the event in Lemma \ref{lem:spectral_norm_concentration}  by Lemmas \ref{lem:deterministic_spectral} and \ref{lem:spectral_norm_concentration} we have that for some constant $C$,
\begin{align*}
      \|U\|_{2,\infty} \bigg( \|\sin\Theta(U, \utilde)\|^2 + \|\sin\Theta(\utilde, \uhat)\|^2\bigg) &\leq C \|U\|_{2,\infty} \bigg( \frac{\|\Gamma(Z)\|^2}{\lambda_r^4} + \|U\|_{2,\infty}^4  \frac{\|\Gamma(Z)\|^2}{\lambda_r^4} \bigg) \\
      &\leq C \|U\|_{2,\infty} \frac{\|\Gamma(Z)\|^2}{\lambda_r^4}
\end{align*}
since $\|\Gamma(Z)\|/\lambda_r^2 \leq C$ under these assumptions and the event in Lemma \ref{lem:spectral_norm_concentration}.
\end{proof}

\begin{lemma}\label{lemma7_spectral}
If $M$ has rank $r$, and $\hat{U}$ is the projector onto the top $r$ left singular vectors of $M+E$, and if $\lambda_r \geq 2 \| E \|$, then 
$$\|(I-P_{\hat{U}})M\|\leq 2\|E\|.$$
\end{lemma}

\begin{proof}
We have that
\begin{align*}
    \| (I - P_{\tilde U} )M \| &\leq \| (I - P_{\tilde U}) (M + E)\| +  \| (I - P_{\tilde U}) E \|\\
    &\leq \lambda_{r+1}(M + E) + \| E \| \\
    &\leq 2 \| E \|
\end{align*}
by Weyl's inequality.
\end{proof}

\bibliography{CMDS.bib}

\end{document}